\documentclass[10pt]{amsart}

\title[]{The transfer map of free loop spaces}
\author{John A. Lind}
\author{Cary Malkiewich}

\usepackage{amsmath}
\usepackage{amssymb}
\usepackage{amsthm}
\usepackage{amsrefs}
\usepackage{enumerate} 
\usepackage{mathrsfs}
\usepackage{stmaryrd}
\usepackage[all]{xy}
\usepackage[mathcal]{eucal}
\usepackage{verbatim}  
\usepackage{leftidx}
\usepackage{graphicx}
\usepackage{hyperref}
\usepackage{float} 
\usepackage{tikz}

\usepackage{xcolor}
\usepackage{transparent}

 \newdir{ >}{{}*!/-5pt/@{>}}  

\numberwithin{equation}{section}

\newtheorem{theorem}[equation]{Theorem}

\newtheorem{proposition}[equation]{Proposition}

\newtheorem{lemma}[equation]{Lemma}
\newtheorem{corollary}[equation]{Corollary}

\theoremstyle{definition}
\newtheorem{definition}[equation]{Definition}
\newtheorem{remark}[equation]{Remark}
\newtheorem{construction}[equation]{Construction}

\newtheorem{example}[equation]{Example}

\newcommand{\bL}{\mathbf{L}}

\newcommand{\bbN}{\mathbb{N}}\newcommand{\bbS}{\mathbb{S}}

\newcommand{\cB}{\mathcal{B}}\newcommand{\cC}{\mathcal{C}}\newcommand{\cD}{\mathcal{D}}\newcommand{\cE}{\mathcal{E}}\newcommand{\cM}{\mathcal{M}}\newcommand{\cN}{\mathcal{N}}\newcommand{\cP}{\mathcal{P}}\newcommand{\cS}{\mathcal{S}}\newcommand{\cT}{\mathcal{T}}

\DeclareMathOperator{\id}{id}
\DeclareMathOperator{\eval}{eval}

\DeclareMathOperator{\tr}{tr}

\DeclareMathOperator{\ho}{ho}
\DeclareMathOperator{\ob}{ob}

\DeclareMathOperator{\op}{op}

\DeclareMathOperator{\Map}{Map}

\DeclareMathOperator{\Mod}{Mod}

\renewcommand{\phi}{\varphi}

\providecommand{\sma}{\wedge}
\providecommand{\osma}{\,\overline{\wedge}\,}

\providecommand{\arr}{\longrightarrow}
\providecommand{\oarr}[1]{\overset{ #1 }{\longrightarrow}}
\providecommand{\xarr}[1]{\xrightarrow{ #1 }}

\renewcommand{\_}[1]{\underline{ #1 }}
\providecommand{\oline}[1]{\overline{ #1 }}
\newcommand{\lan}[1]{\langle {#1} \rangle}

\providecommand{\THH}{\mathrm{THH}}
\newcommand{\cyc}{\textup{cyc}}
\newcommand{\gp}{\textup{gp}}

\newcommand{\Z}{\mathbf{Z}} 
\newcommand{\R}{\mathbf{R}} 
\newcommand{\ra}{\arr}
\newcommand{\Emb}{\textup{Emb}}
\newcommand{\simar}{\overset{\sim}{\ra}}
\newcommand{\hofib}{\textup{hofib}\,}

\providecommand{\Sp}[2]{{}_{#1}\cS p_{#2}}
\providecommand{\Ex}{\cE\mathrm{x}}
\providecommand{\Top}{\cT\mathrm{op}}
\providecommand{\Bimod}{\cB\mathrm{imod}}
\providecommand{\coev}{\mathrm{coev}}
\providecommand{\eval}{\mathrm{eval}}
\providecommand{\ev}{\mathrm{ev}}
\providecommand{\tr}{\mathrm{tr}}
\providecommand{\Perf}{\cP\mathrm{erf}}
\providecommand{\Diff}{\mathrm{Diff}}

\providecommand{\B}{\mathbf{B}}
\providecommand{\E}{\mathbf{E}}

\makeatletter
\newsavebox{\@brx}
\newcommand{\llangle}[1][]{\savebox{\@brx}{\(\m@th{#1\langle}\)}%
  \mathopen{\copy\@brx\kern-0.5\wd\@brx\usebox{\@brx}}}
\newcommand{\rrangle}[1][]{\savebox{\@brx}{\(\m@th{#1\rangle}\)}%
  \mathclose{\copy\@brx\kern-0.5\wd\@brx\usebox{\@brx}}}
\makeatother

\newcommand{\llan}[1]{\llangle {#1} \rrangle} 

\usetikzlibrary{scopes,calc,arrows,patterns,cd}
\tikzset{ed/.style={auto,inner sep=0pt,font=\scriptsize}} 
\tikzset{>=stealth'}

\colorlet{myblue}{blue!40!white}
\colorlet{mydarkblue}{blue!50!white}
\colorlet{myred}{red!35!white}
\colorlet{mygreen}{green!30!white}
\colorlet{myyellow}{yellow!10!white}
\colorlet{myorange}{orange!60!white}
\tikzset{bluefill/.style={fill=myblue}}
\tikzset{darkbluefill/.style={fill=mydarkblue}}
\tikzset{redfill/.style={fill=myred}}
\tikzset{greenfill/.style={fill=mygreen}}
\tikzset{yellowfill/.style={fill=myyellow}}
\tikzset{orangefill/.style={fill=myorange}}

\def\bgcylinder#1#2#3#4#5#6{
  \def\cylempty{}\def\cylfrontcolor{#5}\def\cylbackcolor{#6}
  \begin{pgfonlayer}{background}
    \ifx\cylfrontcolor\cylempty\draw\else\fill[fill=my#5]\fi
    (#1) coordinate (dl)
    -- ++(0,#2) node[coordinate] (ul) {} 
    arc (-180:0:#3 and #4) coordinate (ur)
    -- ++(0,-#2) node[coordinate] (dr) {}
    arc (0:-180:#3 and #4);
    \ifx\cylbackcolor\cylempty\draw\else\fill[fill=my#6!80!black]\fi
    ($(ul)!.5!(ur)$) node[coordinate] (top) {} ellipse (#3 and #4);
    \path (dl) arc (-180:-90:#3 and #4) node[coordinate] (bot) {};
    \clip (dl) -- (ul) arc (-180:0:#3 and #4) -- (dr) arc (0:-180:#3 and #4);
  \end{pgfonlayer}
  \clip (dl) -- (ul) arc (-180:0:#3 and #4) -- (dr) arc (0:-180:#3 and #4);
  \path (ul) ++(-0.1,0.1) coordinate (ul');
  \path (ur) ++(0.1,0.1) coordinate (ur');
  \path (dl) ++(-0.1,-0.1) coordinate (dl');
  \path (dr) ++(0.1,-0.1) coordinate (dr');
  \path (top) ++(0,0.1) coordinate (top');
  \path (bot) ++(-0,-0.1) coordinate (bot');
}

\pgfdeclarelayer{background}
\pgfdeclarelayer{foreground}
\pgfsetlayers{background,main,foreground}

\begin{document}

\maketitle

\begin{abstract}
For any perfect fibration $E \arr B$, there is a ``free loop transfer map'' $LB_+ \arr LE_+$, defined using topological Hochschild homology. We prove that this transfer is compatible with the Becker-Gottlieb transfer, allowing us to extend a result of Dorabia\l{}a and Johnson on the transfer map in Waldhausen's $A$-theory.
In the case where $E \arr B$ is a smooth fiber bundle, we also give a concrete geometric model for the free loop transfer in terms of Pontryagin-Thom collapse maps.  We recover the previously known computations of the free loop transfer due to Schlichtkrull, and make a few new computations as well.
\end{abstract}

\setcounter{tocdepth}{1}
\tableofcontents

\section{Introduction}


Suppose that $f \colon E \arr B$ is a fibration of topological spaces, and that each fiber of $f$ is a retract up to homotopy of a finite CW complex. We call such a fibration a \emph{perfect fibration}. We study a stable, functorial, wrong-way map between the free loop spaces of $E$ and $B$
\begin{equation}\label{eq:transfer_intro}
\xymatrix{ \Sigma^\infty_+ LB \ar[r]^-{\tau_{\THH}} & \Sigma^\infty_+ LE }
\end{equation}
that is built from $f$ using topological Hochschild homology (THH). This map is linked by the topological Dennis trace \cite{BHM} to the transfer on Waldhausen's algebraic $K$-theory of spaces
\[ \xymatrix{ A(B) \ar[r]^-{\tau_A} & A(E) } \]
that plays a central role in the topological Riemann Roch theorem \citelist{\cite{DWW} \cite{williams}}.

We study two questions about the free loop transfer $\tau_\THH$. First, how does it relate to the classical transfer map 
\begin{equation*}
\xymatrix{ \Sigma^\infty_+ B \ar[r]^-{\tau} & \Sigma^\infty_+ E }
\end{equation*}
constructed by Becker and Gottlieb \cite{BG76}?
Second, how may we describe the free loop transfer geometrically, and compute its effect on cohomology?

Our answer to the first question is as follows. We let $c: X \arr LX$ denote inclusion of constant loops, and $e_0: LX \arr X$ denote evaluation of a free loop at the basepoint. 
\begin{theorem}\label{thm:intro_becker_gottlieb_compatibility}
For any perfect fibration $f$, the composite
\[ \xymatrix{
\Sigma^{\infty}_+ B \ar[r]^-{c} & \Sigma^{\infty}_+ LB \ar[r]^-{\tau_{\THH}} & \Sigma^{\infty}_+ LE \ar[r]^-{e_0} & \Sigma^{\infty}_+ E
} \]
is naturally homotopic to the Becker-Gottlieb transfer $\tau$.
\end{theorem}

As we explain in \S\ref{subsec:proof_cor}, this theorem implies a similar statement regarding $A$-theory.  Let $i$ and $p$ denote the maps giving Waldhausen's splitting of $\Sigma^\infty_+ X$ off of $A(X)$.

\begin{corollary}\label{cor:intro_a_theory}
For any perfect fibration $f$, the composite
\[ \xymatrix{
\Sigma^{\infty}_+ B \ar[r]^-{i} & A(B) \ar[r]^-{\tau_{A}} & A(E) \ar[r]^-{p} & \Sigma^{\infty}_+ E
} \]
is naturally homotopic to the Becker-Gottlieb transfer $\tau$.
\end{corollary}

\noindent The corollary resolves a question about the $A$-theory transfer that has remained open in the general case of a fibration with finitely dominated fibers. Dorabia\l{}a and Johnson \cite{DJ} proved this statement for a smaller class of fibrations, using the Becker-Schultz axiomatization of the Becker-Gottlieb transfer. After projecting the space $E$ to a point, our technique also recovers a result of C. Douglas \cite{D}.

For the second question, we provide a geometric model for the free loop transfer $\tau_{\THH}$ when the fibration $f$ is a smooth fiber bundle with closed manifold fiber $M$. We construct the intermediate space $P$ of paths in $E$ whose endpoints lie in the same fiber of $f$:
\begin{figure}[H]\label{fig:intro_PplusEBLB}
\centering
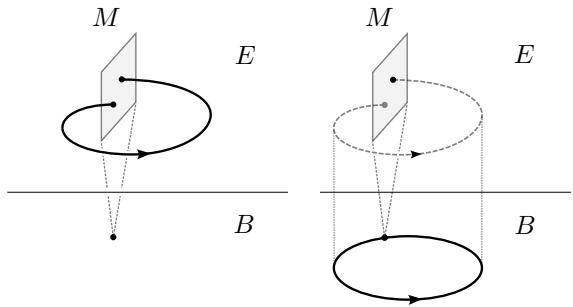
\caption{Generic points in two homotopy equivalent models of $P$.}
\label{fig:two_models_P}
\end{figure}\vspace{-2em}

As indicated in Figure \ref{fig:two_models_P}, the space $P$ is homotopy equivalent to the total space of the bundle $E \times_{B} LB \arr LB$ with fiber $M$. In addition, $P$ contains $LE$ as a subspace, with tubular neighborhood given by a pullback of the vertical tangent bundle $TM$ of the fiber bundle $f$. The free loop transfer of $f$ then decomposes into two different dimension-shifting transfers whose homological behavior is well understood:
\begin{theorem}\label{thm:intro_geometric_THH_transfer} If $f$ is a smooth fiber bundle with compact fibers, the free loop transfer for $f$ is the composite
\[ \xymatrix{ \Sigma^\infty_+ LB \ar[r] & P^{-TM} \ar[r] & \Sigma^\infty_+ LE } \]
of the Pontryagin-Thom umkehr map associated to the bundle $E \times_{B} LB \arr LB$ and the desuspension by $TM$ of the collapse of $P$ onto a tubular neighborhood of $LE$.
\end{theorem}
\noindent Theorem \ref{thm:intro_geometric_THH_transfer} lets us easily characterize the free loop transfer for covering spaces.
\begin{corollary}\label{cor:intro_covering_spaces}
When $f: E \ra B$ is a finite-sheeted covering space, the free loop transfer is the Becker-Gottlieb transfer for the covering space $Lf: LE \ra LB$.
\end{corollary}
\noindent In the case of a covering map $BK \arr BG$, where $G$ is a discrete group and $K < G$ is a subgroup of finite index, this result is analogous to an observation of Bentzen-Madsen about ordinary Hochschild homology \cite{BeMa}*{Prop. 1.4} and recovers Schlichtkrull's computations of the THH transfer \cite{schlichtkrull}. 

Theorem \ref{thm:intro_geometric_THH_transfer} is also useful for cohomology computations:
\begin{proposition}\label{prop:calculation_bs1}
The free loop transfer for the fibration $BS^1 \ra BS^3$ associated to the inclusion of topological groups $S^1 \arr S^3$ induces on cohomology
\[ H^q(LBS^1) \ra H^q(LBS^3) \]
a map of degree $2$ when $q \equiv 0$ mod 4, and an isomorphism when $q \equiv 3$ mod 4.
\end{proposition}
\begin{proposition}\label{prop:calculation_hopf}
The free loop transfer for the Hopf fibration $S^3 \ra S^2$ is nonzero on integral cohomology.
\end{proposition}
\noindent The first of these also follows from a convenient relation between $\tau_\THH$ and the Euler characteristic, which we establish in \S\ref{sec:computations}:
\begin{proposition}\label{prop:intro_euler}
If $f$ is a fibration with finite CW fiber $F$, and $B$ is simply-connected, the composite map
\[ \xymatrix{ H^*(LB) \ar[r]^-{Lf^*} & H^*(LE) \ar[r]^-{\tau_\THH^*} & H^*(LB) } \]
is multiplication by $\chi(F)$.
\end{proposition}

\noindent Finally, our geometric model allows us to improve Theorem \ref{thm:intro_becker_gottlieb_compatibility} in the special case of a smooth fiber bundle.
\begin{corollary}\label{cor:intro_becker_gottlieb_smooth_compatibility}
If $f$ is a smooth fiber bundle with closed manifold fibers, the following square commutes up to homotopy:
\[
\xymatrix @C=4em{
\Sigma^{\infty}_{+} B \ar[r]^-{\tau} \ar[d]_{c} & \Sigma^{\infty}_{+} E \ar[d]^{c} \\
\Sigma^{\infty}_{+} LB \ar[r]^-{\tau_{\THH}} & \Sigma^{\infty}_{+} LE
}
\]
\end{corollary}
\noindent This is a new proof of an old result, as it follows from the topological Riemann-Roch theorem \cite[2.7]{williams}. It should be noted that the corresponding square for the $A$-theory transfer fails to commute in general, and that the corresponding square with evaluation maps $e_0$ fails to commute, even for covering spaces.

We now provide an overview of our definitions and proofs. To define the free loop transfer $\tau_\THH$, we observe that the homomorphism of ring spectra $\Sigma^\infty_+ \Omega E \arr \Sigma^\infty_+ \Omega B$ makes $\Sigma^\infty_+ \Omega B$ into a perfect module over $\Sigma^\infty_+ \Omega E$. This allows us to restrict from the category of perfect $\Omega B$-modules to perfect $\Omega E$-modules, inducing $\tau_{\THH}$ on the topological Hochschild homology of these categories.

To prove our main results, we must recognize that this definition of $\tau_\THH$ coincides with another map of free loop spaces, defined by Ponto and Shulman in the study of parametrized fixed-point invariants \cite{PS14}. They regard the graph $(f, \id) \colon E \arr B \times E$ as a dualizable 1-cell in the bicategory $\Ex$ of parametrized spectra over varying base spaces defined by May and Sigurdsson \cite{MS}. One may take the trace of the identity map of this 1-cell by equipping the bicategory with a shadow \citelist{\cite{PoThesis}\cite{PS13}}.  The result is called the refined Reidemeister trace of the identity of $E$, a map of free loop spaces as in (\ref{eq:transfer_intro}) above, and a generalization of an important classical fixed point invariant to the parametrized setting. The basic result which underlies all of the work in this paper is the following:

\begin{theorem}\label{thm:reid_trace_is_THH_transfer}
The THH transfer $\tau_{\THH}$ for a perfect fibration $E \to B$ is naturally homotopic to the refined Reidemeister trace of the identity map of $E$ over $B$.
\end{theorem}

The proof of this theorem, in turn, has two parts. First, we recognize that any map $\THH(A) \arr \THH(R)$ induced by smashing with a perfect $(A,R)$ bimodule $M$ must coincide with the trace of the 1-cell $M$ in the bicategory $\Bimod_S$ of ring and bimodule spectra. In essence, we do this by generalizing the Dennis-Waldhausen-Morita argument \cite{blumberg_mandell} from Morita equivalences to ``Morita adjunctions,'' which correspond precisely to dualities in $\Bimod_S$. The resulting trace has an effect on the cyclic bar constructions that can be captured pictorially as follows.

\begin{figure}[H]
\def\svgwidth{\linewidth}
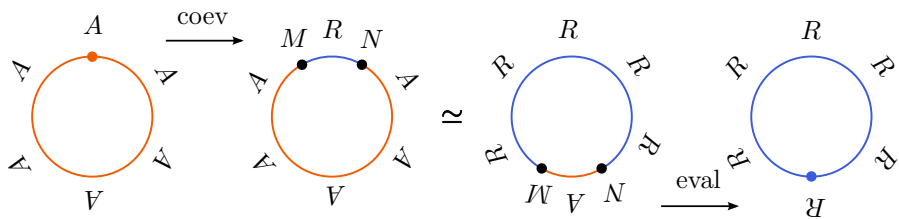
\caption{The bicategorical trace of $\id_M$. Here $N = F_R(M,R)$.}
\label{fig:bicat_trace}
\end{figure}\vspace{-2em}

The second step is to compare the trace of the bimodule $\Sigma^\infty_+ \Omega B$ over the ring spectra $\Sigma^\infty_+ \Omega B$ and $\Sigma^\infty_+ \Omega E$ in the bicategory $\Bimod_{S}$ to the refined Reidemeister trace of the space $E$ over $E \times B$ in the May-Sigurdsson bicategory $\Ex$. This follows from the fact that the fiber functor 
\begin{equation}\label{eq:fiber_functor}
(-)_a \colon \Sp{}{A} \arr \Mod_{\Sigma^{\infty}_{+} \Omega A}
\end{equation}
from parametrized spectra over $A$ to $\Sigma^{\infty}_{+} \Omega A$-modules induces an equivalence of homotopy categories that respects the base change functors in the variable $A$. In the language of \cite{PS12}, it is an equivalence of indexed symmetric monoidal categories. This follows from a corresponding statement at the level of $\infty$-categories which was proved in \cite{ando_blumberg_gepner}*{App B}. In the companion paper \cite{LiMa1} we will present a self-contained proof, using the stronger statement that the fiber functor \eqref{eq:fiber_functor} is a Quillen equivalence of model categories.

Our proof of Theorem \ref{thm:intro_becker_gottlieb_compatibility} relies on a new description of the Becker-Gottlieb transfer, using the language of duality in the bicategory $\Ex_{B}$ of parametrized spectra over spaces over $B$. A key point is the previously unobserved fact that the Becker-Gottlieb pretransfer $\tau_{B} \colon S_B \arr f_{!} S_E$ may be expressed in terms of the coevaluation and evaluation maps for certain dual pairs $(\_{S}_{f}, D \_{S}_{f})$ and $({}_{f}\_{S}, \_{S}_{f})$ in $\Ex_B$ (see Proposition \ref{prop:pretransfer_via_duality}). From here, the rest of the proof is pure bicategory theory.

An outline of the paper follows. In \S\ref{sec:spectral_cats} we review the theory of spectral categories and define $\tau_\THH$. In \S\ref{sec:parametrized_spectra}--\ref{sec:bicats} we review parametrized spectra, and duality, shadows, and traces in a bicategory. In \S\ref{sec:reidemeister} we prove Theorem \ref{thm:reid_trace_is_THH_transfer}, that the free loop transfer agrees with the Reidemeister trace. In \S\ref{sec:fiberwise_CW_duality}--\ref{sec:comparison} we provide our new description of the Becker-Gottlieb transfer and prove Theorem \ref{thm:intro_becker_gottlieb_compatibility}. In \S\ref{sec:geometric} we establish the geometric model of $\tau_\THH$ for smooth fiber bundles and a useful result about the module structure. In \S\ref{sec:computations} we deduce our applications and perform computations.

\subsection*{Conventions}

Most of our work takes place in homotopy categories, so our conventions are not so important. We assume that all the base spaces of our fibrations have the homotopy type of CW complexes. We use EKMM $S$-modules for our work on spectral categories, and orthogonal spectra for our geometric model of $\tau_\THH$.

\subsection*{Acknowledgements}

The authors thank Jon Campbell for asking how the Becker-Gottlieb transfer is related to the THH transfer and Kate Ponto for realizing that the THH transfer sounded suspiciously like the Reidemeister trace. The first author thanks Anssi Lahtinen for helpful conversations on the nature of the ``evil-eye'' product $\odot$.  The second author thanks Andrew Blumberg and Rune Haugseng for ideas and inspiration, and Randy McCarthy for proposing, and losing, a five-dollar bet that this transfer would easily generalize from $LB$ to $\Map(K,B)$ for any finite complex $K$. The joke is due to Bert Guillou. The first author was partially supported by the DFG through SFB1085, and the second author was partially supported by an AMS Simons Travel Grant.

\section{Spectral categories and perfect $R$-modules}\label{sec:spectral_cats}

In this section we recall the theory of spectral categories, and describe the category of perfect modules over a ring spectrum $R$. Then we define the free loop transfer $\tau_\THH$.

\subsection{Review of the $\THH$ of spectral categories}

We first recall how to define the $\THH$ of a spectral category $\cC$, and how to prove when two categories have the same $\THH$. Most of this material is from \cite{blumberg_mandell}. 

Let $\Sp{}{}$ denote any of the symmetric monoidal categories of symmetric spectra, orthogonal spectra \citelist{\cite{mmss} \cite{mandell_may}}, or $S$-modules \cite{ekmm}.  Suppose $\cC$ is a small category enriched in $\Sp{}{}$, and let $\cM$ be a $(\cC, \cC)$-bimodule, i.e. a spectrally enriched functor from $\cC \sma \cC^{\op}$ to the category $\Sp{}{}$ of spectra. 
\begin{definition}
The \emph{cyclic nerve} of $\cC$ with coefficients in $\cM$ is the simplicial spectrum with $n$-simplices
\[ N^\cyc_n (\cC; \cM) = \bigvee_{c_0,\ldots,c_n \in \ob C} \cC(c_0,c_1) \sma \cC(c_1,c_2) \sma \ldots \sma \cC(c_{n-1},c_n) \sma \cM(c_n,c_0), \]
and face and degeneracy maps induced by composition in $\cC$, the actions of $\cC$ on $M$, and the inclusions of identity morphisms in $\cC$.
\end{definition}

The \emph{cyclic bar construction} $N^\cyc(\cC; \cM)$ is the geometric realization of $N_{\bullet}^{\cyc}(\cC; \cM)$ in spectra. As a special case, we may take $\cM = \cC(-, -)$ to be the spectral enrichment of $\cC$, in which case we drop it from the notation and write $N^\cyc \cC$. For example, if $R$ is a ring spectrum and $M$ is an $(R, R)$-bimodule, then we define $N^\cyc(R;M)$ and $N^\cyc R$ by regarding $R$ as a spectral category $\lan{R}$ with one object and $M$ as a bimodule over that spectral category.

\begin{example}
If $G$ is a well-based topological group, or grouplike monoid, the cyclic bar construction of $G$ in spaces admits an equivalence to the free loop space $LBG$. As suspension spectra commute with cyclic bar constructions, we get an equivalence
\[ \xymatrix{ N^\cyc(\Sigma^\infty_+ G) \ar[r]^-\sim & \Sigma^\infty_+ LBG } \]
\end{example}

We recall from \cite{blumberg_mandell} the sense in which the cyclic nerve is homotopy invariant. We say that $\cC$ is \emph{pointwise cofibrant} if each of the maps
\[ S \arr \cC(a,a), \qquad * \arr \cC(a,b), \qquad a \neq b \]
are cofibrations in the usual stable model structure on symmetric spectra, orthogonal spectra, or $S$-modules. We remark that in the case of symmetric or orthogonal spectra, if $\cC(a,a)$ is cofibrant then the unit map from $\bbS$ is automatically a cofibration.
A spectral functor $\cC \ra \cD$ is a \emph{pointwise weak equivalence} if each map $\cC(a,b) \ra \cD(a,b)$ is a stable equivalence of spectra. The ``pointwise'' notions for modules are defined similarly on each spectrum $\cM(a, b)$.

\begin{proposition}\label{prop:cylic_nerve_ho_invariant}
Suppose that $f: \cC \ra \cD$ is a pointwise weak equivalence of spectral categories that is the identity on objects. Let $\cM$ be a $\cC$-module, and $\cN$ a $\cD$-module, $g: \cM \ra f^*\cN$ be a pointwise weak equivalence of $\cC$-modules. If $\cC, \cD, \cM$ and $\cN$ are all pointwise cofibrant then the map induced on cyclic nerves by $f$ and $g$
\[ N^\cyc(\cC;\cM) \ra N^\cyc(\cD;\cN) \]
is an equivalence.
\end{proposition}

In order to be able to make functorial cofibrant approximations of $\cC$ and $\cM$, we recall the following from \cite[Prop 6.3, Thm 7.2]{schwede_shipley}:
\begin{proposition}
Suppose that our enrichments are taken in symmetric and orthogonal spectra, equipped with either the stable model structure or the positive stable model structure from \cite{mmss}.
\begin{itemize}
\item There is a cofibrantly generated model structure on the category of spectral categories with a fixed object set, in which the weak equivalences are measured pointwise and the cofibrant categories are pointwise cofibrant.
\item For each fixed spectral category $\cC$, the category of $\cC$-modules has a cofibrantly generated model structure where the weak equivalences are the pointwise weak equivalences. If $\cC$ is cofibrant then any cofibrant module over $\cC$ is pointwise cofibrant.
\item Given a pointwise weak equivalence of spectral categories $f: \cC \ra \cD$, restriction and extension of modules gives a Quillen equivalence between their module categories.
\end{itemize}
\end{proposition}

By the monoidal Quillen equivalence $(\bbN, \bbN^\#)$ between orthogonal spectra and $S$-modules \cite{mandell_may}, we then get a functorial cofibrant approximation for categories enriched in $S$-modules. Over a fixed category of the form $\bbN \cC$, where $\cC$ is enriched in orthogonal spectra, we get a cofibrant approximation functor for modules given by $\bbN Q(\iota^*\bbN^\# \cM) \overset\sim\to \cM$, where $\iota$ is the map $\cC \ra \bbN^\# \bbN \cC$ and $Q$ refers to cofibrant approximation of $\cC$-modules.

For any of the enrichments considered so far, we let $(-)^{c}$ denote a fixed choice of a cofibrant approximation functor on spectral categories, and also a fixed cofibrant approximation functor for modules over a given spectral category. We derive the cyclic nerve $N^\cyc(\cC;\cM)$ in a two step process.  We first take a cofibrant replacement $\cC^{c}$ of $\cC$ as a spectral category.  The $\cC$-module $\cM$ inherits the structure of a $\cC^{c}$-module by restriction along the approximation map, and we then take a cofibrant replacement $\cM^{c}$ of $\cM$ as a $\cC^{c}$-module. There is a natural map
\[ N^\cyc(\cC^{c};\cM^{c}) \ra N^\cyc(\cC;\cM) \]
which is an equivalence if $\cC$ and $\cM$ are pointwise cofibrant.
\begin{definition}
The topological Hochschild homology of a small spectral category $\cC$ with coefficients in $\cM$ is the derived cyclic nerve
\[ \THH(\cC; \cM) := N^\cyc(\cC^{c};\cM^{c}). \]
\end{definition}
It follows from Proposition \ref{prop:cylic_nerve_ho_invariant} that $\THH(- ; -)$ preserves weak equivalences in both variables. Using the monoidal Quillen adjunctions $(\mathbb P, \mathbb U)$ and $(\bbN, \bbN^\#)$ from \cite{mmss} and \cite{mandell_may}, we may pass freely between enrichments in symmetric spectra, orthogonal spectra, and $S$-modules without contaminating the homotopy type of THH.

\begin{remark}
When $\cC = \lan{R}$ is a spectral category with a single object, the cofibrant approximation $\lan{R}^{c}$ is a cofibrant approximation of $R$ as an $S$-algebra. So $\THH(\lan{R})$ is the cyclic bar construction on a cofibrant approximation of $R$, as in the usual definition of $\THH(R)$ (cf. \cite{shipley_THH}).
\end{remark}

We sometimes need to change the set of objects of $\cC$. For this purpose, we need a more general homotopy invariance statement for $\THH$ than Proposition \ref{prop:cylic_nerve_ho_invariant}.

\begin{definition}
The \emph{homotopy category} $\pi_0\cC$ of a spectral category $\cC$ is the category obtained by applying $\pi_0$ to the mapping spectra.
A \emph{Dwyer-Kan equivalence} or \emph{DK-equivalence} of spectral categories is a spectral functor $f: \cC \ra \cD$ for which all the maps $\cC(a,b) \ra \cD(f(a), f(b))$ are equivalences of spectra, and every object of $\cD$ is equivalent in $\pi_0\cD$ to an object in the image of $f$.
\end{definition}

\begin{proposition} \cite[Thm 5.10]{blumberg_mandell}
Suppose that $f: \cC \ra \cD$ is a Dwyer-Kan equivalence of spectral categories, $\cM$ is a $\cC$-module, and $\cN$ is a $\cD$-module, and there is a pointwise equivalence of $\cC$-modules $g: \cM \ra f^*\cN$. Then the natural map induced by $f$ and $g$ on THH
\[ \THH(\cC;\cM) \ra \THH(\cD;\cN) \]
is an equivalence.
\end{proposition}

We may define the THH of a large spectral category $\cC$ with skeletally small homotopy category $\pi_0 \cC$, by taking the THH of any small full subcategory containing at least one object in every isomorphism class in $\pi_0 \cC$. This is well-defined up to canonical zig-zag equivalence.

We now recall a result that allows us to compare morphisms of spectral categories. If $\cD$ is a spectral category, we let $\_{\cD}$ denote the underlying ordinary category, whose morphisms $\_{\cD}(a,b) = \Sp{}{}(S, \cD(a, b))$ are maps of spectra from the unit object $S$ to $\cD(a,b)$. A natural transformation $\eta$ between spectral functors $F, G: \cC \to \cD$, is the assignment of a morphism $\eta(c) \in \_{\cD}(F(c), G(c))$ to each $c \in \cC$, such the two obvious formulas give the same maps of spectra
\[ \cC(a,b) \to \cD(F(a),G(b)) \]
for every $a,b \in \cC$. We say that $\eta$ is a \emph{weak equivalence} if composing with the maps $\eta(c)$ induces a stable equivalence of mapping spectra in $\cD$.

\begin{proposition}\label{prop:equivalence_spectral_functors}
If $\eta$ is a weak equivalence between $F$ and $G$, then $\THH(F)$ and $\THH(G)$ are the same map in the homotopy category of spectra.
\end{proposition}
\begin{proof}
This follows easily from the observation of Blumberg and Mandell that $\eta$ induces a map into the 1-simplices of the \emph{Moore nerve} $w_1^M \cD$ \cite[2.3.7]{blumberg_mandell2}. This is a spectral category with objects the weak equivalences $a \to a'$ in $\_{\cD}$. The maps from $a \to a'$ to $b \to b'$ are given by a homotopy pullback of $\cD(a,b)$ and $\cD(a',b')$ along $\cD(a,b')$, but the paths in $\cD(a,b')$ are allowed to have any nonnegative length. The inclusion of constant paths gives a DK-equivalence $\cD \ra w_1^M \cD$, split by the evaluation at the two endpoints. The data of the natural transformation $\eta$ gives a spectral functor $\cC \to w_1^M \cD$, whose composition with the two splittings $w_1^M \cD \to \cD$ give the functors $F$ and $G$. It follows from a formal diagram-chase that $F$ and $G$ induce homotopic maps on $\THH$.
\end{proof}

\subsection{Perfect $R$-modules}\label{subsec:Perf}

We are ultimately interested in the THH of a ring spectrum $R$, but the transfer map is defined on the THH of the category of perfect modules over $R$.  Recall that a right $R$-module $M$ is \emph{perfect} if it is a retract in the homotopy category of a finite cell $R$-module. Equivalently, the composition map
\[ 
M \sma_R F_R(M,R) \cong F_R(R,M) \sma_R F_R(M,R) \arr F_R(M,M) 
\]
is an equivalence, where the smash product and mapping spectra are derived. The perfect $R$-modules form the smallest thick subcategory of right $R$-modules containing $R$ itself. In other words, they are the smallest full subcategory containing $R$ and closed under weak equivalences, cofiber sequences, and retracts.  

The notion of a perfect $R$-module makes sense in any symmetric monoidal category of spectra. However, we will need to choose a model whose mapping spectra $F_R(M,N)$ are derived, and which is closed under tensoring with bimodules. There are a few ways to do this, but the easiest way is to use the EKMM category of $S$-modules.
\begin{definition}\label{def:perf_r}
Let $R$ be an $S$-algebra in the sense of \cite{ekmm}. Then $\Perf_{R}$ will refer to the subcategory of perfect right $R$-modules consisting of all the cofibrant objects, and in addition the module $R$ itself. We enrich $\Perf_R$ in $S$-modules by the usual spectral enrichment $F_R(-,-)$.
\end{definition}
\noindent Since every $S$-module is fibrant, when $M$ is cofibrant the mapping spectrum $F_{R}(M, N)$ has the correct homotopy type. We recall that $R$ itself is not cofibrant as an $R$-module \cite[II.1.10]{ekmm}, but of course $F_R(R,N) \cong N$, and so all of the mapping spectra in $\Perf_R$ are actually derived.

It is not difficult to verify that our chosen model is Dwyer-Kan equivalent to several other models of $\Perf_{R}$, including the enriched Waldhausen category of cofibrant perfect $R$-modules as defined in \cite{blumberg_mandell2}.

\begin{lemma}
The homotopy category $\pi_0 \Perf_R$ is skeletally small.
\end{lemma}

\begin{proof}
Each perfect $R$-module $M$ is described up to weak equivalence by a finite $R$-module $X$ and a choice of self-map $e: X \to X$ which is idempotent on the homotopy category. One may embed the spectrum levels of all finite $R$-modules $X$ as set-theoretic subsets of $|R| \times 2^{|\R|}$, where $|R|$ is the total cardinality of the levels of $R$. Therefore the weak equivalence classes of perfect $R$-modules form a set.
\end{proof}

The following Morita invariance result allows us to recast the $\THH$ of $R$ in terms of its category of perfect modules.
\begin{lemma}\cite[Thm 5.12]{blumberg_mandell}\label{lem:BM}
Let $\cC$ be a spectral category that is pretriangulated in the sense of \cite[Def 5.4]{blumberg_mandell}, so in particular $\pi_0\cC$ is triangulated. If $\cC_0$ is a full spectral subcategory, and $\cC'$ is its thick closure in $\cC$, then the inclusion $\THH(\cC_0) \rightarrow \THH(\cC')$ is an equivalence of spectra.
\end{lemma}

The definition of thick closure used by Blumberg-Mandell is given in terms of the spectral enrichment.  In particular, it only agrees with the usual notion in the triangulated homotopy category if the mapping spectra are always derived.  It is for this reason that we must restrict to cofibrant perfect $R$-modules. The lemma provides a composite equivalence
\begin{equation}\label{eq:perfect_equivalence}
\THH(R) \oarr{\cong} \THH(F_{R}(R, R)) = \THH(\lan{R}) \oarr{\simeq} \THH(\Perf_{R}).
\end{equation}
If we had restricted to cofibrant $R$-modules, and included $THH(R)$ into the endomorphisms of the cofibrant $R$-module $R \sma QS$, we would get the same map $\THH(R) \ra \THH(\Perf_R)$ up to homotopy, by Proposition \ref{prop:equivalence_spectral_functors}.



\begin{remark}
If we had used orthogonal spectra to define $\Perf_R$, we would need to pass to bifibrant modules to make the mapping spectra derived. Unfortunately, that model does not accommodate the functor $M \sma_A -$ defined below, because the output of the functor is not fibrant, and small-object fibrant approximation is not spectrally enriched.
\end{remark}

\subsection{Bimodules and maps of spectral categories}\label{subsec:bimodules}

Continuing to work with $S$-modules, suppose that $A$ and $R$ are cofibrant $S$-algebras, and that $M$ is a cofibrant $(A, R)$-bimodule that is perfect as a right $R$-module. The functor $\Mod_{A} \arr \Mod_{R}$ given by $X \mapsto X \sma_{A} M$ is a Quillen left adjoint and restricts to give a spectral functor 
\[
\lambda_{M} = (-) \sma_{A} M \colon \Perf_{A} \arr \Perf_{R}
\]
on the subcategories of cofibrant perfect modules.  We write 
\[
\THH(\lambda_{M}) \colon \THH(\Perf_{A}) \arr \THH(\Perf_{R})
\]
for the induced map of derived cyclic nerves. When working in the homotopy category, we will sometimes write this as a map $\THH(A) \arr \THH(R)$ using (\ref{eq:perfect_equivalence}). If $M$ is not cofibrant, we write $\THH(\lambda_{M})$ for the induced map obtained by taking a cofibrant replacement of $M$. By Proposition \ref{prop:equivalence_spectral_functors} it immediately follows that this is invariant under the choice of $M$ up to weak equivalence:

\begin{proposition}\label{prop:lambda_m_invariant}
If $M \ra M'$ is a weak equivalence of cofibrant $(A,R)$ bimodules that are perfect over $R$, then $\THH(\lambda_M)$ and $\THH(\lambda_{M'})$ are homotopic.
\end{proposition}

It follows that when $A$ and $R$ are not cofibrant, we can take cofibrant replacements before defining $\lambda_M$, and the resulting map $\THH(\lambda_{M})$ on the homotopy category will be independent of the choice of replacement. We conclude this section by considering two examples, and defining the free loop transfer $\tau_\THH$.

\begin{example}
Suppose that $f \colon A \arr R$ is any map of ring spectra.  Then the $(A, R)$-bimodule $R$ is perfect as an $R$-module, and the operation $\lambda_{R} = (-) \sma_A R$ induces the usual covariant functoriality of $\THH$:
\[ \THH(\lambda_{R}) \colon \THH(A) \arr \THH(R) \]
\end{example}

\begin{example}\label{ex:lambda}
Suppose that $f \colon R \arr A$ is a map of ring spectra and that $A$ is perfect as an $R$-module. Then the functor $\lambda_{A} \colon \Perf_{A} \arr \Perf_{R}$ given by smashing with the $(A, R)$-bimodule $A$ is naturally isomorphic to the restriction functor $f^*$. We define the $THH$ transfer of $f$ to be the map induced by $\lambda_A$:
\[
 \tau_{\THH} = \THH(\lambda_{A}) \colon \THH(A) \arr \THH(R)
\]
\end{example}

\begin{definition}\label{def:tau_THH}
If $f: E \arr B$ is a perfect fibration of pointed, path-connected spaces, we define the free loop transfer of $f$ to be the THH transfer 
\[
  \Sigma^{\infty}_{+} LB \simeq \THH(\Sigma^{\infty}_{+} \Omega B) \xarr{\tau_{\THH}} \THH(\Sigma^{\infty}_{+} \Omega E) \simeq \Sigma^{\infty}_{+} LE
\]
of the map of ring spectra
\[ \Sigma^{\infty}_{+} \Omega f \colon \Sigma^{\infty}_{+} \Omega E \arr \Sigma^{\infty}_{+} \Omega B. \]
This definition is invariant under weak equivalence, so we may tacitly use the Kan loop group model for $\Omega B$ so that it is a topological group. That $\Sigma^{\infty}_{+} \Omega B$ is perfect over $\Sigma^{\infty}_{+} \Omega E$ can be deduced from extending the fiber sequence to the left
\[
\dotsm \arr \Omega E \oarr{\Omega f} \Omega B \arr F \arr E \oarr{f} B,
\]
and realizing $\Omega B \arr F$ as a principal $\Omega E$-bundle.  We will see that $\tau_\THH$ does not depend on the choice of basepoint. When $E$ is not connected, we construct $\tau_{\THH}$ on each component and add the resulting maps together. Similarly, when $B$ is not connected, we define the transfer on each component and add the results together.
\end{definition}

\section{Parametrized spectra}\label{sec:parametrized_spectra}

We now recall the facts that we need about parametrized spectra, following the book of May-Sigurdsson \cite{MS}.  Let $B$ be a topological space, homotopy equivalent to a cell complex.  Let $\Top/B$ denote the category of spaces $(Y, p) = (p \colon Y \arr B)$ over $B$. We write $\Top_{B}$ for the category of retractive spaces, or ex-spaces, which are spaces $(Y, p, s)$ over $B$ equipped with a section $s \colon B \arr Y$. We let $S_B^m$ denote the trivial product fibration over $B$ whose fiber is the sphere $S^m$; this naturally belongs to $\Top_B$.

We recall from \cite{mmss} that orthogonal spectra are given by diagrams of spaces over a certain topological category $\mathscr J$. We define a (parametrized) spectrum $X$ over $B$ to be a diagram of retractive spaces $\Top_B$ over $\mathscr J$. Concretely, $X$ consists of a sequence of $O(n)$-equivariant ex-spaces $(X(n), p(n), s(n))$ for $n \geq 0$, along with $(O(m) \times O(n))$-equivariant spectrum structure maps
\[
\sigma_{m, n} \colon S_{B}^{m} \sma_{B} X(n) \arr X(m + n) \quad \text{in $\Top_{B}$}
\] 
where $\sma_{B}$ denotes the fiberwise smash product of retractive spaces.  We write $\Sigma^{\infty}_{B} Y = \Sigma^{\infty}_{B} (Y, p, s)$ for the fiberwise suspension spectrum of an ex-space $(Y, p, s)$, defined by $(\Sigma^{\infty}_{B} Y)(n) = S^n_{B} \sma_{B} Y$, and write $\Sigma^{\infty}_{B} (Y, p)_{+}$ for the fiberwise suspension spectrum of the ex-space $(Y \sqcup B, p \sqcup \id_B, \id_B)$ obtained from $(Y, p)$ by adjoining a disjoint section. There is a fiberwise smash product $\sma_B$ of spectra over $B$, and it is symmetric monoidal with unit the fiberwise sphere spectrum $S_{B} = \Sigma^{\infty}_{B} (B, \id)_{+}$.  We write $\ho \Sp{}{B}$ for the homotopy category of May-Sigurdsson's stable model category of spectra over $B$.  

When $X$ is a parametrized spectrum over $B$, for each point $b \in B$, the fibers $X(n)_b$ give an orthogonal spectrum $X_b$, called the fiber spectrum of $X$ over $b$. This notion is homotopically meaningful if the projection maps $X(n) \arr B$ are at least quasifibrations. More formally, we characterize the fiber as a derived pullback. Recall that for every map of spaces $f \colon A \arr B$, May and Sigurdsson construct derived base-change functors
\[
f_{!} \colon \ho \Sp{}{A} \arr \ho \Sp{}{B} \quad f^* \colon \ho \Sp{}{B} \arr \ho \Sp{}{A} \quad f_* \colon \ho \Sp{}{A} \arr \ho \Sp{}{B}
\]
and two adjunctions, $(f_{!}, f^*)$ and $(f^*, f_*)$.  We emphasize that these categories are homotopy categories, and the symbols $f_{!}, f^*, f_*$ denote derived base-change functors.  The derived fiber spectrum $F_{b} X$ is the pullback $i_{b}^* X$, where $i_{b} \colon \{b\} \arr B$ is the inclusion of the point $b$.  A map of parametrized spectra $X \arr Y$ is an equivalence precisely when the map of derived fibers $F_b X \overset\sim\arr F_b Y$ is an equivalence for every $b \in B$.

We also remark that the fiberwise suspension spectrum operation $\Sigma^\infty_B (-)_+$ commutes with pushforward and pullback, both on a point-set level and in the derived sense:
\[
f_{!} \Sigma^{\infty}_{A} (Y, p)_{+} \cong \Sigma^{\infty}_{B} (Y, f \circ p)_{+}, \qquad f^* \Sigma^{\infty}_B (Z, p)_{+} \cong \Sigma^{\infty}_A (f^* Z, f^*(p))_{+}.
\]
Of course, to make the pullback $f^*Z$ derived we first replace $Z \arr B$ by a fibration.

\begin{example}\label{ex:def_of_r}
We generically write $r \colon B \arr \ast$ for the projection map to a point.  The base-change functor $r_{!} \colon \ho \Sp{}{B} \arr \ho \Sp{}{}$ simply collapses the section $s$ to a single basepoint.  In particular, $r_!$ takes each fiberwise suspension spectrum $\Sigma^{\infty}_{B} (Y, p)_{+}$ to an ordinary suspension spectrum $\Sigma^{\infty} Y_+$.
\end{example}

\begin{example}\label{ex:loops}
Let $\Delta \colon B \arr B \times B$ denote the diagonal map.  Then there is a canonical equivalence
\[
\Delta^* \Delta_{!} S_{B} \cong \Delta^* \Sigma^{\infty}_{B \times B} (B, \Delta)_{+} \cong \Sigma^{\infty}_{B} (LB, e_0)_+,
\]
where $LB = \Map(S^1, B)$ is the free loop space of $B$ and $e_0$ is the evaluation of a loop at the basepoint $0 \in S^1$.  Under this identification, the unit $\eta \colon S_{B} \arr \Delta^* \Delta_{!} S_{B}$ of the adjunction $(\Delta_{!}, \Delta^*)$ is the map of suspension spectra over $B$ given by the constant loops map $c \colon B \arr LB$.
\end{example}

\begin{example}
If $X$ is a spectrum over $A$ and $Y$ is a spectrum over $B$, then the external smash product $X \osma Y$ is a spectrum over $A \times B$ whose fiber at $(a, b)$ is the smash product $X_a \sma Y_b$.  When $A = B$, the fiberwise smash product $X \sma_{B} Y$ is canonically equivalent to the pullback $\Delta^* (X \osma Y)$ along the diagonal.
\end{example}

Finally we discuss some compatibilities between base-change functors.
Suppose we are given a commutative diagram of maps of topological spaces
\begin{equation}\label{eq:simple_square}
\xymatrix{
A \ar[r]^{f} \ar[d]_{i} & B \ar[d]^{j} \\
C \ar[r]^{g} & D 
}
\end{equation}
and a spectrum $X$ over $C$.  Then there is a natural map $\alpha \colon f_{!} i^* X \arr j^* g_{!} X$ of spectra over $B$ defined by the composite
\begin{equation}\label{eq:beck_chevalley}
\alpha \colon f_{!} i^* \oarr{\eta} f_{!} i^* g^* g_{!} \cong f_{!} f^* j^* g_{!} \oarr{\epsilon} j^* g_{!}
\end{equation}
of the unit and counit for the adjunctions of base change functors induced by $f$ and $g$.  When the square (\ref{eq:simple_square}) is a homotopy pullback square, then the transformation $\alpha$ is an equivalence of derived functors, called the \emph{Beck-Chevalley isomorphism} \cite{MS}*{Thm. 13.7.7}. There is also a map
\[
\rho \colon f_{!}(f^*X \sma_{A} Y) \arr X \sma_{B} f_{!} Y
\]
induced by the diagram
\[
\xymatrix{
A \ar[r]^{f} \ar[d]_{(f, 1)} & B \ar[d]^{\Delta} \\
B \times A \ar[r]^{1 \times f} & B \times B
}
\]
and $\rho$ is an equivalence of derived functors, even though the diagram need not be a homotopy pullback \cite{MS}*{(11.4.5) and proof of Thm. 13.7.6}.  The equivalence $\rho$ is often called the \emph{projection formula}.

The construction of the map \eqref{eq:beck_chevalley} is not only natural in $X$ but also functorial with respect to pasting of commutative squares.  If we mark the square with the name of the transformation constructed in this manner, then the equality of commutative diagrams 
\[
\xymatrix{
A \ar[r]^{f} \ar[d]_{i} \ar@{}[dr]|{\Downarrow \alpha} & B \ar[d]^{j} \\
C \ar[r]^{g} \ar[d]_{k} \ar@{}[dr]|{\Downarrow \beta} & D  \ar[d]^{l} \\
U \ar[r]^{h} & V
}
\quad \xymatrix{ \ar@{}[dr]|{=} & \\ & } \quad  
\xymatrix{
A \ar[d]_{k \circ i} \ar[r]^{f} \ar@{}[dr]|{\Downarrow \gamma} & B \ar[d]^{l \circ j} \\
U \ar[r]^{h} & V
}
\]
implies the equality of natural transformations $\beta \circ \alpha = \gamma$.  

\section{Duality, shadows, and traces in a bicategory}\label{sec:bicats}

In this section we briefly recall the theory of duality and traces in a bicategory.  See the papers of Ponto-Shulman \citelist{\cite{PS12} \cite{PS13}} for a more extensive treatment.

\subsection{Bicategories: review and examples}

Let $\cB$ be a bicategory.  $\cB$ consists of objects $A, B, C, \dotsc$, and a category $\cB(A, B)$ for each pair of objects.  We write $X \colon A \arr B$ for an object of $\cB(A, B)$ and call it a 1-morphism in $\cB$.  A morphism $f \colon X \arr Y$ between 1-morphisms is called a 2-morphism.  The bicategory $\cB$ also has a horizontal composition functor
\[
- \otimes_{B} - \colon \cB(A, B) \times \cB(B, C) \arr \cB(A, C)
\]
which is associative and unital with respect to unit 1-morphisms $U_A \in \cB(A, A)$, all up to coherent 2-isomorphisms.  A map of bicategories $\cB \arr \cB'$ consists of maps of 0-cells, 1-cells, and 2-cells, which commute with vertical composition strictly and horizontal composition up to coherent 2-isomorphisms.

Note that we use diagrammatic order in our notation for horizontal composition, so that the composite of $X \colon A \arr B$ and $Y \colon B \arr C$ is $X \otimes Y \colon A \arr C$, and not $Y \otimes X$ as in functional notation.  For simplicity, we will often omit the associativity and unit isomorphisms, as is common practice with symmetric monoidal categories.

\begin{example}
There is a bicategory $\Bimod$ whose objects are rings and whose category $\Bimod(A, B)$ of 1- and 2-morphisms from a ring $A$ to a ring $B$ is the category of $(A, B)$-bimodules and bimodule homomorphisms.  The horizontal composition of $M \in \Bimod(A, B)$ and $N \in \Bimod(B, C)$ is the tensor product $M \otimes_{B} N \in \Bimod(A, C)$.

Similarly, if we let $S$ denote the sphere spectrum, there is a bicategory $\Bimod_{S}$ whose 0-cells are ring spectra $A$. The 1-morphisms from $A$ to $B$ are $(A, B)$-bimodules and the 2-morphisms are morphisms in the homotopy category of bimodules. The composition of 1-cells in $\Bimod_{S}$ is given on homotopy categories by the derived smash product over a ring $A$. When the rings and bimodules are ``pointwise cofibrant'' as in section \ref{sec:spectral_cats}, this can be modeled by the two-sided bar construction:
\begin{align*}
\Bimod_{S}(A_1, A_2) \times \Bimod_{S}(A_2, A_3) &\oarr{\circ} \Bimod_{S}(A_1, A_3) \\
(M, N) &\longmapsto M \sma^{\bL}_{A_2} N = B(M, A_2, N).
\end{align*}
This definition may be carried out in orthogonal spectra or in $S$-modules, but the resulting bicategories are equivalent.
\end{example}

\begin{example}\label{ex:Ex}
There is a bicategory $\Ex$ whose objects are topological spaces and whose category $\Ex(A, B)$ of 1- and 2-morphisms from a space $A$ to a space $B$ is the homotopy category $\ho \Sp{}{A \times B}$ of spectra over the cartesian product $A \times B$.  The horizontal composition of $X \in \Ex(A, B)$ and $Y \in \Ex(B, C)$ is the spectrum
\[
X \odot_{B} Y = {\pi_{!}^B} \Delta_{B}^* (X \osma Y) \quad \text{over $A \times C$},
\] 
where 
\begin{align*}
\Delta_{B} \colon A \times B \times C &\arr A \times B \times B \times C \\
\pi^B \colon A \times B \times C &\arr A \times C 
\end{align*}
are the map induced by the diagonal of $B$ and the projection off of $B$.


One of the important features of $\Ex$ is that certain 1-cells encode the base-change functors $f_{!}$ and $f^*$.  For each map of spaces $f \colon A \arr B$, we define a 1-morphism $S_{f} \in \Ex(B, A)$ by
\[
S_{f} = \Sigma^{\infty}_{B \times A} (A, (f, \id))_{+} \cong (f, \id)_{!} S_{A}.
\]
We also define ${}_fS \in \Ex(A, B)$ by the same rule, but over $A \times B$:
\[
{}_fS = \Sigma^{\infty}_{A \times B} (A, (\id, f))_{+} \cong (\id, f)_{!} S_{B}.
\]
Because of the convention on the order of composition for $\odot$, it is important to distinguish ${}_{f}S$ from $S_{f}$.  They are different one-cells in the bicategory $\Ex$.  Composing with $S_{f}$ and ${}_{f}S$ on the right in $\Ex$ gives the pullback and pushforward functors in the second entry \cite{MS}*{\S 17.2}:
\begin{align*}
X \in \Ex(C, B) \quad & \Longrightarrow \quad X \odot_{B} S_{f} \cong (\id \times f)^* X \in \Ex(C, A) \\
X \in \Ex(C, A) \quad & \Longrightarrow \quad X \odot_{A} {}_{f}S \cong (\id \times f)_{!} X \in \Ex(C, B).
\end{align*}
Similarly, composing with $S_{f}$ and ${}_{f}S$ on the left gives pushforward and pullback functors in the first entry:
\begin{align*}
X \in \Ex(A, C) \quad & \Longrightarrow \quad S_{f} \odot_{A} X \cong (f \times \id)_{!} X \in \Ex(B, C) \\
X \in \Ex(B, C) \quad & \Longrightarrow \quad {}_{f}S \odot_{B} X \cong (f \times \id)^* X \in \Ex(A, C).
\end{align*}
\end{example}
\noindent One may adapt the mnemonic that the $f$ decoration is on the same side as the source of the map $f$, and that tensoring with a base-change functor gives a pushforward if $f$ is on the same side as the tensor, and a pullback if $f$ is opposite the tensor.

\begin{example}\label{ex:Ex_B}
Fix a space $B$.  There is a variant of $\Ex_{B}$ of the bicategory $\Ex$ whose objects are spaces $A$ over $B$, and with $\Ex_{B}(A, C)$ given by the homotopy category of parametrized spectra over $A \times_{B} C$.  This is the example most relevant to our proof of Theorem \ref{thm:intro_becker_gottlieb_compatibility}.
\end{example}

\subsection{Duality in a bicategory}
We now fix a bicategory $\cB$ and recall the notion of duality internal to $\cB$.

\begin{definition}\label{def:dual_pair_bicat}
Suppose that $X \in \cB(A, B)$ and $Y \in \cB(B, A)$ are 1-morphisms in $\cB$.  We say that $(X, Y)$ is a \emph{dual pair} if there are 2-morphisms
\[
\coev(X) \colon U_{A} \arr X \otimes_{B} Y \qquad \eval(X) \colon Y \otimes_{A} X \arr U_{B}
\]
such that 
\begin{align*}
X \cong U_{A} \otimes_{A} X \xarr{\coev \otimes 1} X \otimes_{B} Y \otimes_{A} X \xarr{1 \otimes \eval} X \otimes_{B} U_{B} \cong X \\
Y \cong Y \otimes_{A} U_{A} \xarr{1 \otimes \coev} Y \otimes_{A} X \otimes_{B} Y \xarr{\eval \otimes 1} U_{B} \otimes_{B} Y \cong Y
\end{align*}
are the identity 2-morphisms of $X$ and $Y$.

If in addition the maps $\coev(X)$ and $\ev(X)$ are isomorphisms, then $X$ and $Y$ form an \emph{equivalence} between the 0-cells $A$ and $B$. An \emph{equivalence of bicategories} is a morphism $\phi: \cB \ra \cB'$ giving equivalences of categories $\cB(A,B) \arr \cB'(\phi(A),\phi(B))$, such that every 0-cell in $\cB'$ is equivalent to a 0-cell in $\phi(\cB)$.
\end{definition}

\noindent Notice that Definition \ref{def:dual_pair_bicat} is \emph{not} symmetric.  If $(X, Y)$ form a dual pair, then it need not be true that $(Y, X)$ is a dual pair.  To emphasize this, we say that $X$ is right-dualizable, or dualizable over $B$, with right dual $Y$.

It is a straightforward consequence of the definitions that if $(X, Y)$ and $(X', Y')$ are dual pairs, then $(X \otimes X', Y' \otimes Y)$ is a dual pair (assuming that the one-cells are composable).  The evaluation and coevaluation morphisms for the new dual pair are the composites of the evaluation and coevaluation morphisms for the constituent dual pairs. One easy consequence is that dual pairs may be transported along any equivalence of bicategories.

In addition, if the categories $\cB(A,B)$ have finite biproducts, and the horizontal composition in $\cB$ has the distributive property, then a finite sum of dual pairs $(X_i,Y_i)$ is a dual pair. The coevaluation and evaluation morphisms are the block sum of the coevaluation and evaluation morphisms for each of the pairs $(X_i,Y_i)$.

\begin{remark}
The definition of duality implies that $- \otimes_A X$ and $- \otimes_B Y$ are adjoint functors between $\cB(C,A)$ and $\cB(C,B)$ for all 0-cells $C$. As a consequence, if the functors $- \otimes_A X: \cB(C,A) \arr \cB(C,B)$ come with right adjoints $R$, the evaluation map is adjoint to an isomorphism $Y \cong R(U_B)$, and the coevaluation map is the unit $U_A \arr R(X) \cong X \otimes_B Y$ of the adjunction. A consequence of this observation is that duals are unique. We will use this explicitly in the proof of Proposition \ref{prop:geometric_duality}.
\end{remark}

\begin{example}\label{ex:dualizable_in_bimodS}
If $A$ and $R$ are ring spectra, an $(A, R)$-bimodule $M$ is right dualizable if and only if $M$ is perfect as an $R$-module. The right dual admits a canonical equivalence $DM \simeq F_{R}(M,R)$ such that the coevaluation and evaluation maps are given by the left $A$-module action on $M$ and the composition pairing:
\[
\coev(M) \colon A \arr  F_{R}(M,M) \simeq M \sma_{R}^{\bL} DM \]
\[ \eval(M) \colon DM \sma_{A}^{\bL} M \simeq F_{R}(M,R) \sma_A^{\bL} F_{R}(R,M) \arr R
\]
\end{example}

\begin{example}\label{ex:dual_pair_f_S}
Duality in the bicategory $\Ex$ is called \emph{Costenoble-Waner duality}, and was studied extensively by May-Sigurdsson. 
For any map of spaces $f \colon A \arr B$, the base-change spectra ${}_{f}S$ and $S_{f}$ form a Costenoble-Waner dual pair $({}_{f}S, S_{f})$ in $\Ex$ \cite{MS}*{17.3.1}. This is a formal consequence of the properties of the diagonal map.  Note that the order here is crucial, since $(S_{f}, {}_{f}S)$ is rarely a dual pair.
\end{example}

\begin{example}\label{ex:CWduality_of_M}
If $M$ is a closed smooth manifold, then the sphere spectrum $S_M$ may be considered as a 1-morphism in $\Ex(\ast, M)$, and it is right dualizable. The right dual is the stable spherical fibration $S^{-TM}$ associated to the negative tangent bundle of $M$ \cite{MS}*{18.2.5}. The Costenoble-Waner dual pair $(S_M, S^{-TM})$ in the bicategory $\Ex$ gives rise to Atiyah duality, namely the dual pair 
\[
(M_+, M^{-TM}) = (S_M \odot_{M} {}_{r}S, S_{r} \odot_{M} S^{-TM}) 
\]
in the homotopy category of spectra. Here, $r \colon M \arr \ast$ is the projection off of $M$ so that the base-change spectra ${}_{r}S$ and $S_{r}$ encode the Thom space construction $r_{!}$, as in Example \ref{ex:def_of_r}.

More generally, the sphere spectrum $S_{X}$ over a space $X$, considered as a 1-morphism in $\Ex(\ast, X)$, is right dualizable if $X$ is a retract of a homotopy finite space \cite{MS}*{18.5.1}.  Even more generally, we will see below in Proposition \ref{prop:base_change_dualble_Ex} that the base-change spectrum $S_{f}$ associated to a perfect fibration $f$ is right dualizable.  In this case, $S_{f}$ participates in both the canonical dual pair $({}_{f}S, S_{f})$ and in the dual pair $(S_{f}, D S_{f})$. When $f$ is a smooth bundle with closed fibers, we will see that the dual pair $(S_f, D S_f)$ is given by a fiberwise variant of Atiyah duality.
\end{example}

A monoidal category $(\cM, \otimes, U)$ may be considered as a one-object bicategory with $\cB(\ast, \ast) = \cM$, and the usual notion of a dual pair in $\cM$ coincides with the notion of a dual pair in $\cB$.  If $(X, Y)$ is a dual pair in $\cM$ and $\cM$ is a symmetric monoidal category, then the trace of a morphism $\phi \colon P \otimes X \arr X \otimes Q$ is defined to be the composite
\[
\tr(\phi) \colon P \cong P \otimes U \oarr{\coev} P \otimes X \otimes Y \oarr{\phi \otimes 1} X \otimes Q \otimes Y \oarr{\gamma} Y \otimes X \otimes Q \oarr{\eval} U \otimes Q \cong Q.
\]
The structure of the symmetry isomorphism $\gamma$ of $\cM$ also means that we don't need to be careful about the order of $\otimes$ in defining coevaluation and evaluation maps: $(X, Y)$ is a dual pair if and only if $(Y, X)$ is a dual pair.  

\begin{example}\label{ex:becker_gottlieb_transfer}
An object of the symmetric monoidal category of spectra $(\ho \Sp{}{}, \sma, S)$ is dualizable if and only if it is a finite spectrum.  Similarly, an object of the symmetric monoidal category $(\ho \Sp{}{B}, \sma_{B}, S_{B})$ of spectra over $B$ is dualizable if and only if each derived fiber is dualizable in $\ho \Sp{}{}$ \cite{MS}*{15.1.1}. If $f \colon E \arr B$ is a perfect fibration, then each fiber of the parametrized spectrum $f_{!} S_{E} \cong \Sigma^{\infty}_{B} (E, f)_+$ is a dualizable spectrum.  Thus $f_{!}S_{E}$ admits a fiberwise dual $D f_{!} S_{E}$.  The trace of the fiberwise diagonal map 
\[
\delta_{} \colon f_{!} S_{E} \cong \Sigma^{\infty}_{B} (E, f)_+ \arr  \Sigma^{\infty}_{B} (E \times_{B} E, f)_+ \cong f_{!} S_{E} \sma_{B} f_{!} S_{E}
\]
is the composite
\[
\tr(\delta_{}) \colon S_{B} \oarr{\coev} f_{!} S_{E} \sma_{B} D f_{!} S_{E} \xarr{\delta_{} \sma \id} f_{!} S_{E} \sma_{B} f_{!} S_{E} \sma_{B} D f_{!} S_{E} \xarr{\eval \circ \gamma} f_{!} S_{E}
\]
This trace is called the Becker-Gottlieb pretransfer and we use the notation $\tau_{B}(f) = \tr(\delta_{})$.  The Becker-Gottlieb transfer is obtained by base-change to a point \cite{BG76}:
\[
\tau(f) = r_{!} \tau_{B}(f) \colon \Sigma^{\infty}_{+} B \cong r_{!} S_{B} \arr r_{!} f_{!} S_{E} \cong \Sigma^{\infty}_{+} E.
\]
\end{example}

\begin{remark}
Morphisms in a symmetric monoidal category, such as $\Sp{}{B}$, may be described using the calculus of string diagrams (see, for example, \citelist{\cite{JS} \cite{Sel}} and references therein).  Using this notation, the Becker-Gottlieb pretransfer is represented by Figure \ref{fig:pretransfer1}, where we have used the abbreviations $E_+ = f_{!} S_{E}$ and $DE_+ = D f_{!} S_{E}$.
\begin{figure}
\begin{tikzpicture}[scale=0.8]
      \node[fill=white,draw,circle,inner sep=1pt, font=\scriptsize] (delta) at (3,0) {$\delta$};
      \draw (delta) to[out=-135, in=90] node [ed, near end] {$E_+$} (2.5, -1.5);
      \draw (delta) to[out=-45, in=90] node [ed, near end] {$E_+$} (3.5, -1.5) to  ++(0,-1.5);
      \draw[out=90, in=-45] (delta) to (2,2) coordinate (cross);
      \draw[out=135, in= -90] (cross) to ++(-.5,1) 
      to[out=90, in=180] node [ed,near start] {$E_+$} ++(.5, 1)
      to[out=0, in=90] node [ed,near end] {$DE_+$} ++(.5, -1)
      to[out=-90, in=45] (cross);
      \draw (cross) to[out=-135, in=90] (1.5,0) to node [ed, near end, swap] {$DE_+$}  ++(0,-1.5) coordinate (bl);
      \draw (bl) to[out=-90, in=180] ++(.5, -1)
      to[out=0, in=-90]  ++(.5, 1);
\end{tikzpicture}
\caption{The Becker-Gottlieb pretransfer $\tau_{B}(f) \colon S_{B} \arr f_{!} S_{E}$ as a morphism in the symmetric monoidal category $(\Sp{}{B}, \sma_{B}, S_{B})$.}\label{fig:pretransfer1}
\end{figure}
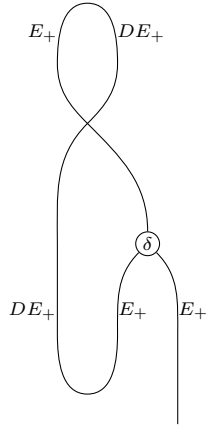
\end{remark}

We would like to use the same definition to take the trace of an endomorphism $\phi \colon X \arr X$ of a right dualizable 1-morphism $X \in \cB(A, B)$ in a bicategory.  Unfortunately, the formula for $\tr(\phi)$ does not make sense in a general bicategory because there is no symmetry isomorphism $\gamma \colon X \otimes_{B} Y \arr Y \otimes_{A} X$.  Notice that when $A \neq B$, these objects don't even live in the same category.

\subsection{Shadows}
Ponto provided a solution to this problem by introducing the notion of a shadow functor \cite{PoThesis}.  A shadow functor on a bicategory $\cB$ with values in a category $\cC$ consists of a functor
\[
\llan{-} \colon \cB(A, A) \arr \cC \qquad \text{for each object $A$ in $\cB$}
\] 
along with a natural transformation
\[
\theta \colon \llan{X \otimes_{B} Y} \arr \llan{Y \otimes_{A} X}
\]
for $X \in \cB(A, B)$ and $Y \in \cB(B, A)$ such that the following diagrams commute (when they parse correctly).
\begin{equation}\label{eq:shadow1}
\xymatrix{
\llan{(X \otimes Y) \otimes Z} \ar[r]^{\theta} \ar[d] & \llan{Z \otimes (X \otimes Y)} \ar[r] & \llan{(Z \otimes X) \otimes Y} \ar[d]^{\theta} \\
\llan{X \otimes (Y \otimes Z)} \ar[r]^{\theta} & \llan{(Y \otimes Z) \otimes X} \ar[r] & \llan{Y \otimes (Z \otimes X)}
}
\end{equation}
\begin{equation}
\xymatrix{
\llan{X \otimes U_{A}} \ar[r]^{\theta} \ar[dr] & \llan{U_{A} \otimes X} \ar[r]^{\theta} \ar[d] & \llan{X \otimes U_A} \ar[dl] \\
& \llan{X}
}
\end{equation} 
The unmarked arrows are induced by the associativity and unit isomorphisms in the bicategory $\cB$.  On first glance \eqref{eq:shadow1} might look like one of the axioms for a braiding on a monoidal category, but it is a straightforward consequence of the axioms that $\theta^2 = 1$.  Instead, one should think of $\theta$ as a cyclic rotation.  In fact, the manipulation of diagrams for a shadow on a bicategory can be done using a calculus of string diagrams on a cylinder \cite{PS13}.  We call $\theta$ the cyclic isomorphism associated to the shadow functor $\llan{-}$.

If $(X, Y)$ is a dual pair in $\cB$, and $\cB$ is equipped with a shadow functor valued in $\cC$, then the trace $\tr(\phi)$ of a 2-morphism $\phi \colon P \otimes_{A} X \arr X \otimes_{B} Q$ is defined to be the composite
\begin{align*}
\tr(\phi) \colon \llan{P} \cong \llan{P \otimes_{A} U_{A}} \oarr{\coev} &\llan{P \otimes_{A} X \otimes_{B} Y} \oarr{\phi \otimes 1} \llan{X \otimes_{B} Q \otimes_{B} Y} \\
&\oarr{\theta} \llan{Y \otimes_{A} X \otimes_{B} Q} \oarr{\eval} \llan{ U_{B} \otimes_{B} Q} \cong \llan{Q}
\end{align*}
in the category $\cC$.  Notice that this only makes sense if $P \in \cB(A, A)$ and $Q \in \cB(B, B)$. We will focus on the case where $P = U_A$, $Q = U_B$, and $\phi$ is the identity of $X$.

\begin{example}\label{ex:THH_as_shadow}
The bicategory $\Bimod$ has a shadow functor taking values in abelian groups.  The shadow of an $(R, R)$-bimodule $M$ is the coequalizer
\[
\xymatrix{
R \otimes M \ar@<.5ex>[r] \ar@<-.5ex>[r] & M \ar[r] & \llan{M}
}
\]
of the left and right actions of $R$ on $M$. The bicategory $\Bimod_S$ of ring spectra and bimodules has a similar shadow given by THH:
\[ \llan{M} = \THH(R;M) \]
We choose the cyclic isomorphism $\theta$ for an $(A,R)$ bimodule $M$ and an $(R,A)$ bimodule $N$ to be the map
\begin{equation}\label{eq:def_theta}
\xymatrix{N^{\cyc}(A; B(M, R, N)) \ar[r]^{\theta} & N^{\cyc}(R; B(N, A, M)) }
\end{equation}
which interchanges the two simplicial directions and on the $(m, n)$-simplices applies the symmetry isomorphism in the category of spectra
\[
(\overbrace{A \sma \dotsm \sma A}^{m} \sma M) \sma ( \overbrace{R \sma \dotsm \sma R}^{n} \sma N) \oarr{\gamma} (R \sma \dotsm \sma R \sma N) \sma (A \sma \dotsm \sma A \sma M).
\]

If $A$ and $R$ are both suspension spectra of topological groups, this symmetry isomorphism may be recast as the canonical isomorphism
\begin{equation}
\xymatrix{S \sma_A [B(M, R, R) \sma_R B(N, A, A)] \ar[r]^{\theta} & S \sma_R [B(N, A, A) \sma_A B(M, R, R)] }
\end{equation}
where $A$ acts on the left on both $B(M,R,R)$ and $B(N,A,A)$, the left action on $B(N,A,A)$ being the composition of the right $A$-action on the rightmost copy of $A$ and the involution of $A$.
\end{example}

\begin{example}\label{ex:shadow_Ex} The bicategory $\Ex$ has a shadow functor with values in $\ho \Sp{}{}$.  The shadow of a spectrum $X$ over $B \times B$ is the spectrum
\[
\llan{X} = r_{!} \Delta_{B}^* X.
\]
For example, the shadow of the unit 1-cell $U_{B} \in \Ex(B, B)$ is the suspension spectrum of the free loop space (Example \ref{ex:loops}):
\[
\llan{U_{B}} \cong r_{!} \Delta^* \Delta_{!} S_{B} \cong \Sigma^{\infty}_{+} LB.
\]
Now suppose that $f \colon E \arr B$ is a perfect fibration.  Later we will prove that the base-change spectrum $S_{f}$ is right dualizable in $\Ex$ (Proposition \ref{prop:base_change_dualble_Ex}).  In other words, there is a dual pair $(S_{f}, D S_{f})$ of 1-morphisms in $\Ex$.  The trace of the identity 2-morphism $\id \colon S_{f} \arr S_{f}$ is the map of spectra
\[
 \Sigma^{\infty}_{+} LB \cong \llan{U_{B}} \oarr{\coev} \llan{S_{f} \odot_{E} D S_{f}} \oarr{\theta} \llan{D S_{f} \odot_{B} S_{f}} \oarr{\eval} \llan{U_{E}} \cong \Sigma^{\infty}_{+} LE.
\]
This map is known to be a special case of the fiberwise Reidemeister trace by the work of Ponto-Shulman \cite{PS14}*{\S 7}.  We will simply call it the \emph{Reidemeister trace} associated to $f$.  
\end{example}

\begin{remark}\label{rem:additive}
In the previous two examples, if the bimodule or parametrized spectrum $M$ is expressed as a finite wedge of dualizable pieces $M_i$, then the trace of the identity of $M$ is the sum of the traces of the identity maps of the pieces $M_i$.
\end{remark}

\section{The free loop transfer is the Reidemeister trace}\label{sec:reidemeister}

In this section, we show that the Reidemeister trace associated to a perfect fibration $f \colon E \arr B$ (Example \ref{ex:shadow_Ex}) agrees with the THH transfer (Definition \ref{def:tau_THH}), proving Theorem \ref{thm:reid_trace_is_THH_transfer}.

\subsection{Indexed monoidal categories and an equivalence of shadows}

The first step is to understand the Reidemeister trace as a trace in the bicategory $\Bimod_S$ of ring spectra and bimodules.  Our treatment will use a result whose proof appears in the companion paper \cite{LiMa1}: the derived fiber functor $F_b \colon \Sp{}{B} \arr \Mod_{\Sigma^{\infty}_{+}\Omega B}$ is the right adjoint in a Quillen equivalence between the May-Sigurdsson model category of parametrized spectra over $B$ and the stable model structure on $\Sigma^{\infty}_{+}\Omega B$-modules. In fact, the Quillen equivalence respects the symmetric monoidal structures and the functoriality in the entry $B$ along base change functors.  We now introduce the language of indexed symmetric monoidal categories in order to make this agreement precise.

Let $S$ be a cartesian monoidal category.  An $S$-indexed symmetric monoidal category is a pseudofunctor $\cM $ from $S^{\op}$  to the 2-category of symmetric monoidal categories, strong symmetric monoidal functors, and monoidal transformations.  In other words, for each object $A$ of $S$ there is a symmetric monoidal category $\cM^{A}$, and for each morphism $f \colon A \arr B$ there is a strong symmetric monoidal functor $f^* \colon \cM^{B} \arr \cM^{A}$, along with natural monoidal isomorphisms $(g \circ f)^* \cong f^* \circ g^*$ and $(\id_{A})^* \cong \id_{\cM^{A}}$ satisfying associativity and unit conditions.

For example, when $S = \Top$, the assignment $B \longmapsto \ho \Sp{}{B}$, along with the symmetric monoidal structure and base change functors, defines a $\Top$-indexed symmetric monoidal category $\ho \Sp{}{(-)}$.  When $S$ is the category $\Top_*^{\mathrm{conn}}$ of based connected topological spaces, the assignment $B\longmapsto \ho \Mod_{\Sigma^{\infty}_{+}\Omega B}$ defines a  $\Top_*^{\mathrm{conn}}$-indexed symmetric monoidal category $\ho \Mod_{\Sigma^{\infty}_{+}\Omega(-)}$.  We continue to write $\ho\Sp{}{(-)}$ for the restriction of $B \longmapsto \ho \Sp{}{B}$ to a $\Top_*^{\mathrm{conn}}$-indexed symmetric monoidal category.  
\begin{proposition}\label{prop:companion}\cite[Thm. 1.2]{LiMa1}
The derived fiber functor induces an equivalence 
\[
\ho \Sp{}{(-)} \simeq \ho \Mod_{\Sigma^{\infty}_{+}\Omega(-)}
\]
of $\Top_*^{\mathrm{conn}}$-indexed symmetric monoidal categories.  
\end{proposition}
\noindent This result may also be deduced from a corresponding result at the level of $\infty$-categories proved by Ando-Blumberg-Gepner \cite{ando_blumberg_gepner}.

As explained in \cite[Thm. 5.2]{PS12}, each $S$-indexed symmetric monoidal category $\cM$ gives in a canonical way a bicategory $\cB$ equipped with a shadow functor that takes values in the symmetric monoidal category $\cM^{\ast}$.  We will not need all of the details, but the essential point is that the 0-cells of the bicategory are the objects of $S$ and the category $\cB(A, B)$ of 1-cells and 2-cells is the category $\cM^{A \times B}$ indexed by the cartesian product $A \times B$.  The shadow functor is defined on an endo 1-cell $X \in \cB(A, A)$ by the formula $\llan{X} = (\pi_A)_{!} (\Delta_{A})^* X$, where $\Delta_{A} \colon A \arr A \times A$ is the diagonal and $\pi_A \colon A \arr \ast$ is the canonical morphism to the final object.  The composition $-\otimes_{A} -$ in the bicategory $\cB$ is defined in a similar manner and the cyclic isomorphism 
\[
\theta \colon \llan{X \otimes_{B} Y} \arr \llan{Y \otimes_{A} X}
\]
is given by a series of canonical isomorphisms interchanging the constituent derived base change functors. 

For the $\Top$-indexed symmetric monoidal category $\ho\Sp{}{(-)}$, this process gives the bicategory $\Ex$ and its shadow functor $\llan{-}$ defined as in Example \ref{ex:shadow_Ex}. In this section, we will work with the full sub-bicategory $\Ex_*^{\mathrm{conn}}$ spanned by the pointed connected spaces which arises by restricting the indexing category from $\Top$ to $\Top_*^{\mathrm{conn}}$.

On the other hand, applying this procedure to the $\Top_*^{\mathrm{conn}}$-indexed symmetric monoidal category $B \mapsto \ho \Mod_{\Sigma^{\infty}_{+}\Omega B}$ gives the full sub-bicategory $\Bimod_{S}^{\gp}$ of $\Bimod_{S}$ on the rings of the form $\Sigma^{\infty}_{+}\Omega B$, where $B$ is a pointed connected space. The shadow functor takes each $(\Sigma^{\infty}_{+} \Omega B, \Sigma^{\infty}_{+} \Omega B)$-bimodule $M$ to the spectrum $\llan{M} = \epsilon_{!} \Delta^* M$, where $\Delta^*$ is the derived restriction along the composite ring map
\[
\Sigma^{\infty}_{+} \Omega B \oarr{\Delta} \Sigma^{\infty}_{+} \Omega B \sma \Sigma^{\infty}_{+} \Omega B \xarr{ \id \sma \chi} \Sigma^{\infty}_{+} \Omega B \sma \Sigma^{\infty}_{+} \Omega B^{\op} =: \Sigma^{\infty}_{+} \Omega B^{e}
\]
induced by the diagonal of $B$ and the anti-automorphism of $\Omega B$ taking an element to its inverse. The functor $\epsilon_{!}$ is the left adjoint of the derived restriction functor along the augmentation $\epsilon \colon \Sigma^\infty_{+} \Omega B \arr S$, and takes a module $M$ to the derived smash product $M \sma^{\bL}_{\Sigma^{\infty}_{+} \Omega B} S$.  Expanding and simplifying, we find that the shadow is given by the topological Hochschild homology of $\Sigma^{\infty}_{+} \Omega B$ with coefficients in $M$:
\[
\llan{M} \simeq M \sma^{\bL}_{\Sigma^{\infty}_{+} \Omega B^e} \Sigma^{\infty}_{+} \Omega B^{e} \sma^{\bL}_{\Sigma^{\infty}_{+} \Omega B} S \simeq M \sma^{\bL}_{\Sigma^{\infty}_{+} \Omega B^{e}} \Sigma^{\infty}_{+} \Omega B \simeq \THH(\Sigma^{\infty}_{+} \Omega B; M).
\]
A careful reading of the proof of \cite[Thm. 5.2]{PS12} shows that this shadow functor and its cyclic symmetry isomorphism $\theta$ agree with the ones we constructed in Example \ref{ex:THH_as_shadow}.  The equivalence  of $\Top_*^{\mathrm{conn}}$-indexed symmetric monoidal categories from Prop. \ref{prop:companion} now implies:

\begin{proposition}\label{prop:equiv_of_shadows}
The derived fiber functors
\[
F_{(a, b)} \colon \Ex(A, B) \arr \Bimod_{S}(\Sigma^{\infty}_{+} \Omega A, \Sigma^{\infty}_{+} \Omega B)
\]
induce an equivalence of bicategories $\Ex_*^{\mathrm{conn}} \simeq \Bimod_{S}^{\gp}$. There is a canonical natural equivalence of shadow functors
\begin{equation}\label{eq:equiv_of_shadows}
r_{!} \Delta_{B}^* X \simeq \THH(\Sigma^{\infty}_{+} \Omega B; F_{(b,b)} X)
\end{equation}
commuting with the cyclic isomorphisms $\theta$.
\end{proposition}

\begin{remark}
In the case $X = U_B$, we recover the well-known equivalence 
\[
\Sigma^{\infty}_{+} LB \simeq \THH(\Sigma^{\infty}_{+} \Omega B).
\]
\end{remark}

Proposition \ref{prop:equiv_of_shadows} allows us to compare shadows. In particular, if $(X, Y)$ is a dual pair in $\Ex_*^{\mathrm{conn}}$, then under the equivalence \eqref{eq:equiv_of_shadows}, the trace of $\id_X$ in $\Ex$ agrees with the trace of $\id_{F_{(a, b)} X}$ in $\Bimod_{S}$. Let us take $X = S_{f} \in \Ex(B, E)$ for a perfect fibration $f \colon E \arr B$ of path-connected spaces. Pick basepoints $b \in B$ and $e \in E$ such that $f(e) = b$. Recall that $S_{f} = \Sigma^{\infty}_{B \times E} (E, (f, \id))_{+}$ and observe that the homotopy fiber at $(b, e)$ is $\Sigma^{\infty}_{+} \Omega B$, with the concatenation action of $\Sigma^{\infty}_{+} \Omega B$ on the left and the action induced by $\Sigma^{\infty}_{+} \Omega f \colon \Sigma^{\infty}_{+} \Omega E \arr \Sigma^{\infty}_{+} \Omega B$ on the right.  Since $f$ is a perfect fibration, $\Sigma^{\infty}_{+} \Omega B$ is perfect as a $\Sigma^{\infty}_{+} \Omega E$-module, hence right-dualizable as a $(\Sigma^{\infty}_{+} \Omega B, \Sigma^{\infty}_{+} \Omega E)$-bimodule.  

\begin{corollary}\label{cor:Reid_trace_is_bimodule_trace}
The Reidemeister trace $\tr(S_{f})$ agrees with the trace
\[
\Sigma^{\infty}_{+} LB \simeq \THH(\Sigma^{\infty}_{+} \Omega B) \xarr{\tr(\id_{\Sigma^{\infty}_{+} \Omega B})} \THH(\Sigma^{\infty}_{+} \Omega E) \simeq \Sigma^{\infty}_{+} LE
\]
of the identity map of the $(\Sigma^{\infty}_{+} \Omega B, \Sigma^{\infty}_{+} \Omega E)$-bimodule $\Sigma^{\infty}_{+} \Omega B$.
\end{corollary}

\subsection{The free loop transfer as a bicategorical trace}

We spend the rest of the section comparing the bicategorical trace $\tr(\id_{\Sigma^{\infty}_{+} \Omega B})$ to the free loop transfer $\tau_\THH$, working entirely in the bicategory $\Bimod_{S}$. At this point, it is easiest to proceed in greater generality.

\begin{proposition}\label{prop:bicat_trace_is_THH_map} Suppose that $M$ is an $(A, R)$-bimodule and that $M$ is perfect as an $R$-module.  Then the map $\THH(\lambda_{M}) \colon \THH(A) \arr \THH(R)$ induced by
\[ \lambda_{M} = (-) \sma_{A} M \colon \Perf_{A} \arr \Perf_{R} \]
is homotopic to the trace of the identity of $M$ in the bicategory $\Bimod_{S}$.
\end{proposition}
\noindent

Along with Corollary \ref{cor:Reid_trace_is_bimodule_trace}, this proves Theorem \ref{thm:reid_trace_is_THH_transfer} in the case where $E$ and $B$ are connected.  But the general case follows because both the Reidemeister trace and the free loop transfer are additive over the components of $E$ (see Remark \ref{rem:additive}).


\begin{remark}
Proposition \ref{prop:bicat_trace_is_THH_map} shows that \emph{every} map $\THH(A) \arr \THH(R)$ arising from a functor of module categories of the form $- \sma_A M$ is given directly on the cyclic nerves by the bicategorical trace of $M$. By the variant of Eilenberg-Watts found in \cite{bgt}*{Cor 3.3}, we have therefore given a small, computable model for every map on $\THH$ induced from an exact functor of stable $\infty$-categories $\Perf_A \arr \Perf_R$.
\end{remark}

\begin{proof}[Proof of Proposition \ref{prop:bicat_trace_is_THH_map}]
We observe that the statement of the proposition is invariant under weak equivalence in $A$, $R$, and $M$. For $\lambda_M$ this follows from Proposition \ref{prop:lambda_m_invariant}, while for the bicategorical trace it follows from the naturality and homotopy-invariance of our choice of shadow functor and cyclic isomorphism $\theta$. Note that for $\lambda_M$ to be derived, $M$ has to be made cofibrant as an $(A,R)$-bimodule, whereas for the trace to be derived we need the maps $S \ra A$, $S \ra R$, $* \ra M$, and $* \ra DM \simeq F_R(M,R)$ to be cofibrations of spectra. (We are using $S$-modules throughout this proof.)

Without loss of generality, we assume that $A$ and $R$ are cofibrant $S$-algebras, and that $M$ is a cofibrant $(A, R)$-bimodule. We let $\cP_{A} = \Perf_{A}$ and $\cP_R = \Perf_R$ denote the categories of perfect cofibrant modules as in \S\ref{subsec:Perf}, and we write $\cP_{A}^c$ and $\cP_{R}^c$ for cofibrant approximations of $\cP_{A}$ and $\cP_{R}$ as spectral categories. We let $A^c = F_A(A,A)^c$, $R^c = F_R(R,R)^c$, $M^c = F_R(R,M)^c$, and $DM^c = F_R(M,R)^c$. These four spectra are appropriately cofibrant, so the bicategorical trace $\tr(\id_{M})$ is given by the zig-zag
\begin{equation}\label{eq:new_trM}
\xymatrix @C=-2em @R=2em{
N^\cyc(A^c; A^c) \ar@{-->}[rr]^-{ \coev} \ar[rd]_-{N^\cyc(1;\lambda_M) \hspace{2pc}} & & N^{\cyc}(A^c; B(M^c, R^c, DM^c)) \ar[dd]^{\theta} \ar[ld]_-\sim \\
& N^\cyc(A^c; F_R(M,M)^c) & \\
N^\cyc(R^c; R^c) & & \ar[ll]_-{\eval} N^{\cyc}(R^c; B(DM^c, A^c, M^c)) \ar[dl]^-{\hspace{5pc}N^\cyc(1;B(1,\lambda_M,1)) \hspace{3em}} \\
& N^{\cyc}(R^c; B(DM^c, F_R(M,M)^c, M^c)) \ar[ul]_-\sim & \\
}
\end{equation}
The map labeled coev is indeed a coevaluation map, because along our cofibrant replacements, it agrees with a map $A \arr M \sma_R^L DM$ for which the composite
\[ \xymatrix{ A \ar[r] & M \sma_{R}^L DM \ar[r]^-\sim & F_R(M,M) } \]
is the $A$ action on $M$.

Recall that the inclusions of spectral categories $A^c \arr \cP_A^c$ and $R^c \arr \cP_{R}^c$ induce weak equivalences of cyclic nerves by Lemma \ref{lem:BM}. Along these equivalences, the zig-zag \eqref{eq:new_trM} agrees with the zig-zag along the right-most route through the following diagram:
\[
\xymatrix{
N^{\cyc}(\cP_{A}^c; \cP_{A}^c) \ar[d]_{N(1; \lambda_{M})} & N^{\cyc}(\cP_{A}^c; B(\cP_{R}^c, \cP_{R}^c, \cP_{R}^c)) \ar[dl]_{\simeq}  \ar[d]^{N(\lambda_{M}; 1)} \ar[r]^{\theta} & N^{\cyc}(\cP_{R}^c; B(\cP_{R}^c, \cP_{A}^c, \cP_{R}^c)) \ar[d]^{N(1; \lambda_{M})} \\
N^{\cyc}(\cP_{A}^c; \cP_{R}^c) \ar[d]_{N(\lambda_{M}; 1)} & N^{\cyc}(\cP_{R}^c; B(\cP_{R}^c, \cP_{R}^c, \cP_{R}^c)) \ar[dl]_{\simeq} \ar[r]^{\theta} & N^{\cyc}(\cP_{R}^c; B(\cP_{R}^c, \cP_{R}^c, \cP_{R}^c))  \ar[dll]^{\simeq} & \\
N^{\cyc}(\cP_{R}^c; \cP_{R}^c) & & 
}
\]
Here the maps $\theta$ are rotation maps for the cyclic nerve of a two-sided bar construction defined in the same way as the cyclic isomorphism \eqref{eq:def_theta}. The unmarked equivalences are all augmentation maps for the bar construction.  The parallelogram commutes by naturality and the square involving $\theta$ commutes by inspection of the definition. The remaining triangle commutes in the homotopy category, by observing that the explicit homotopy in Lemma A.2 of \cite{lind_diagram_spaces} works on the cyclic bar construction as well.
The left vertical maps compose to give $\THH(\lambda_{M})$. This finishes the proof of Proposition \ref{prop:bicat_trace_is_THH_map}. 
\end{proof}

\section{The Becker-Gottlieb transfer in the fiberwise bicategory $\Ex_{B}$}\label{sec:fiberwise_CW_duality}

Now that the free loop transfer $\tau_\THH$ has been identified with the Reidemeister trace $\tr(S_{f})$, our next task is to express the Becker-Gottlieb transfer in terms that can be compared to the Reidemeister trace. In this section, we show how it fits into fiberwise bicategory $\Ex_B$ of spectra over spaces over $B$ from \cite{MS}*{19.2}.

\subsection{The fiberwise bicategory $\Ex_B$}
Fix a base space $B$. The bicategory $\Ex_{B}$ is defined as follows. The 0-cells are spaces $A$ equipped with a map into $B$. Given a pair $C, D$ of spaces over $B$, the category $\Ex_{B}(C, D)$ is the homotopy category of parametrized spectra over the pullback $C \times_{B} D$. The horizontal composition is defined by
\begin{gather*}
- \boxdot_{C} - \colon \Ex_{B}(A, C) \times \Ex_{B}(C, D) \arr \Ex_{B}(A, D) \\
X \boxdot_{C} Y = \pi^{C}_{!} \Delta_{C}^* (X \osma Y),
\end{gather*}
where $X \osma Y$ is the external smash product spectrum over $(A \times_{B} C) \times (C \times_{B} D)$ and 
\begin{align*}
\Delta_{C} \colon A \times_{B} C \times_{B} D &\arr (A \times_{B} C) \times (C \times_{B} D) \\
\pi^{C} \colon A \times_{B} C \times_{B} D  &\arr A \times_{B} D
\end{align*}
are the maps induced by the diagonal of $C$ and the projection $C \arr B$. We abbreviate $\boxdot = \boxdot_{C}$ when possible to keep the notation less cluttered. Notice that our spaces over $B$ are not required to be fibrant, so that the pullbacks of spaces may not be homotopy pullbacks, but all of the operations on spectra over these spaces are required to be derived.

The identity 1-morphism in $\Ex_{B}(A, A)$ is the spectrum
\[
\_{U}_{A} = \delta_{!} S_{A} \cong \Sigma^{\infty}_{A \times_{B} A} (A, (\id, \id))_{+}  \quad \text{over $A \times_{B} A$,}
\] 
where $\delta \colon A \arr A \times_{B} A$ is the fiberwise diagonal map. Since $\delta \colon B \arr B \times_{B} B$ is an isomorphism, we may identify $\_{U}_{B}$ with the sphere spectrum $S_{B}$ over $B$. In the special case of $A = B = C = D$, the horizontal composition operator $-\boxdot_{B}-$ is just the fiberwise smash product $- \sma_{B} -$ in $\Sp{}{B}$.

\begin{definition}
The objects $X \in \Ex_{B}(A, C)$ and $Y \in \Ex_{B}(C, A)$ form a \emph{fiberwise Costenoble-Waner dual pair over $B$} if $(X, Y)$ is a dual pair in the bicategory $\Ex_{B}$. So there are coevaluation and evaluation maps
\[
\coev(X) \colon \_{U}_{A} \arr X \boxdot_{C} Y \qquad \eval(X) \colon Y \boxdot_{A} X \arr \_{U}_{C}
\]
satisfying two triangle identities (Definition \ref{def:dual_pair_bicat}). As before, in this case we say that $X$ is right dualizable with dual $Y$. Since $Y$ is unique up to canonical isomorphism in $\Ex_{B}(C, A)$, we will often write $DX$ for the right dual of $X$.
\end{definition}

When $B = \ast$, the bicategory $\Ex_{*}$ coincides with $\Ex$. However, we will distinguish this from the case of $B \neq *$ by writing $\odot$ for the composition product in $\Ex$ and $\boxdot$ for the composition product in $\Ex_B$. \footnote{The reader is warned that May-Sigurdsson write $\odot_{B}$ for the product $\boxdot$ in $\Ex_{B}$.  We instead let the subscript denote the 0-cell over which the composition is taken, in analogy with the tensor products of bimodules over a ring.}

Following \cite{MS}*{\S 19.3}, there is a morphism of bicategories $\iota_{!} \colon \Ex_{B} \arr \Ex$ that takes each 0-cell $A \arr B$ to the space $A$, and pushes forward each spectrum $X$ over $A \times_{B} C$ to the spectrum $\iota_{!} X$ over $A \times C$, where
\[
\iota = (\pi_1, \pi_2) \colon A \times_{B} C \arr A \times C.
\]
We will study the interaction of $\iota_{!}$ with $\boxdot$, dual pairs, and shadows in \S\ref{sec:IOTA}.



For each map of spaces $f \colon C \arr A$ over $B$, there is a base-change object $\_{S}_{f}$ in $\Ex_{B}(A, C)$ defined by
\[
\_{S}_{f} = \Sigma^{\infty}_{A \times_{B} C} (C, (f, \id))_{+} \cong (f, \id)_{!} S_{C}.
\]
For example, $\_{U}_{A} = \_{S}_{\id_{A}}$ is the unit object for the composition product $\boxdot$.  We write ${}_{f}\_{S} = \Sigma^{\infty}_{C \times_{B} A} (C, (\id, f))_{+}$ for the base change spectrum $\_{S}_{f}$, considered as a spectrum parametrized over $C \times_{B} A$.  

In the absolute case when $B = \ast$, we recover the definition of the base-change spectra $S_{f}$ (Example \ref{ex:Ex}).  The underline on $\_{S}_{f}$ is meant to remind that we are working relative to the base $B$.  The comparison functor $\iota_{!} \colon \Ex_{B} \arr \Ex$ takes $\_{S}_{f}$ to $S_{f}$.

\begin{example}\label{ex:fiberwise_base_change_duality}
For any map of spaces $f \colon C \arr A$ over $B$, the base change spectra $\_{S}_f$ and ${}_{f}\_{S}$ form a fiberwise Costenoble-Waner dual pair $({}_{f}\_{S}, \_{S}_{f})$ in $\Ex_{B}$ \cite{MS}*{17.3.1}.  We will be particularly interested in a special case arising from a map $f \colon E \arr B$, regarded as a map $(E, f) \arr (B, \id)$ of spaces over $B$. Writing $\delta_{} \colon E \arr E \times_{B} E$ for the fiberwise diagonal, the coevaluation and evaluation maps
\begin{align*}
\coev({}_{f}\_{S}) \colon U_{E} \cong \Sigma^{\infty}_{E \times_{B} E} (E, \delta_{})_+ &\arr  \Sigma^{\infty}_{E \times_{B} E} (E \times_{B} E, \id)_{+} \cong {}_{f}\_{S} \boxdot_{B} \_{S}_{f} \\
\eval({}_{f}\_{S}) \colon \_{S}_{f} \boxdot_{E} {}_{f}\_{S} \cong  \Sigma^{\infty}_{B} (E, f)_{+} &\arr  \Sigma^{\infty}_{B} (B, \id)_{+} = U_{B}
\end{align*}
are the fiberwise stabilizations of $\delta_{}$ and $f$, respectively.
\end{example}

We next show how dualizability in $\Ex_B$ can be detected in $\Ex$.
\begin{lemma}\label{lemma:dualizable_in_exb_detected_on_fibers}
Suppose that $A \arr B$ and $C \arr B$ are fibrations. Then a 1-morphism $X \in \Ex_{B}(A, C)$ is right dualizable in $\Ex_{B}$ if and only if the derived fiber $X_{b} \in \Ex(A_b, C_b)$ is right dualizable in $\Ex$ for every $b \in B$.
\end{lemma}
\begin{proof}
By \cite{MS}*{19.3.6}, restriction to a single point $b \in B$ defines a map of bicategories $\Ex_B^{\textup{fib}} \arr \Ex$, where $\Ex_B^{\textup{fib}}$ consists of only those 0-cells $A$ for which the map into $B$ is a fibration. In particular, for each $X \in \Ex_{B}(A, C)$ and $Y \in \Ex_{B}(C, D)$ there is an isomorphism in $\Ex(A_b,D_b)$
\begin{equation}\label{eq:fiber_circle_product}
(X \boxdot_C Y)_b \cong X_b \odot_{C_b} Y_b
\end{equation}
Next, recall that $\Ex_{B}$ and $\Ex$ are closed bicategories \cite{MS}*{\S 16}.  This means that there are left and right internal hom objects $\llbracket-,-\rrbracket^{l}$ and $\llbracket-,-\rrbracket^{r}$ in $\Ex_{B}$ satisfying the adjunctions
\[
\Ex_{B}(Y, \llbracket X, Z \rrbracket^{l}) \cong \Ex_{B}(X \boxdot Y, Z) \cong \Ex_{B}(X, \llbracket Y, Z\rrbracket^{r}).
\]
for every $Z \in \Ex_{B}(A, D)$. Similarly, there are left and right internal hom objects $[-,-]^{l}$ and $[-,-]^{r}$ in $\Ex$ satisfying the analogous adjunctions with respect to the composition operator $\odot$. Setting $Y = \llbracket X, Z \rrbracket^l$ or $X =  \llbracket Y, Z \rrbracket^r$ in these adjunctions and using the isomorphisms (\ref{eq:fiber_circle_product}) gives maps
\begin{equation}\label{eq:fiber_internal_hom}
(\llbracket X, Z \rrbracket^l)_b \arr [ X_b, Z_b ]^l, \qquad (\llbracket Y, Z \rrbracket^r)_b \arr [ Y_b, Z_b ]^r
\end{equation}
Using the definitions of the internal hom objects in \cite{MS}*{17.1.4} and \cite{MS}*{19.2.8}, and the adjoint of the Beck-Chevalley isomorphism for pullbacks and their derived right adjoints $f_*$, the maps (\ref{eq:fiber_internal_hom}) are equivalences when $A$ and $D$ are fibrations over $B$.

Now consider the map
\[
\mu_{X} \colon X \boxdot_{C} \llbracket X, U_{C} \rrbracket^{r} \arr \llbracket X, X \rrbracket^{r}
\]
of spectra over $A \times_{B} A$ that is adjoint to the composite
\begin{equation}\label{eq:adjoint_lemma63}
X \boxdot_{C} \llbracket X, U_{C} \rrbracket^{r} \boxdot_{A} X \xarr{1 \boxdot \epsilon} X \boxdot_{C} U_{C} \cong X
\end{equation}
of the counit for the adjunction defining $\llbracket-,-\rrbracket^{r}$ and the unit isomorphism in $\Ex_{B}$.  As in the theory of duality in symmetric monoidal categories, the internal hom object $\llbracket X, U_{C} \rrbracket^{r}$ provides a canonical choice for the right dual of $X$ when $X$ is right dualizable.  In fact, $X$ is right dualizable if and only if the map $\mu$ is an equivalence \cite{MS}*{16.4.12}.  Similarly, the fiber $X_{b} \in \Ex(A_b, C_b)$ is right dualizable in $\Ex$ if and only if the analogous map
\[
\mu_{X_b} \colon X_{b} \odot_{C_b} [X_b, U_{C_{b}}]^{r} \arr [X_b, X_b]^{r}
\]
of spectra over $A_b \times A_b$ is an equivalence. This analogous map is adjoint to the fiber of \eqref{eq:adjoint_lemma63} along (\ref{eq:fiber_circle_product}) and (\ref{eq:fiber_internal_hom}), so it is identified with the fiber of the map $\mu_{X}$. Therefore $\mu_X$ is an equivalence if and only if $\mu_{X_b}$ is an equivalence for every $b \in B$.
\end{proof}

\begin{proposition}\label{prop:base_change_dualble}
Suppose that $f \colon E \arr B$ is a perfect fibration.  Then the base change spectrum $\_{S}_{f}$ is right dualizable and there is a dual pair $(\_{S}_{f}, D \_{S}_{f})$ in $\Ex_{B}$.
\end{proposition}
\begin{proof}
This follows from the lemma since the fiber of $\_{S}_{f} \in \Ex_{B}(B, E)$ over $b \in B$ is the sphere spectrum $S_{E_{b}} = \Sigma^{\infty}_{E_b} (E_b, \id)_+$ over $E_b$, which is right dualizable (Example \ref{ex:CWduality_of_M}).  
\end{proof}

\begin{remark}
We write $D \_{S}_{f}$ for the right dual of $\_{S}_{f}$.  Notice that the right dual $D \_{S}_{f}$ is generally \emph{not} the base-change spectrum ${}_{f}\_{S}$.  In fact, this will only happen when the homotopy fibers of $f$ are finite homotopy 0-types, i.e. equivalent to a finite discrete space.  It is important that we carefully distinguish the dual pairs $({}_{f}\_{S}, \_{S}_{f})$ and $(\_{S}_{f}, D \_{S}_{f})$.

\end{remark}

\subsection{The Becker-Gottlieb transfer via Costenoble-Waner duality}

Let us now turn to the Becker-Gottlieb transfer associated to a perfect fibration $f \colon E \arr B$. Recall from Example \ref{ex:becker_gottlieb_transfer} that the spectrum $f_{!} S_{E}$ is dualizable in $\ho \Sp{}{B}$, and that the trace of the stabilized fiberwise diagonal $\delta_{} \colon E \arr E \times_{B} E$ defines the pretransfer
\[
\tau_{B} = \tau_{B}(f) \colon S_{B} \oarr{\coev} f_{!} S_{E} \sma_{B} D f_{!} S_{E} \xarr{ (\id \sma_{B} \delta_{}) \circ \gamma} D f_{!} S_{E} \sma_{B} f_{!} S_{E} \sma_{B} f_{!} S_{E} \oarr{\eval} f_{!} S_{E}.
\]
Applying the base change functor $r_{!}$ gives the Becker-Gottlieb transfer
\[
\tau = r_{!} \tau_{B} \colon \Sigma^{\infty}_{+} B \cong r_{!} S_{B} \arr r_{!} f_{!} S_{E} \cong \Sigma^{\infty}_{+} E.
\]
In this section we decompose the parametrized spectrum $f_{!} S_{E}$ using the $\boxdot$ product in $\Ex_{B}$.  This will allow for a description of the pretransfer that is not available when working only with fiberwise duality in $\ho \Sp{}{B}$.

\begin{lemma} Considering the spectrum $f_{!} S_{E}$ as a one-cell in $\Ex_{B}(B, B)$, there is a canonical isomorphism $f_{!} S_{E} \cong \_{S}_{f} \boxdot_{E} {}_{f}\_{S}$.  Under this identification, the fiberwise diagonal $\delta_{} \colon f_{!} S_{E} \arr f_{!} S_{E} \sma_{B} f_{!} S_{E}$ coincides with the map
\[
1 \boxdot \coev({}_{f}\_{S}) \boxdot 1 \colon \_{S}_{f} \boxdot_{E} {}_{f} \_{S} \arr \_{S}_{f} \boxdot_{E} ({}_{f} \_{S} \boxdot_{B} \_{S}_{f}) \boxdot_{E} {}_{f}\_{S}.
\]
\end{lemma}

\begin{proof}
Both $f_{!} S_{E}$ and $\_{S}_{f} \boxdot_{E} {}_{f}\_{S}$ are identified with the fiberwise suspension spectrum of $E$ over $B$. From Example \ref{ex:fiberwise_base_change_duality}, the coevaluation map $\coev({}_{f}\_{S})$ is the stabilization of the fiberwise diagonal $\delta: E \arr E \times_B E$ of spaces over $E \times_B E$. The tensoring on both sides by base-change spectra pushes the base forward from $E \times_B E$ to $B$, leaving us with $\delta$ considered as a map of spaces over $B$.
\end{proof}

The decomposition $f_{!} S_{E} \cong \_{S}_{f} \boxdot_{E} {}_{f}\_{S}$ gives another proof that $f_{!}S_{E}$ is fiberwise dualizable: each of $\_{S}_{f}$ and ${}_{f}\_{S}$ is right CW dualizable, and so their $\boxdot$ product is right CW dualizable.  By the uniqueness of duals, this gives a canonical isomorphism $D f_{!} S_{E} \cong \_{S}_{f} \boxdot_{E} D \_{S}_{f}$.  The coevaluation map $\coev( f_{!} S_{E})$ may be identified with the composite
\[
U_{B} \xarr{\coev(\_{S}_{f})} \_{S}_{f} \boxdot_E D \_{S}_{f} \xarr{ 1 \boxdot \coev({}_{f}\_{S}) \boxdot 1} \_{S}_{f} \boxdot_E {}_{f}\_{S} \boxdot_B \_{S}_{f} \boxdot_E D \_{S}_{f},
\]
and similarly for the evaluation maps.  This allows us to rewrite the pretransfer entirely in terms of data coming from fiberwise Costenoble-Waner duality:
\[
\xymatrix{
S_{B} = \_{U}_{B} \ar[d]^{\coev(\_{S}_{f})} \\
\_{S}_{f} \boxdot D \_{S}_{f} \ar[d]^{1 \boxdot \coev({}_{f}\_{S}) \boxdot 1} \\
\_{S}_{f} \boxdot {}_{f}\_{S} \boxdot \_{S}_{f} \boxdot D\_{S}_{f} \ar[d]^{\gamma} \\
\_{S}_{f} \boxdot D \_{S}_{f} \boxdot \_{S}_{f} \boxdot {}_{f}\_{S} \ar[d]^{1 \boxdot 1 \boxdot 1 \boxdot \coev({}_{f}\_{S}) \boxdot 1} \\
\_{S}_{f} \boxdot D \_{S}_{f} \boxdot \_{S}_{f} \boxdot {}_{f} \_{S} \boxdot \_{S}_{f} \boxdot {}_{f}\_{S} \ar[d]^{1 \boxdot \eval(\_{S}_{f}) \boxdot 1 \boxdot 1 \boxdot 1} \\
\_{S}_{f} \boxdot {}_{f}\_{S} \boxdot \_{S}_{f} \boxdot {}_{f}\_{S} \ar[d]^{\eval({}_{f}\_{S}) \boxdot 1 \boxdot 1} \\
\_{S}_{f} \boxdot {}_{f}\_{S}
}
\]
The map $\gamma$ is the symmetry isomorphism of the fiberwise smash product $\sma_B = \boxdot_B$. We may take this diagram and switch the order of $\coev({}_{f}\_{S})$ and $\eval(\_{S}_{f})$ because they involve different summands.  Applying one of the triangle identities for the dual pair $({}_{f}\_{S}, \_{S}_{f})$, we simplify and deduce:

\begin{proposition}\label{prop:pretransfer_via_duality}
After identifying $f_{!} S_{E}$ with $\_{S}_{f} \boxdot_{E} {}_{f}\_{S}$, the pretransfer $\tau_{B}(f)$ is canonically homotopic to the following composite
\[
\xymatrix{
S_{B} = \_{U}_{B} \ar[d]^{\coev(\_{S}_{f})} \\
\_{S}_{f} \boxdot D \_{S}_{f} \ar[d]^{1 \boxdot \coev({}_{f}\_{S}) \boxdot 1} \\
\_{S}_{f} \boxdot {}_{f}\_{S} \boxdot \_{S}_{f} \boxdot D\_{S}_{f} \ar[d]^{\gamma} \\
\_{S}_{f} \boxdot D \_{S}_{f} \boxdot \_{S}_{f} \boxdot {}_{f}\_{S} \ar[d]^{1 \boxdot \eval(\_{S}_{f}) \boxdot 1} \\
\_{S}_{f} \boxdot {}_{f}\_{S}
}
\]
\end{proposition}

\begin{remark}
2-cells in a bicategory, such as $\Ex_{B}$, may be described using the calculus of colored string diagrams, where the colors that label two dimensional regions correspond to 0-cells of the bicategory and strings correspond to 1-cells \cite{PS13}.  Using this notation, Proposition \ref{prop:pretransfer_via_duality} asserts that the Becker-Gottlieb pretransfer is represented by Figure \ref{fig:pretransfer2}.
\begin{figure}
\begin{tikzpicture}[scale=0.8]
    \begin{pgfonlayer}{background}
      \fill[bluefill] (2.5,-7) rectangle (-2.5,.5);
       \fill[orangefill] (.5,-1) to[out=270,in=180]  (1,-2) 
       to[out=0,in=90]  (1.5,-3)
       to (.5, -4) to (0,-3.5) to (.5,-3)
       to[out=90,in=0] node [ed,near end] {} (0,-2)
       to[out=180,in=90] node [ed,near end] {} (-.5,-3)
       to (0,-3.5) to (-.5,-4) to (-1.5,-3)
       to[out=90, in=180] (-1,-2)
       to[out=0, in=270] (-.5,-1)
       to[out=90,in=180] node [ed,near start] {} (0,0)
       to[out=0,in=90] node [ed,near end] {} (.5,-1) ;
       \fill[orangefill] (0,-3.5) to (.5, -4) to (0, -4.5) to (-.5, -4); 
       \fill[orangefill] (-.5,-7) to[out=90, in=0] (-1,-6) 
       to[out=180, in=270] (-1.5,-5) to (-.5,-4) to (0,-4.5) to (-.5,-5)
       to[out=270, in=180] (0,-6) to[out=0, in=270] (.5, -5)
       to (0, -4.5) to (.5, -4) to (1.5, -5) 
       to[out=270, in=0] (1,-6)
       to[out=180, in=90] (.5, -7);
      \end{pgfonlayer}
      \draw (-.5,-1) to[out=90,in=180] node [ed,near start] {$\_{S}_{f}$} (0,0)
      to[out=0,in=90] node [ed,near end] {$D\_{S}_{f}$} (.5,-1) ;
      \draw (.5,-1) to[out=270,in=180] node [ed,near start] {} (1,-2)
      to[out=0,in=90] node [ed,near end] {} (1.5,-3);
      \draw (-.5,-1) to[out=270,in=0] node [ed,near start] {} (-1,-2)
      to[out=180,in=90] node [ed,near end] {} (-1.5,-3);
      \draw (-.5,-3) to[out=90,in=180] node [ed,near start] {${}_{f}\_{S}$} (0,-2)
      to[out=0,in=90] node [ed,near end] {$\_{S}_{f}$} (.5,-3);
      \draw (.5,-3) to (0,-3.5);
      \draw [dashed] (0,-3.5) to (-.5,-4);
      \draw (-.5,-4) to (-1.5, -5);
      \draw (1.5,-3) -- (.5,-4);
      \draw [dashed] (.5,-4) to (0,-4.5);
      \draw (0,-4.5) to (-.5,-5);
      \draw (-1.5,-3) -- (.5,-5);
      \draw (-.5,-3) -- (1.5,-5);
      \draw (-.5,-5) to[out=270,in=180] node [ed,near start,swap] {$D\_{S}_{f}$} (0,-6)
      to[out=0,in=270] node [ed,near end,swap] {$\_{S}_{f}$} (.5,-5);
      \draw (1.5,-5) to[out=270,in=0] node [ed,near start] {} (1,-6)
      to[out=180,in=90] node [ed] {${}_{f}\_{S}$} (.5,-7);
      \draw (-1.5,-5) to[out=270,in=180] node [ed,near start,swap] {} (-1,-6)
      to[out=0,in=90] node [ed,swap] {$\_{S}_{f}$} (-.5,-7);
      \node[anchor=north west,blue] at (-2,.25) {$B$};
      \node[anchor=north east,blue] at (2,.25) {$B$};
      \node[anchor=north, orange!50!black] at (0,-1) {$E$};
      \node[anchor=south, orange!50!black] at (0,-6.75) {$E$};
      \node[anchor=south, blue] at (0,-5.5) {$B$};
      \node[anchor=north, blue] at (0, -2.5) {$B$};
\end{tikzpicture}
\caption{The Becker-Gottlieb pretransfer $\tau_{B}(f) \colon S_{B} \arr f_{!}S_{E}$ as a 2-cell in the bicategory $\Ex_{B}$.}\label{fig:pretransfer2}
\end{figure}
Note that the crossing in the middle of the figure is the symmetry isomorphism $\gamma$ in the symmetric monoidal category $\Ex_{B}(B, B) = \Sp{}{B}$.  It cannot be decomposed into multiple crossings of individual strands, but rather must be taken as a single crossing of the two bands labeled by the 0-cell $(E \arr B)$.  
\end{remark}

\section{The Becker-Gottlieb transfer and the Reidemeister trace}\label{sec:comparison}

Proposition \ref{prop:pretransfer_via_duality} gives a description of the Becker-Gottlieb pretransfer as a 2-cell in the bicategory $\Ex_B$. In this section we discuss how its pushforward to $\Ex$ compares with the Reidemeister trace. We finish by proving Theorem \ref{thm:intro_becker_gottlieb_compatibility}.

\subsection{The comparison of $\Ex_B$ with $\Ex$ on shadows}\label{sec:IOTA}

We now study how the oplax morphism of bicategories
\[ \iota_{!} \colon \Ex_B \arr \Ex \]
interacts with shadows. Recall that $\iota_{!}$ takes each space $A \arr B$ over $B$ to the underlying space $A$, and each spectrum $X$ over $A \times_{B} C$ to the spectrum $\iota_{!} X := (\pi_1 \times \pi_2)_{!} X$ over $A \times C$.  The term oplax refers to the natural transformation
\[ \psi \colon \iota_{!}(X \boxdot Y) \arr \iota_{!} X \odot \iota_{!} Y \]
relating the composition products in the bicategories $\Ex_{B}$ and $\Ex$. $\psi$ is induced by the commutative diagram of spaces with spectra over them
\[
\xymatrix{
X \boxdot_C Y \ar@{--}[dr] & & & & \iota_{!} X \odot_C \iota_{!} Y \ar@{--}[dl] \\
& A \times_B D \ar[rr]^{\iota} & & A \times D \\
X \osma Y \ar@{--}[dr] & A \times_B C \times_B D \ar[rr] \ar[d]_{\Delta} \ar@{}[drr]|{ \Downarrow \psi}  \ar[u]_-{\pi_{13}} & & A \times C \times D \ar[u]_{1 \times r \times 1} \ar[d]^{1 \times \Delta \times 1} & \iota_{!} X \osma \iota_{!} Y \ar@{--}[dl] \\
& A \times_B C \times C \times_B D \ar[rr]^-{\iota \times \iota} & & A \times C \times C \times D
}
\]
as in \eqref{eq:beck_chevalley}. It fails to be an isomorphism because the bottom square is not in general a homotopy pullback. However, we have the following special case:

\begin{lemma}\label{lem:psi_equiv}  Let $p \colon A \arr B$ be a space over $B$, and suppose that $X$ and $Y$ are spectra over $A$, considered as 1-cells $X \in \Ex_{B}(B, A)$ and $Y \in \Ex_{B}(A, B)$.  Then the natural transformation 
\[
\psi \colon \iota_{!}(X \boxdot_{A} Y) \arr \iota_{!} X \odot_{A} \iota_{!} Y
\]
is a stable equivalence of spectra over $B \times B$.
\end{lemma}

\begin{proof} We subdivide the bottom square from the diagram just above:
\[
\xymatrix{
A  \ar@{}[dr]|{\Downarrow \rho} \ar[r]^-{g} \ar[d]_{(g, 1)} & B \times A \times B \ar[d]^{\Delta_{B \times A \times B}} \\
B \times A \times B \times A  \ar@{}[dr]|{\Downarrow \beta} \ar[r]^-{1 \times g} \ar[d]_{\pi_{24}} & B \times A \times B \times B \times A \times B \ar[d]^{\pi_{1256}} \\
A \times A \ar[r]^-{(p, 1) \times (1, p)} & B \times A \times A \times B
}.
\]
Here $g \colon A \arr B \times A \times B$ is the map $g(a) = (p(a), a, p(a))$.
The map $\psi$ is given by the composite $\beta \circ \rho$ as defined in \eqref{eq:beck_chevalley}.  Given a spectrum $Z$ over $B \times A \times B$ and a spectrum $Y$ over $A$, the map $\rho$ induces the equivalence in the projection formula:
\[
\rho \colon g_{!} (g^* Z \sma_{A} Y) \oarr{\simeq} Z \sma_{B \times A \times B} g_{!} Y.
\]
The natural transformation $\beta$ is always an equivalence because the lower square is a homotopy pullback square.  Together, this implies that $\psi$ is an equivalence.
\end{proof}

\noindent Since $\iota_{!} \_{S}_{f} \cong S_{f}$, the lemma implies that if $\_{S}_{f}$ is right dualizable in $\Ex_{B}$, then $S_{f}$ is right dualizable in $\Ex$.  Proposition \ref{prop:base_change_dualble} then gives the next result, which was also observed by Ponto-Shulman \cite{PS14}*{Prop. 4.7}.

\begin{proposition}\label{prop:base_change_dualble_Ex}
If $f \colon E \arr B$ is a perfect fibration, then the base change spectrum $S_{f}$ is right dualizable and there is a dual pair $(S_{f}, D S_{f})$ in $\Ex$.
\end{proposition}

Now we may examine how $\iota_!$ interacts with the shadow functor 
\[
\llan{-} \colon \Ex(A, A) \arr \ho \Sp{}{}
\]
from Example \ref{ex:shadow_Ex}, which takes a spectrum $X$ over $A \times A$ to the spectrum $r_{!} \Delta_{A}^* X$.
Notice that for $X \in \Ex_{B}(B, B)$, considered as a spectrum over $B$, the image of $X$ in $\Ex$ is the pushforward $\iota_{!} X = \Delta_{!} X$ along the diagonal.  After applying the shadow functor, we have a canonical identification $\llan{ \iota_{!} X} \cong r_{!} \Delta^* \Delta_{!} X$.  In particular, the unit of the adjunction $(\Delta_{!} ,\Delta^*)$ induces a natural transformation
\[ \eta \colon r_{!} X \arr \llan{ \iota_{!} X }. \]
When $X = S_{B}$ is the sphere spectrum over $B$, the map $\eta$ may be identified with the inclusion of constant loops $c \colon \Sigma^{\infty}_{+} B \arr \Sigma^{\infty}_{+} LB$.

If $X$ and $Y$ are spectra over $B$, considered as 1-cells in $\Ex_{B}$, then $X \boxdot_{B} Y = X \sma_{B} Y$ is just the fiberwise smash product over $B$. We write
\[
\gamma \colon X \boxdot_B Y \cong Y \boxdot_B X
\]
for the symmetry isomorphism coming from the symmetric monoidal structure of the fiberwise smash product.
On the other hand, $\iota_! X$ and $\iota_! Y$ are spectra over $B \times B$, and the cyclic isomorphism for the shadow functor is a natural map
\[
\theta \colon \llan{\iota_! X \odot_B \iota_! Y} \arr \llan{\iota_! Y \odot_B \iota_! X}
\]
One might expect that $\gamma$ and $\theta$ are compatible after taking shadows, in the sense that Figure \ref{magritte} commutes in the homotopy category.
\begin{center}
\begin{figure}[h]
\[
\xymatrix{
\llan{\iota_{!} (X \boxdot_{B} Y)} \ar[r]^{\psi} \ar[d]_{\llan{ \iota_{!} \gamma}} & \llan{ \iota_{!} X \odot_{B} \iota_{!} Y } \ar[d]^{\theta} \\
\llan{\iota_{!} (Y \boxdot_{B} X)} \ar[r]^{\psi} & \llan{ \iota_{!} Y \odot_{B} \iota_{!} X } 
}
\]
    \caption{This is not a commutative diagram.}
    \label{magritte}
    \vspace{-2em}
\end{figure}
\end{center}
But this is usually not the case. In fact, if $X$ and $Y$ are fiberwise suspension spectra of fibrations over $B$, then the failure of the diagram to commute measures the non-triviality of the monodromy of $X \sma_{B} Y$ around each free loop in $B$. However, after precomposing with the unit map $\eta$, which corresponds to the inclusion of constant loops, Figure \ref{magritte} does commute.


\begin{lemma}\label{lemma:DOES_commute}
When $X$ and $Y$ are 1-cells in $\Ex_B(B,B)$, the diagram
\begin{equation*}\label{eq:DOES_commute}
\xymatrix{
r_{!} (X \sma_{B} Y) \ar[r]^-{\eta} \ar[d]_{r_{!} (\gamma)}^-\cong & \llan{\iota_{!} (X \boxdot_{B} Y)} \ar[r]^{\psi}_-\cong & \llan{ \iota_{!} X \odot_{B} \iota_{!} Y } \ar[d]^{\theta}_-\cong \\
r_{!} (Y \sma_{B} X) \ar[r]^-{\eta} & \llan{\iota_{!} (Y \boxdot_{B} X)} \ar[r]^{\psi}_-\cong & \llan{ \iota_{!} Y \odot_{B} \iota_{!} X }
}
\end{equation*}
commutes.
\end{lemma}
\begin{proof}
As observed above, the map $\iota$ on the $0$-cell $B$ in $\Ex_B$ is simply the diagonal map $\Delta \colon B \arr B \times B$, so we freely use $\Delta$ in place of $\iota$ throughout this proof. The natural transformation
\[ \xymatrix{ \Delta_! (X \boxdot_{B} Y) \ar[r]^{\psi} & \Delta_! X \odot_{B} \Delta_{!} Y } \]
is induced by the Beck-Chevalley isomorphism in the diagram of spaces
\begin{equation*}\label{eq:Lshape_small}
\xymatrix{
B \ar@{}[dr]|{\Downarrow \psi} \ar[d]_{\Delta} \ar[r]^-{\Delta^3} & B \times B \times B \ar[d]^{1 \times \Delta \times 1} \ar[r]^-{\pi_{13}} & B \times B \\
B \times B \ar[r]^-{\Delta \times \Delta} & B \times B \times B \times B & 
} 
\end{equation*}
by starting in the lower left corner with the spectrum $X \osma Y$, then pushing forward and pulling up along two routes to the upper-right (cf. diagram before Lemma \ref{lem:psi_equiv}). We write $\pi_i$, $\pi_{ij}$ for the projection map to the $i$-th, or $i$-th and $j$-th factor of a product, respectively.

We may express the composite
\[ \xymatrix{ r_{!} (X \sma_{B} Y) \ar[r]^-{\eta} & \llan{\Delta_{!} (X \boxdot_{B} Y)} \ar[r]^{\psi} & \llan{ \Delta_{!} X \odot_{B} \Delta_{!} Y } } \]
in the same way with the larger diagram
\begin{equation}\label{eq:Lshape2}
\xymatrix{
B \ar@{=}[rr] \ar@{=}[d] \ar@{}[drr]|{\Downarrow \eta} & & B \ar[d]^{\Delta} \ar[r]^-r & {*} \\ 
B \ar[r]^-{\Delta^3} \ar@{}[dr]|{\Downarrow \psi} \ar[d]_{\Delta} & B \times B \times B \ar[r]^{\pi_{13}}  \ar[d]^{1 \times \Delta \times 1} & B \times B \\
B \times B \ar[r]^-{\Delta \times \Delta} & B \times B \times B \times B & 
}
\end{equation}
which may be subdivided as
\begin{equation}\label{eq:Lshape}
\xymatrix{
B \ar@{=}[d] \ar[r]^-{\Delta} \ar@{}[dr]|{\Downarrow \alpha} & B \times B \ar@{}[dr]|{ \Downarrow \phi} \ar[d]^{\Delta_{o}} \ar[r]^{\pi_{1}} & B \ar[d]^{\Delta} \ar[r]^-r & {*} \\
B \ar@{}[dr]|{\Downarrow \psi} \ar[d]_{\Delta} \ar[r]^-{\Delta^3} & B \times B \times B \ar[d]^{1 \times \Delta \times 1} \ar[r]^-{\pi_{13}} & B \times B \\
B \times B \ar[r]^-{\Delta \times \Delta} & B \times B \times B \times B & 
} 
\end{equation}
where $\Delta_{o} \colon B \times B \arr B \times B \times B$ is the map $(a, b) \mapsto (a, b, a)$. The coincidence of \eqref{eq:Lshape2} and \eqref{eq:Lshape} gives the commuting rectangle
\begin{equation}\label{eq:fromLdiagram}
\xymatrix{
r_{!} (X \sma_{B} Y) \ar@{=}[d] \ar[r]^-{\eta}  & \llan{\Delta_{!}(X \boxdot_{B} Y)} \ar[r]^-{\psi}_-\cong & \llan{\Delta_{!} X \odot_{B} \Delta_{!} Y} \\
r_{!} \Delta^*(X \osma Y) \ar[r]^-{\alpha} & r_{!} \Delta_{o}^* \Delta^3_{!} \Delta^*(X \osma Y) \ar[r]^-{\psi}_-\cong & r_{!}D^* (\Delta \times \Delta)_{!} (X \osma Y) \ar[u]_{\phi}^{\cong}
}
\end{equation}
where we abbreviate $D = (1 \times \Delta \times 1) \circ \Delta_{o}$. Note that the maps $\psi$ and $\phi$ are isomorphisms in the homotopy category because they come from homotopy pullback squares.

Next we examine the effect of the symmetry isomorphism $\gamma \colon X \sma_{B} Y \arr Y \sma_{B} X$ on the lower route in diagram \eqref{eq:fromLdiagram}. We think of $\gamma$ as constructed by applying $\Delta^*$ to the natural isomorphism $ \gamma \colon X \osma Y \arr \gamma^* (Y \osma X)$ of spectra over $B \times B$. The equality 
\[
\xymatrix{
B \ar@{=}[d] \ar[r]^-{\Delta} \ar@{}[dr]|{\Downarrow \alpha} & B \times B \ar[d]^{\Delta_o} \\
B \ar[d]_{\Delta} \ar[r]^-{\Delta^3} \ar@{}[dr]|{\Downarrow \psi} & B \times B \times B \ar[d]^{1 \times \Delta \times 1} \\
B \times B \ar[d]_{\gamma} \ar[r]^-{\Delta \times \Delta} & B \times B \times B \times B \ar[d]^{\gamma_{B \times B}} \\
B \times B \ar[r]^-{\Delta \times \Delta} & B \times B \times B \times B 
}
\quad \xymatrix{ &  \\ \ar@{}[dr]_{=} & \\ & } \quad
\xymatrix{
B \ar[r]^{\Delta} \ar@{=}[d] & B \times B \ar[d]^{\gamma} \\
B \ar@{=}[d] \ar[r]^-{\Delta} \ar@{}[dr]|{\Downarrow \alpha} & B \times B \ar[d]^{\Delta_o} \\
B \ar[d]_{\Delta} \ar[r]^-{\Delta^3} \ar@{}[dr]|{\Downarrow \psi} & B \times B \times B \ar[d]^{1 \times \Delta \times 1} \\
B \times B \ar[r]^-{\Delta \times \Delta} & B \times B \times B \times B 
}
\] 
implies that the following diagram of natural transformations commutes:
\[
\xymatrix{
\Delta_{!} \Delta^* \gamma^* \ar[d]_{\cong} \ar[r]^-{\psi \circ \alpha} & D^* (\Delta \times \Delta)_{!} \gamma^* \ar[r]^-{\cong} & D^* \gamma_{B \times B}^* (\Delta \times \Delta)_{!} \ar[d]^{\cong} \\
\Delta_{!} \Delta^* \ar[r]^-{\cong} &  \gamma^* \Delta_{!} \Delta^* \ar[r]^-{\psi \circ \alpha} & \gamma^* D^* (\Delta \times \Delta)_{!}
}
\]
We apply each of these functors to the external smash product $Y \osma X$ over $B \times B$ and take $r_{!}$ of the results. Suppressing the canonical isomorphisms $r_{!} \gamma^* \cong r_{!}$, this gives the lower rectangle in the diagram
\[
\xymatrix{
r_{!} \Delta^*(X \osma Y) \ar[d]^{\gamma}_-\cong \ar[r]^-{\psi \circ \alpha} & r_{!} D^* (\Delta \times \Delta)_{!} (X \osma Y) \ar[d]^{\gamma}_-\cong & \\
r_{!} \Delta^* \gamma^* (Y \osma X) \ar[r]^-{\psi \circ \alpha} \ar[d]_{\cong} & r_{!} D^* (\Delta \times \Delta)_{!} \gamma^* (Y \osma X)  \ar[r]^-\cong & r_{!} D^* \gamma_{B \times B}^* (\Delta \times \Delta)_{!} (Y \osma X) \ar[d]^{\cong} \\
r_{!} \Delta^* (Y \osma X)  \ar[rr]^{\psi \circ \alpha} & & r_{!} D^* (\Delta \times \Delta)_{!} (Y \osma X)
}
\]
The upper rectangle also commutes by the naturality of the canonical isomorphism $\gamma \colon X \osma Y \arr \gamma^* (Y \osma X)$, so the entire diagram commutes.  Paste a copy of diagram \eqref{eq:fromLdiagram} on top of this diagram, and another underneath with the roles of $X$ and $Y$ swapped.  The resulting vertical composite on the right, involving a chosen inverse to the equivalence $\phi$, is the definition of $\theta$ from \cite{PS12}*{Thm. 5.2}, and the commutative diagram we have constructed is the desired one.
\end{proof}

\subsection{The proof of Theorem \ref{thm:intro_becker_gottlieb_compatibility}}\label{sec:proof_main_thm}


Now, apply the natural transformations $\eta$ and $\psi$ to the Becker-Gottlieb transfer $\tau_{B}(f)$ as described in Proposition \ref{prop:pretransfer_via_duality}.  The result is the diagram
\[
\resizebox{\textwidth}{!}{
\xymatrix{
r_{!} \_{U}_{B} \ar[r]^-{\eta} \ar[d]^-{\coev(\_{S}_{f})} & \llan{\iota_{!} \_{U}_{B}} \ar[r]^{\psi} \ar[d]^{\coev(S_{f})} & \llan{U_{B}} \ar[d]^{\coev(S_{f})} \\
r_{!}(\_{S}_{f} \boxdot D \_{S}_{f}) \ar[r]^-{\eta} \ar[d]^-{\coev({}_{f}\_{S})} & \llan{\iota_{!}(\_{S}_{f} \boxdot D \_{S}_{f})} \ar[d]^{\coev({}_{f}S)} \ar[r]^{\psi} & \llan{S_{f} \odot D S_{f}} \ar[d]^{\coev({}_{f}S)} \\
r_{!}(\_{S}_{f} \boxdot {}_{f}\_{S} \boxdot \_{S}_{f} \boxdot D \_{S}_{f}) \ar[r]^-{\eta} \ar[d]^-{r_{!} ( \gamma )} & \llan{\iota_{!}(\_{S}_{f} \boxdot {}_{f}\_{S} \boxdot \_{S}_{f} \boxdot D \_{S}_{f})}  \ar[r]^-{\psi} & \llan{S_{f} \odot {}_{f}S \odot S_{f} \odot D S_{f}} \ar[d]^{\theta} \\
r_{!} (\_{S}_{f} \boxdot D \_{S}_{f} \boxdot \_{S}_{f} \boxdot {}_{f}\_{S}) \ar[r]^-{\eta} \ar[d]^-{\eval(\_{S}_{f})} & \llan{\iota_{!}(\_{S}_{f} \boxdot D \_{S}_{f} \boxdot \_{S}_{f} \boxdot {}_{f}\_{S})} \ar[d]^{\eval(S_{f})} \ar[r]^-{\psi} & \llan{S_{f} \odot D S_{f} \odot S_{f} \odot {}_{f}S}  \ar[d]^{\eval(S_{f})} \\
r_{!}(\_{S}_{f} \boxdot {}_{f}\_{S}) \ar[r]^-{\eta} & \llan{\iota_{!}(\_{S}_{f} \boxdot {}_{f}\_{S})} \ar[r]^{\psi} & \llan{S_{f} \odot {}_{f}S} 
}
}
\]
where some of the instances of $\psi$ are really iterated applications of $\psi$. The large rectangle commutes by Lemma \ref{lemma:DOES_commute}, and the remaining squares commute by naturality and the oplax structure of $\psi$.

The vertical composite on the left is the Becker-Gottlieb transfer. Along the top we get the inclusion of constant loops $B \arr LB$, and along the bottom we get a map described in Remark \ref{rmk:easy}. To compare the vertical composite on the right to the THH transfer, we observe that it is the rightmost route in the next commutative diagram.
\[
\xymatrix @C=4em{
\llan{U_{B}}  \ar[d]_-{\coev(S_{f})} & & & \\
\llan{S_{f} \odot DS_{f}} \ar[r]^-{1 \odot \,\coev({}_{f}S) \odot 1} \ar[d]_{\theta} & \llan{S_{f} \odot {}_{f}S \odot S_{f} \odot D S_{f} } \ar[d]_{\theta} \ar[dr]^{\theta} & \\ 
\llan{DS_{f} \odot S_{f}} \ar[d]_{\eval(S_{f})} \ar[r]^-{1 \odot 1 \odot \,\coev({}_{f}S)} & \llan{D S_{f} \odot S_{f} \odot {}_{f}S \odot S_{f}} \ar[d]^-{\eval(S_{f}) \odot 1 \odot 1} \ar[r]^-{\theta} & \llan{S_{f} \odot D S_{f} \odot S_{f} \odot {}_{f}S } \ar[d]^{1 \odot \,\eval(S_{f}) \odot 1} \\
\llan{U_{E}} \ar[r]^-{\coev({}_{f}S)} & \llan{{}_{f}S \odot S_{f} } \ar[r]^{\theta} & \llan{S_{f} \odot {}_{f}S}
}
\]
All of the squares are naturality squares, and the triangle commutes by one of the shadow axioms for $\theta$.  The left-hand route is the THH transfer followed by $\theta \circ \coev({}_{f}S)$.  We now give an explicit description of the coevaluation map.

The operator $ - \odot_{B} S_{f}$ encodes the pullback functor $(\id \times f)^*$.  Since we are working with derived functors, the pullback
\[
{}_{f}S \odot_{B} S_{f} \cong (\id \times f)^*(E, (\id, f))_+
\]
is given by the fiberwise suspension spectrum over $E \times E$ of the homotopy pullback
\[
\xymatrix{
E^{I} \times_{B} E \ar[r]^-{\pi_1} \ar[d]_{(e_0 \pi_1, \pi_2)} & E^{I} \ar[d]^{(e_0, fe_1)} \\
E \times E \ar[r]^{\id \times f} & E \times B
}
\]
To compute its shadow $\llan{{}_{f}S \odot_{B} S_{f}}$, we take the pull back of $E^{I} \times_{B} E$ along the diagonal of $E$:
\[
\xymatrix{
P \ar[r] \ar[d] & E^{I} \times_{B} E \ar[d]^{(e_0 \pi_1, \pi_2)} \\
E \ar[r]^-{\Delta_{E}} & E \times E 
}
\]
The space $P$ consists of paths $\gamma \in E^I$ for which the endpoints $\gamma(0)$ and $\gamma(1)$ lie in the same fiber over $B$.

\begin{lemma} There is a natural equivalence $\llan{{}_{f}S \odot_{B} S_{f}} \simeq \Sigma^{\infty}_{+} P$ and under this equivalence 
\[
\llan{\coev({}_{f}S)} \colon \llan{U_{E}} \arr \llan{{}_{f}S \odot_{B} S_{f}}
\]
may be identified with the stabilization of the inclusion map $i \colon LE \arr P$.  
\end{lemma}

\begin{proof}
It remains to prove the second claim. Using Example \ref{ex:fiberwise_base_change_duality}, the map 
\[
\coev({}_{f}S) \colon U_{E} \arr {}_{f}S \odot_{B} S_{f}
\]
of spectra over $E \times E$ is the fiberwise stabilization of the diagonal map $E \arr E \times_{B} E$. To apply $\llan{-}$ we make both sides fibrant over $E$, giving the map 
\[
(\id, e_1) \colon E^I \arr E^I \times_{B} E
\]
over $E \times E$.  Pulling back along the diagonal of $E$, we get the inclusion of the loop space $LE$ into $P$.
\end{proof}

\begin{remark}
The space $P$ is also equivalent to the pullback $E \times_{B} LB$, and under this equivalence the coevaluation map $LE \ra E \times_B LB$ becomes $(e_0, Lf)$. We will use both descriptions of $P$ in our geometric model for $\tau_\THH$.
\end{remark}

\begin{remark}\label{rmk:easy}
By a more elementary argument, the map
\[
\psi \circ \eta \colon r_{!}(\_{S}_{f} \boxdot {}_{f}\_{S}) \arr \llan{S_{f} \odot {}_{f}S}
\]
is equivalent to the stabilization of the inclusion of $E$ into $P$ as the constant loops. 
\end{remark}

\noindent This finishes the proof of the following strengthening of Theorem \ref{thm:intro_becker_gottlieb_compatibility}.

\begin{theorem}\label{thm:stronger_compatibility}
For any perfect fibration $f \colon E \arr B$, the diagram
\[
\xymatrix{
\Sigma^{\infty}_{+} B \ar[rr]^-{\tau(f)} \ar[d]_{c} & & \Sigma^{\infty}_{+} E \ar[d]^{c} \\
\Sigma^{\infty}_{+} LB \ar[r]^-{\tau_{\THH}} & \Sigma^{\infty}_{+} LE \ar[r]^-{i} & \Sigma^{\infty}_{+} P
}
\]
commutes up to a natural homotopy.
\end{theorem}

\begin{remark}
In terms of the string diagram calculus, applying the shadow functor $\llan{-}$ to a 2-cell in $\Ex$ corresponds to placing the string diagram on a cylinder \cite{PS13}.  Using the representation of the Becker-Gottlieb pretransfer given in Figure \ref{fig:pretransfer2}, the proof of Theorem \ref{thm:stronger_compatibility} may then be summarized by the isotopy of string diagrams indicated in Figure \ref{fig:main_string_diagram}.

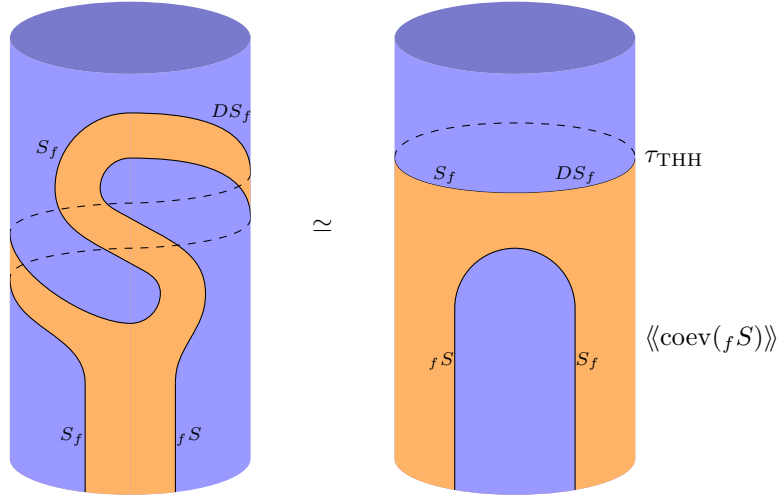
\begin{figure}
\centering
\begin{tikzpicture}[scale=0.8]
\bgcylinder{0,0}{7}{2}{.6}{blue}{blue}
  \coordinate (StopB) at (2,5.75) ;
  \coordinate (StopA) at (2,5) ;
  \draw (StopB) to[out=0, in=90] node [ed] {$DS_f$} ++(2,-1) coordinate (RedgeB);
  \draw (StopA) to[out=0, in=90] ++(2,-1) coordinate (RedgeA);
  \draw (StopB) to[out=180, in=90] node [ed, near end, swap] {$S_f$} ++(-1.25,-1.25) to[out=-90, in=150] ++(1.25, -1.25) coordinate (SmidB);
  \draw (StopA) to[out=180, in=90] ++(-.5, -.5) to[out=-90, in=150] ++(.5, -.5) coordinate (SmidA);
  \draw (SmidA) to[out=-30, in=90] ++(1.25,-1.25) to[out=-90, in=60] ++(-.25, -.75) to[out=240, in=90] ++(-.25, -.75) coordinate (RbotA);
  \draw (SmidB) to[out=-30, in=90] ++(.5, -.5) to[out=-90, in=0] ++(-.5, -.5) coordinate (SbotB);
  \draw[dashed] (RedgeB) to[out=-90, in=0, looseness=.5] ++(-2,-.5) to[out=180, in=90, looseness=.5] ++(-2,-.5) coordinate (LedgeB);
  \draw[dashed] (RedgeA) to[out=-90, in=0, looseness=.5] ++(-2,-.5) to[out=180, in=90, looseness=.5] ++(-2,-.5) coordinate (LedgeA);
  \draw (LedgeB) to[out=-90, in=180, looseness=.7] ++(2, -1.5);
  \draw (LedgeA) to[out=-90, in=90] ++(1.25, -1.75) coordinate (LbotA);
  \draw (LbotA) to node [ed, near start, swap] {$S_f$} ++(0, -3.5);
  \draw (RbotA) to node [ed, near start] {${}_{f}S$} ++(0, -3.5);
  \begin{pgfonlayer}{background}
  \fill[orangefill] (StopB) to[out=180, in=90] ++(-1.25,-1.25) to[out=-90, in=150] ++(1.25, -1.25) to (StopB);
  \fill[orangefill] (SmidA) to[out=-30, in=90] ++(1.25,-1.25) to[out=-90, in=60] ++(-.25, -.75) to[out=240, in=90] ++(-.25, -.75) to ++(0, -3.5) to ++(-.75,0) to (SmidA);
  \fill[orangefill] (LedgeB) to[out=-90, in=180, looseness=.7] ++(2, -1.5) to ++(0,-3) to ++(-2.1,0) to (LedgeB);
  \fill[orangefill] (StopB) to[out=0, in=90] ++(2,-1) to ++(0, -.75) to[out=90, in=0] (StopA) to (StopB);
  \fill[bluefill] (StopA) to[out=180, in=90] ++(-.5, -.5) to[out=-90, in=150] ++(.5, -.5) to (StopA);
  \fill[bluefill] (SmidB) to[out=-30, in=90] ++(.5, -.5) to[out=-90, in=0] ++(-.5, -.5) to (SmidB);
  \fill[bluefill] (LedgeA) to[out=-90, in=90] ++(1.25, -1.75) to  ++(0, -3.5) to ++(-2.1,0) to (LedgeA);
  \end{pgfonlayer}
  \begin{pgfonlayer}{foreground}
  \end{pgfonlayer}
\end{tikzpicture}
  \quad\quad \raisebox{3.5cm}{$\simeq$} \quad\quad 
\begin{tikzpicture}[scale=0.8]
\bgcylinder{0,0}{7}{2}{.6}{blue}{blue}
    \path ($(ul)!(0,5)!(dl)$) coordinate (topcoevL);
  \draw[dashed] (topcoevL) to [out=90,in=90,looseness=0.5] ($(ur)!(topcoevL)!(dr)$) coordinate (topcoevR);
  \draw (topcoevL) to [out=-90, in=-90, looseness=0.5] node [ed, near start] {$S_{f}$} node [ed, near end] {$DS_{f}$} (topcoevR);
  \fill[orangefill] (topcoevL) to [out=-90, in=-90, looseness=0.5] (topcoevR) to ++(.1, 0) to ++(0,-6) 
  to ++(-4.2, 0) to ++(0,6) to (topcoevL);
  \path ($(top)!(0,3)!(bot)$) coordinate (bottomcoev);
  \draw[bluefill] (2,3.5) to[out=180, in=90] (1, 2.5) to node[ed, near start, swap] {${}_{f}S$} (1,-1) to (3, -1) to node [ed, near end, swap] {$S_{f}$} (3, 2.5) 
  to[out=90, in=0] (2,3.5);
\begin{pgfonlayer}{foreground}
\node[anchor=west] at (topcoevR) {$\tau_{\mathrm{THH}}$};
\node[anchor=west] at (4,2) {$\llan{\mathrm{coev}({}_{f}S)}$};
\end{pgfonlayer}
\end{tikzpicture}
\caption{The proof of Theorem \ref{thm:stronger_compatibility}}\label{fig:main_string_diagram}
\end{figure}
\end{remark}

\noindent As the inclusion of constant paths is a section of the evaluation map $e_0 \colon P \arr E$, Theorem \ref{thm:intro_becker_gottlieb_compatibility} follows immediately from Theorem \ref{thm:stronger_compatibility}: the composite
\[
\Sigma^{\infty}_+ B \oarr{c} \Sigma^{\infty}_+ LB \xarr{\tau_{\THH}} \Sigma^{\infty}_+ LE \oarr{e_0} \Sigma^{\infty}_+ E
\]
is naturally homotopic to the Becker-Gottlieb transfer.

\subsection{The proof of Corollary \ref{cor:intro_a_theory}}\label{subsec:proof_cor}

To see how Theorem \ref{thm:intro_becker_gottlieb_compatibility} implies Corollary \ref{cor:intro_a_theory}, it suffices to recall the construction of the maps $i \colon \Sigma^{\infty}_{+}X \arr A(X)$ and $p \colon A(X) \arr \Sigma^{\infty}_{+} X$ that split stable homotopy off of Waldhausen's $A$-theory, and to argue that the following diagram commutes up to canonical homotopy.
\begin{equation} \label{eq:trace_and_loops} \xymatrix{
	\Sigma^{\infty}_+ B \ar[r]^-{i} \ar[rd]_-c & A(B) \ar[d] \ar[r]^-{\tau_{A}} & A(E) \ar[d] \ar[rd]^-{p} & \\
	& \Sigma^{\infty}_+ LB \ar[r]_-{\tau_{\THH}} & \Sigma^{\infty}_+ LE \ar[r]_-{e_0} & \Sigma^{\infty}_+ E
} \end{equation}
The vertical maps are the topological Dennis trace map---we will use two different definitions of this map, but they are known to be equivalent. As we remarked after Definition \ref{def:perf_r}, our chosen model of $\Perf_{R}$ is Dwyer-Kan equivalent to the one produced from the enriched Waldhausen category of cofibrant perfect $R$-modules as defined in \cite{blumberg_mandell2}, and therefore the model of the Dennis trace from \cite[\S 2.5]{blumberg_mandell2} makes the middle square of \eqref{eq:trace_and_loops} commute. Indeed, this agreement with the $A$-theory transfer is the reason why we first defined $\tau_{\THH}$ using spectral categories, instead of simply taking the Reidemeister trace.

The map $i$ is defined in e.g. \cite[\S 1]{wald2} as an inclusion-of-units map
\[ B(\Omega B) \to BGL_1(\Sigma^\infty_+ \Omega B) \to \Omega B \left( \coprod_k BGL_k(\Sigma^\infty_+ \Omega B) \right) \simeq \Omega^\infty A(B) \]
although it is more commonly known as a composite of the unit $S \to A(*)$ and an assembly morphism in the sense of \cite{weiss_williams_assembly}. Composing this definition of $i$ with B\" okstedt's original definition of the topological Dennis trace \cite[\S 5]{BHM},\cite[\S 2.6]{madsen_survey} shows that the left-hand triangle of \eqref{eq:trace_and_loops} commutes. 

Finally, the map $p$ has two definitions, one of which is simply by the right-hand triangle of \eqref{eq:trace_and_loops}. We briefly recall the other definition and why it is equivalent. It uses Waldhausen's stabilization procedure\footnote{This stabilization is similar to, but distinct from, the linear approximation of $A(X)$ in the functor calculus of Goodwillie \cite{calc1}} $F \leadsto F^S$ for homotopy functors $F$ from well-based spaces to spectra. The composite of $i$ with $A(X) \to A^S(X)$ is an equivalence of spectra, and $p$ is defined to be the composition of $A(X) \to A^S(X)$ with the inverse of this equivalence. Applying stabilization to the topological Dennis trace gives a commuting diagram
\[ \xymatrix{
	\Sigma^\infty_+ X \ar[r]^-i \ar[d]^-\sim & A(X) \ar[r] \ar[d] & \Sigma^\infty_+ LX \ar[r]^{e_0} \ar[d] & \Sigma^\infty_+ X \ar[d]^-\sim \\
	(\Sigma^\infty_+)^S (X) \ar[r]^-\sim & A^S(X) \ar[r] & (\Sigma^\infty_+ L)^S (X) \ar[r]^{e_0} & (\Sigma^\infty_+)^S (X).
} \]
By the previous paragraph, the composite along the top row is homotopic to the identity.
The commutativity of this diagram then implies that our two definitions of $p$ agree in the homotopy category. (This conclusion is also essentially contained in \cite{wald2}.) This finishes the justification of Corollary \ref{cor:intro_a_theory}.

\section{A geometric model of the Reidemeister trace}\label{sec:geometric}

In this section we will prove Theorem \ref{thm:intro_geometric_THH_transfer} by giving a geometric model for the Reidemeister trace, which by Theorem \ref{thm:reid_trace_is_THH_transfer} is the $\THH$ transfer. Along the way we will establish a useful multiplicative structure on $\tau_\THH$ (Proposition \ref{prop:comodule}).

The main idea is to give explicit geometric descriptions of the evaluation and coevaluation morphisms for the dual pair $(S_{f}, D S_{f})$ (Proposition \ref{prop:base_change_dualble_Ex}) in the case where $f$ is a smooth fiber bundle.  Intuitively, $S_f$ is a copy of $E$ sitting over $B \times E$, while $DS_f$ is a copy of $E$ that has been desuspended by the vertical tangent bundle, sitting over $E \times B$. The coevaluation map is a Pontryagin-Thom collapse and the evaluation map is a scanning map. As a result, the Reidemeister trace takes a free loop in $B$ to a free loop in $E$ by the following sequence of steps:
\begin{figure}[H]
\def\svgwidth{\linewidth}
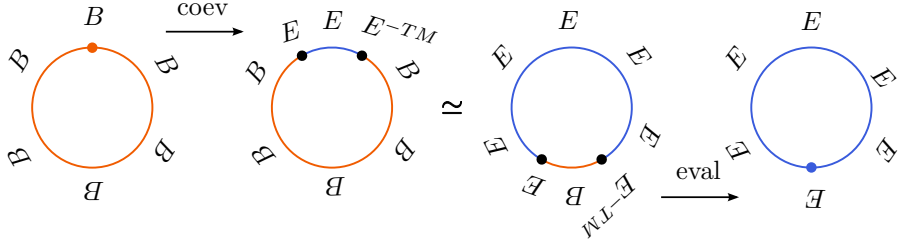
\caption{The Reidemeister trace (compare with Figure \ref{fig:bicat_trace})}
\end{figure}

There are a few different ways to make this idea precise -- see for instance \cite[\S 18.6]{MS} in the case where $B = *$. Here we will follow the techniques of \cite{cohen_multiplicative} and the appendix of \cite{malkiewich_thesis} closely. The idea is to model the $n$-th spectrum level of $E^{-TM}$ by embedding the fibers of $E$ into $\R^n$, and then taking the Thom space of the normal bundle of this embedding. To make an orthogonal spectrum, and to handle a possibly infinite base space, we allow ourselves to make different choices of embeddings, and we allow $n$ to go to $\infty$.

We begin with the geometric model for the circle product. Throughout this section $\odot_B$ will refer to the strict point-set formula for the circle product $\pi^{B}_{!}\Delta_B^*(- \osma -)$ (see Example \ref{ex:Ex}), while $\odot^{\bL}_B$ will refer to its derived form. We also abbreviate the suspension spectrum $\Sigma^{\infty}X_{+}$ of a space with a disjoint basepoint or section to $X_{+}$. We now give a simple condition for when $\odot$ is derived. Recall that an $h$-cofibration is a map satisfying the homotopy extension property, and an $h$-fibration is a Hurewicz fibration.
\begin{lemma}\label{lem:only_one_fibration}
If $X$ is an $h$-cofibrant retractive space (or spectrum with $h$-cofibrant levels) over $A \times B$, and $Y$ is an unbased space over $B \times C$ with $Y \ra B$ an $h$-fibration, then the point-set formula $\pi^{B}_{!} \Delta_{B}^* (X \oline{\sma} Y_+)$ for $X \odot_B Y_+$ is equivalent to the derived circle product.
\end{lemma}

\begin{proof} 
It suffices to assume $X$ is a retractive space. By a diagram-chase the point-set formula for the circle product is given by the pushout square
\[ \xymatrix{
A \times B \times_B Y \ar[r] \ar[d] & X \times_B Y \ar[d] \\
A \times C \ar[r] & \pi^B_! \Delta_B^* X \osma Y_+ }
\]
Using for instance the main theorem of \cite{kieboom}, the top horizontal map is a cofibration, so this is a homotopy pushout square. Therefore it suffices to check that $X \times_B Y$ has the correct homotopy type, and for this it is enough to assume that $Y \arr B$ is an $h$-fibration.
\end{proof}

So if $X$ is a level-wise $h$-cofibrant spectrum over $A \times B$, and we regard the path space $B^I$ as a space over $B \times B$ by evaluating at the endpoints, the operation $X \longmapsto X \odot_B (B^I)_+$ is a replacement by a weakly equivalent spectrum whose levels are $h$-fibrations over $B$. When $A = *$ this is the fibrant replacement functor $P(-)$ used in \cite{coassembly}. The previous lemma eliminates the need to check that many of our objects are fibrant, and it also gives us geometric control of the shadow and its cyclic isomorphism. We record the details below.

\begin{proposition}\label{prop:model_for_odot_shadow} \hspace{2in}
\begin{itemize}
\item[(i)] If $X$ is a level-wise $h$-cofibrant spectrum over $A \times B$ and $Y_+$ is an ex-space with disjoint section over $B \times C$, then the point-set formula for the circle product $X \odot_{B} B^{I}_{+} \odot_{B} Y_{+}$ is equivalent to the derived circle product $X \odot_{B}^{\bL} Y_{+}$ in the homotopy category $\ho \Sp{}{A \times C}$.
\item[(ii)] If $X$ is a level-wise $h$-cofibrant spectrum over $A \times A$, then the point-set formula for the circle product $X \odot_{A \times A} A^{I}_{+}$ is equivalent to the shadow $\llan{X}$ of $X$ in the stable homotopy category.
\item[(iii)] If $X$ is a level-wise $h$-cofibrant spectrum over $E \times B$, and $f\colon E \arr B$ is a fibration, then both $\llan{X \odot^{\bL}_B S_f}$ and $\llan{S_f \odot^{\bL}_E X}$ are derived using the strict circle product and the model of the shadow from (ii). The cyclic isomorphism $\theta$ is given at spectrum level $n$ by the map of cofiber sequences
\[ \xymatrix{
E^I \ar[r] \ar[d] & (1 \times f)^* X_n \times_{E \times E} E^I \ar[d] \ar[r] & \llan{X \odot^{\bL}_B S_f}_n \ar[d]^-\theta \\
B^I \times_B E \ar[r] & X_n \times_{B \times B} B^I \ar[r] & \llan{S_f \odot^{\bL}_E X}_n } \]
where the vertical maps come from the obvious projection $E^I \ra B^I \times_B E$.
\end{itemize}
\end{proposition}

\begin{proof}

(i) By Lemma \ref{lem:only_one_fibration} and the identity
\[ B^I_+ \odot_B Y_+ = (B^I \times_B Y)_+, \]
this follows from the fact that $B^I \times_B Y$ is a fibration over $B$.

(ii) A diagram chase shows that $r_!\Delta_A^* X$ is isomorphic to $X \odot_{A \times A} A_+$, and that it is derived if the levels of $X$ are $h$-cofibrant and $h$-fibrant. So if $X$ is merely cofibrant, we derive the shadow by
\[ r_!\Delta_A^*(A^I_+ \odot_A X \odot_A A^I_+) \cong (A^I_+ \odot_A X \odot_A A^I_+) \odot_{A \times A} A_+ \cong X \odot_{A \times A} A^I_+ \]
using the isomorphism $A^I_+ \odot_A A^I_+ \cong A^I_+$ that concatenates the paths. We will use later that this construction is given on each spectrum level by the pushout
\[ \xymatrix{
{*} \ar[r] & X \odot_A A^I_+ \\
A^I \ar[r] \ar[u] & X_n \times_{A \times A} A^I \ar[u] } \]

(iii) As in (ii), the operation $r_!\Delta_E^*(1\times f)^* X$ is isomorphic to $(1 \times f)^* X \odot_{E \times E} E_+$ and is derived when the levels of $X$ are $h$-cofibrant and $h$-fibrant. When $X$ is only cofibrant, we replace it by $E^I_+ \odot_E X \odot_B B^I_+$. The resulting model for $\llan{ X \odot^{\bL}_{B} S_f}$ is:
\[ (E^I_+ \odot_E X \odot_B B^I_+ \odot_B S_f) \odot_{E \times E} E_+ \cong (E^I_+ \odot_E X \odot_B B^I_+) \odot_{B \times E} E_+. \]
Similarly, the derived form of the operation $r_!\Delta_B^*(f\times 1)_! X$ is represented by
\[ (S_f \odot_E E^I_+ \odot_E X \odot_B B^I_+) \odot_{B \times B} B_+ \cong (E^I_+ \odot_E X \odot_B B^I_+) \odot_{B \times E} E_+. \]
The cyclic isomorphism for the strict shadow is given by the canonical identification that comes from describing both operations as a pullback of the external smash product from $E \times B \times B \times E$ to $E \times B$, and then pushed forward to a point. For the derived shadow, we apply the same operation to input that is cofibrant and fibrant. In this case, that operation gives a spectrum which at each level is a strict cofiber of the map
\[ \xymatrix{ (E^I \times B^I) \times_{B \times E} E \ar[r] & (E^I \times_E X_n \times_B B^I) \times_{B \times E} E } \]
coming from the basepoint section $E \times B \ra X_n$. If we don't derive the circle product but do derive the shadow, we get the same expression but with $E^I$ or $B^I$ replaced by constant paths. With those models, the cyclic isomorphism of shadows is given by the zig-zag of maps of cofibers
\[ \xymatrix{
E^I \ar[r] \ar[d]^-\sim & (1 \times f)^* X_n \times_{E \times E} E^I \ar[d]^-\sim \ar[r] & \llan{X \odot_B S_f}_n \ar[d]^-\sim \\
(E^I \times B^I) \times_{B \times E} E \ar[r] & (E^I \times_E X_n \times_B B^I) \times_{B \times E} E \ar[r] & \llan{X \odot_B B^I \odot_B S_f}_n \\
B^I \times_B E \ar[r] \ar[u]_-\sim & X_n \times_{B \times B} B^I \ar[u]_-\sim \ar[r] & \llan{S_f \odot_E X}_n \ar[u]_-\sim } \]
The marked equivalences are straightforward, but use the fact that $f$ is a fibration. It now suffices to check that the map in the claim commutes with these two equivalences up to homotopy. This reduces to the fact that the two maps
\[ E^I \rightrightarrows B^I \times_B E^I, \qquad \gamma \mapsto (c_{f(\gamma(0))},\gamma), \qquad \gamma \mapsto (f \circ \gamma,c_{\gamma(1)}) \]
are homotopic over $B \times E$.
\end{proof}

\noindent These geometric models for $\odot$ allow us to prove the following result.
\begin{proposition}\label{prop:comodule} For any perfect fibration $f \colon E \arr B$, the free loop transfer $\tau_\THH \colon LB_+ \arr LE_+$ is a map of $LB_+$-comodules.
\end{proposition}

\begin{proof}
We prove that each step of the Reidemeister trace
\[
LB_+ \arr \llan{S_f \odot^{\bL}_{E} D S_{f}} \oarr{\theta} \llan{D S_{f} \odot^{\bL}_{B} S_f} \arr LE_+
\]
is a map of $LB_+$-comodules.
The comodule structure on $LB_+$ is defined by the diagonal and on $LE_+$ by the diagonal and $Lf$. The intermediate terms require more explanation. By Lemma \ref{lemma:dualizable_in_exb_detected_on_fibers} and Lemma \ref{lem:psi_equiv}, we may model $DS_f$ by the pushforward of the spectrum $D\_{S}_f$ from $E$ to $E \times B$ along $1 \times f$, and form its evaluation and coevaluation maps between appropriately cofibrant and fibrant models in the category $\Ex_B$, before pushing forward to $\Ex$ and taking shadows as in part (ii) of Proposition \ref{prop:model_for_odot_shadow}.

Therefore both $\llan{S_{f} \odot^{\bL}_E D S_{f}}$ and $\llan{D S_{f} \odot^{\bL}_B S_{f}}$ come with projection maps into $B^I$, which land in paths whose endpoints coincide. This gives $LB$-comodule structures, and they are compatible with $\theta$ by part (iii) of Proposition \ref{prop:model_for_odot_shadow}. The structure on $\llan{S_{f} \odot^{\bL}_E D S_{f}}$ is natural in maps of spectra over $B \times B$ that are pushforwards along $\Delta_B = \iota$. Since the coevaluation map and the cofibrant and fibrant replacement maps all have this property, they agree with the comodule structure. Similarly, the evaluation map and its attendant cofibrant and fibrant replacements are all maps over $E \times E$ that come from maps over $E \times_B E$, and every such map preserves the projection $E^I \times_{B \times B} B \arr LB$, so they preserve the comodule structure as well.
\end{proof}

Now we will build our geometric model for $\tau_\THH$, when $f: E \ra B$ is a smooth fiber bundle over a connected CW complex $B$, with compact fiber $M$. Fix $\epsilon > 0$. Let $\Emb_{\epsilon}(M,\R^\infty)$ be the space of smooth embeddings of $M$ into $\R^\infty$ with a tubular neighborhood of radius $\epsilon$. This is contractible and is filtered by the closed subspaces $\Emb_{\epsilon}(M,\R^n)$. The diffeomorphism group $\Diff(M)$ acts on the left on $\Emb_{\epsilon}(M,\R^n)$ by $\psi(i) = i \circ \psi^{-1}$. By the mixing construction, this defines a fiber bundle
\[
B_n := \overline{f}^*E\Diff(M) \times_{\Diff(M)} \Emb_{\epsilon}(M,\R^n) \arr B,
\]
where $\overline{f}^*E\Diff(M)$ is the pullback of $E\Diff(M)$ along a chosen classifying map $\overline{f} \colon B \arr B\Diff(M)$ for $f$.  It is natural to think of the fiber of this bundle as the space of embeddings of the fiber $E_b$ into $\R^n$ with tubular neighborhood $\epsilon$. When describing maps with formulas, we will describe a point in $B_n$ by naming the embedding $i \colon M \arr \R^n$, and leaving the choice of $b$ understood.

We let $O(n)$ act on $B_n$ by acting on the fiber by $\rho(i) = (\rho \circ i)$. This action is compatible with the inclusions $B_n \ra B_{n+1}$ along the homomorphisms $O(n) \ra O(n+1)$, so we get a parametrized spectrum over $B$ whose $n$-th space is $\Sigma^n_B (B_n)_+$.  We write $\B$ for the parametrized spectrum over $B \times B$ obtained by pushing forward along the diagonal, so that $\B(n) = \Sigma^{n}_{B \times B} (B_n)_+$.  Collapsing away the embeddings gives a stable equivalence of parametrized spectra $\B \ra \Sigma^{\infty}_{B \times B} B_+ = U_B$, and we use this map to identify $\B$ with the unit 1-cell $U_B$ in the homotopy category $\ho \Sp{}{B \times B} = \Ex(B, B)$.

%


Each point of $B_n$ determines an embedding of a compact manifold $E_b \cong M$ into $\R^n$. This manifold therefore inherits a smooth metric from $\R^n$, and we need a simple lemma about its geodesics.

\begin{lemma}
If $i \colon M \ra \R^n$ is a smooth embedding with tubular neighborhood of radius $\epsilon$, and $v \in \R^n$ is in this tubular neighborhood with closest point $x \in M$, then the open subspace of $M$ given by $U = \{y \in M : \|v - i(y)\| < \epsilon\}$ is contractible.
\end{lemma}

\begin{proof}
If $y$ is a critical point of the smooth function $d(i(-),v)^2: M \ra \R$ then the line from $i(y)$ to $v$ is perpendicular to $i(M)$. So if $v$ is in the tubular neighborhood, there is only one such $y$ within $\epsilon$ of $v$. Therefore $U$ contains only one critical point, the minimum $x$. Flowing along the gradient provides the contracting homotopy. 
\end{proof}

\begin{construction}\label{construction_gamma}
To an embedding $i$ of the fiber $E_{b}$, a point $v$ in the $\epsilon$-tubular neighborhood of $i(E_b)$ with closest point $x \in E_b$, and a point $y \in E_b$ such that $\|v - i(y)\| < \epsilon$ we associate a path $\gamma_{x, y}$ in $E_b$ from $x$ to $y$ in the following way.  The open set $B_{\epsilon}(v) \cap i(E_b)$ is contractible and contains both $x$ and $y$, so we may find a path from $x$ to $y$ whose image under $i$ lies in this open set. We let $\gamma_{x, y}$ be the unique geodesic homotopic to this path. We parametrize $\gamma_{x,y}$ so that it is always defined on the unit interval and moves at constant speed. This guarantees that the choice of path depends in a continuous way on $x$, $y$, and $i$.
\end{construction}


In order to define our model for $U_{E}$, we write $\Sigma^{n, \epsilon}$ for the fiberwise suspension functor obtained by taking the fiberwise quotient by the subspace consisting of points whose suspension coordinates have total length at least $\epsilon$.  We write $\Sigma^{\infty, \epsilon}$ for the suspension spectrum functor given by $\Sigma^{n, \epsilon}$ at each level.  We let
\[
\E = \Sigma^{\infty, \epsilon}_{E \times E} (E^I, (e_0, e_1))_+
\]
be the modified suspension spectrum of the space of paths in $E$, considered as a space over $E \times E$ by evaluation at the endpoints.  The inclusion of constant paths defines a stable equivalence $U_{E} \simar \E$ of spectra over $E \times E$.

Next, we define our model for the right dual $DS_f$, which is a modified version of the spherical fibration whose Thom space is $E^{-TM}$. Let $\tilde T_n$ be the closed $\epsilon$-tube consisting of points $(i,v)$ with $i \in \Emb_{\epsilon}(M,\R^n)$, $v \in \R^n$, and
\[ \underset{m \in M}\min \|v - i(m)\| \leq \epsilon \]
Let $\partial \tilde T_n$ consist of those points for which the minimum distance is exactly $\epsilon$.
The diffeomorphism group acts on both of these spaces by $\psi(i,v) = (i \circ \psi^{-1},v)$, and $O(n)$ acts by $\rho(i,v) = (\rho \circ i,\rho(v))$.  Thus we get a bundle $T_n \ra B$ whose fiber is $\tilde T_n$, with a sub-bundle $\partial T_n$, for each $n$. The projection of each point $(i,v)$ to the closest point of $i(E_b)$ also defines a bundle $p_n: T_n \ra E$ whose fiber is $\Emb_{\epsilon}(M,\R^n) \times D^{n-d}_\epsilon$. We let $T_n /_{E} \partial T_n$ denote the fiberwise quotient over $E$, and write
\[
\E^{-TM} \colon n \longmapsto (T_n /_{E} \partial T_n) \cup_{E} (E \times B)
\]
for the pushforward parametrized spectrum over $E \times B$.

Now we define our model for the coevaluation map $\coev \colon U_{B} \arr S_f \odot_{E} D S_{f}$.  The target is equivalent to the derived base-change $(f, 1)_{!}DS_{f}$ which in our model has $n$-th spectrum level given by
\[
(f, 1)_{!} \E^{-TM} = T_n/_{E}\partial T_n \cup_{E} B \times B.
\]
Our geometric model for the coevaluation map is the map of spectra over $B \times B$
\begin{equation}\label{eq:coev_model}
\mathrm{\bf coev} \colon \B \arr (f, 1)_{!} \E^{-TM} 
\end{equation}
whose $n$-th level is the Pontryagin-Thom collapse map
\begin{align*}
(B_n)_{+} \sma_{B \times B} S^n_{B \times B} &\arr T_n/_{E} \partial T_n \cup_{E} (B \times B) \\
i \sma w &\longmapsto (i, w) 
\end{align*}

The model for the evaluation map $\ev \colon D S_{f} \odot_{B} S_{f} \arr U_{E}$ has source given by the derived base-change $(1 \times f)^* \E^{-TM}$.  Since the composite
\[
T_n /_{E} \partial T_n \oarr{p_n} E \oarr{ (1, f) } E \times B
\]
is a fibration, we may compute the derived pullback of $\E^{-TM}(n)$ along $1 \times f \colon E \times E \arr E \times B$ by the strict pullback, which we find to be the ex-space 
\[
(T_n /_{E} \partial T_n \times_{B} E) \cup_{E \times_{B} E} (E \times E)
\]
over $E \times E$. Our geometric model for the evaluation map is the map of spectra over $E \times E$
\begin{equation}\label{eq:eval_model}
\mathrm{\bf eval} \colon (1 \times f)^* \E^{-TM} \arr \E
\end{equation}
which is defined on the $n$-th spectrum level by the scanning map
\begin{align*}
(T_n /_{E} \partial T_n \times_{B} E) \cup_{E \times_{B} E} (E \times E) &\arr \Sigma^{n, \epsilon}_{E \times E} E^{I}_{+} \\
(i \in \Emb_{\epsilon}(E_b, \R^n), v \in \R^n, y \in E_b) &\longmapsto  (v - i(y)) \sma \gamma_{x, y},
\end{align*}
where $x = p_n(i, v)$ is the closest point in $M$ to $v$. Notice that we need only define the map when $\| v - i(y) \| < \epsilon$, in which case we may find a unique geodesic $\gamma_{x, y}$ from $x$ to $y$ in $M$ using Construction \ref{construction_gamma}.


\begin{proposition}\label{prop:geometric_duality}
The maps $\mathrm{\bf coev}$ and $\mathrm{\bf eval}$ define a dual pair $(S_f, \E^{-TM})$ in the bicategory $\Ex$.
\end{proposition}

\begin{proof}
We first identify the operation $- \odot^{\bL}_{E} E^{-TM}$ with the derived right adjoint of $- \odot^{\bL}_B S_f$ using a scanning map, and then we check that our evaluation and coevaluation maps agree with the counit and unit of this adjunction. Notice that the levels of $E^{-TM}$ are $h$-cofibrant, so $(-)_+ \odot_{E} E^{-TM}$ is derived on fibrations over $E$, and $- \odot_B S_f$ is always derived.

The operation $- \odot_B S_f$ gives a functor from spectra over $A \times B$ to spectra over $A \times E$, for any CW complex $A$, that is isomorphic to the pullback operation $(1 \times f)^*$. By \cite{MS} this immediately implies that it has a right adjoint $(1 \times f)_*$, though its definition may take us outside the category of weak Hausdorff spaces. To describe this right adjoint more geometrically, we take each fiber bundle $Y$ over $A \times E$ to a fiber bundle over $A \times B$ whose fiber over $(a,b)$ is the space of sections $\Gamma(Y|_{\{a\} \times f^{-1}(b)})$. Let $Z_{a,b}$ refer to the space $Y|_{\{a\} \times f^{-1}(b)}$, a bundle over $f^{-1}(b) \cong M$ that is in general non-trivial. To assemble these spaces of sections into a bundle, we use that $A$ and $B$ are CW complexes and therefore locally contractible. Over each contractible neighborhood $U \times V$ containing $(a,b)$, we choose an isomorphism between the preimage in $A \times E$ with $U \times V \times M$. Then we use the contraction of this onto $* \times * \times M$ to choose a compatible isomorphism of $Y$ with $U \times V \times Z_{a,b}$ over $M$. This gives a local topology which is invariant under the choices of isomorphisms made above.

Let us call this geometric construction $\Gamma(-)$. It is the right adjoint of $(1 \times f)^* = - \times_B E$, since in each trivial neighborhood $U \times V$, the operation $- \times M$ is left adjoint to $\Map(M,-)$. This adjunction is derived when we require all spectrum levels to be fiber bundles.

Now that we have identified the right adjoint of $- \odot^{\bL}_B S_f$, we compare it to $- \odot^{\bL}_E \E^{-TM}$. It suffices to do this on $\Sigma^\infty_{E \times E} E^I_+$, the fibrant replacement of $U_E$, but we consider $\Sigma^\infty_{A \times E} Y_+$ for any bundle $Y$ over $A \times E$. We define a scanning map of spectra over $A \times B$
\[ \sigma: Y_+ \odot_E \E^{-TM} \ra \Gamma(\Sigma^{\infty,\epsilon}_{A \times E} (Y \times_E E^I)_+) \]
\[ y \sma (i,w) \mapsto (m \mapsto (w - i(m)) \sma y \sma \gamma_{p_n(i,w),m}) \]
Of course, $\sigma$ is only nontrivial when $i(m)$ is within $\epsilon$ of $w$, which ensures that the geodesic $\gamma_{p_n(i,w),m}$ is well-defined. The map $\sigma$ is an equivalence, because both source and target are bundles over $A \times B$, and on each fiber it is given up to homotopy equivalence by the classical scanning equivalence
\[ Z_{a,b}^{-TM} \overset\sim\ra \Gamma_M(\Sigma^{\infty,\epsilon}_{M} Z_{a,b} \times_M M^I_+) \]
\[ z \sma (i,w) \mapsto (m \mapsto w - i(m) \sma z \sma_M \gamma_{p_n(i,w),m}) \]
where $Z = Y|_{a \times f^{-1}(b)}$ is a bundle over $f^{-1}(b) \cong M$ (cf. \cite{malkiewich_thesis}*{App A}).

Therefore $- \odot^{\bL}_E \E^{-TM}$ is equivalent to the derived right adjoint of $- \odot^{\bL}_B S_f$, so $\E^{-TM}$ is the dual of $S_f$. To get the evaluation map, we compose $\sigma$ for $Y = U_E$ with the counit $\mathrm{ev}$ of the adjunction to get a map of spectra over $E \times E$
\[ E_+ \odot_E E^{-TM} \odot_B E_+ \overset\sigma\arr \Gamma(E^I_+) \odot_B E_+ \overset{\mathrm{ev}}\arr E^I_+ \]
In formulas, $(i,w)$ with closest point $e$ and a second point $e'$ in the same fiber are sent to the point $w - i(e') \sma \gamma_{e,e'}$ when $e$ and $e'$ are sufficiently close. This agrees on the nose with $\mathrm{\bf eval}$.

For coevaluation, we take instead $Y = U_B \odot_B S_f = S_f$ and show that the composite of $\mathrm{\bf coev}$ and $\sigma$ agrees with the unit of the adjunction, up to homotopy:
\[ \B \overset{\mathrm{\bf coev}}\arr (f \times 1)_!E^{-TM} \overset\sigma\arr \Gamma((f \times 1)_!E^I_+) \]
In formulas, this composite is
\[ w \sma i \mapsto (i,w) \mapsto (e' \mapsto w - i(e') \sma \gamma_{p_n(i,w),e'}) \]
where $e'$ must lie in the same fiber as $p_n(i,w)$. We modify this up to homotopy, first by pulling $p_n(i,w)$ along the path $\gamma_{p_n(i,w),e'}$ while shortening the path so that it ends up constant, and then by taking a straight-line homotopy between $i$ and the function sending all of $E$ to $0 \in \R^n$ to get rid of the term $i(e')$. These moves all respect the orthogonal spectrum structure and the projections to $B$ and $E$ on the left and right, respectively. At the end of the homotopy, we get the composite of the projection $\B \simar U_B$ and the unit of the adjunction 
\[ w \mapsto (e' \mapsto w \sma \gamma_{e',e'}) \]
This finishes the proof.
\end{proof}

By the uniqueness of duals, it follows from the Proposition that the dual pairs $(S_{f}, D S_{f})$ and $(S_{f}, \E^{-TM})$ are isomorphic in $\Ex$.  We will now use the new models for the coevaluation and evaluation maps to give concrete geometric descriptions of the three stages of the $\THH$ transfer. Observe that the map $Lf: LE \arr LB$ can be factored into three maps as follows:
\[ \xymatrix{ LE \ar[r] & P = E^I \times_{B \times B} B \ar[r]^-\sim & E \times_{B \times B} B^I \ar[r] & LB } \]
The second map projects the path in $E$ down to $B$ and remembers only the endpoint of the path in $E$; it is clearly an equivalence (see Figure \ref{fig:intro_PplusEBLB}). The last map is a pullback of $f: E \arr B$ and so it is also a fiber bundle with fiber $M$. The first map is a closed inclusion and a pullback of the fiberwise diagonal $E \arr E \times_B E$. It has a ``tubular neighborhood'' of all paths whose endpoints are at most $\epsilon$ apart along the metric internal to $M$. For instance, this point lies on its boundary:
\begin{figure}[H]
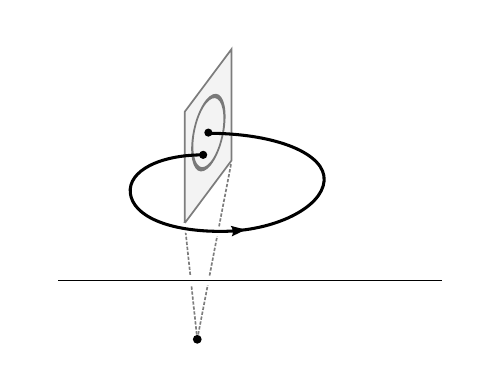
\end{figure}
The tubular neighborhood and its boundary are homotopy equivalent as a pair to the disc bundle and sphere bundle of the vertical tangent bundle $TM$, pulled back from $E$ to $LE$. By abuse of notation, we will let $TM$ denote the pullback of the vertical tangent bundle to any of the first three spaces above.

\begin{theorem}\label{thm:geometric_THH_transfer}
When $f$ is a fiber bundle with compact manifold fiber, the $THH$ transfer for $f$ is the composite
\[ \xymatrix{ \Sigma^\infty_+ LB \ar[r] & \Sigma^{-TM} E \times_{B \times B} B^I \ar@{<-}[r]^-\sim & \Sigma^{-TM} E^I \times_{B \times B} B \ar[r] & \Sigma^\infty_+ LE } \]
of the Pontryagin-Thom umkehr map for the fiber bundle $LB \times_{B} f$, the inverse of the above equivalence, and the desuspension by $TM$ of the umkehr map which collapses onto the tubular neighborhood of $LE$.
\end{theorem}


\begin{proof}
By Proposition \ref{prop:geometric_duality}, the $\THH$ transfer is equal to the composite
\[
\llan{\B} \oarr{\mathrm{\bf coev}} \llan{S_f \odot^{\bL}_{E} \E^{-TM}} \oarr{\theta} \llan{\E^{-TM} \odot^{\bL}_{B} S_f} \oarr{\mathrm{\bf eval}} \llan{\E}.
\]
We use Proposition \ref{prop:model_for_odot_shadow} to give geometric models for the derived circle products and shadows occurring in this composite.  

The first step is the top map in the commutative diagram
\[
\xymatrix{
\B \odot_{B \times B} B^{I}_{+} \ar[d]_{\simeq} \ar[r]^-{\mathrm{\bf coev}} & ((f, 1)_! \E^{-TM}) \odot_{B \times B} B^{I}_{+} \ar[d]^{\simeq} \\
B_+ \odot_{B \times B} B^{I}_{+} \ar[r]^-{\mathrm{PT}} & ((f, f)_! S_{E}^{-TM}) \odot_{B \times B} B^{I}_{+},
}
\]
where the vertical maps are the equivalences induced by projecting the spaces of embeddings to a point.  The bottom map is isomorphic to the fiberwise Pontryagin-Thom collapse map
\[
LB_+ \arr (E \times_{B} LB)^{-TM}
\]
associated to the fiber bundle $f \times_{B} LB \colon E \times_{B} LB \arr LB$.

The second step is immediate from part (iii) of Proposition \ref{prop:model_for_odot_shadow}, setting $X = E^{-TM}$. We get the inverse of the desuspension of the claimed projection map $E^I \times_{B \times B} B \overset\sim\arr E \times_B LB$.

The final step is our geometric evaluation map, joined along $E \times E$ to the space of paths in $E$:
\[ (E^I \times_{B \times B} B)^{-TM} \cong E^I_+ \odot_{E \times E} (E^{-TM} \odot_B S_f) \arr E^I_+ \odot_{E \times E} (\Sigma^{\infty,\epsilon} E^I_+) \cong \Sigma^{\infty,\epsilon} LE_+ \]
We will compare it to the following collapse map. Let $P = E^I \times_{B \times B} B$, and let $D(TM)$ refer to the closed $\epsilon$-disc bundle of the vertical tangent bundle, embedded into $E \times_B E$ by the exponential map in the first coordinate. Consider the sequence of pullback squares, in which the vertical maps are fibrations and the horizontal maps are closed inclusions:
\[ \xymatrix{
LE \ar[r]^-\sim \ar[d]^-{e_0} & P|_{D(TM)} \ar[r] \ar[d] & P \ar[r] \ar[d] & E^I \ar[d] \\
E \ar[r]^-\sim & D(TM) \ar[r] & E \times_B E \ar[r] & E \times E
} \]
Since open and closed inclusions are preserved by pullbacks, $P|_{D(TM)}$ is the closure of an open neighborhood $\mathring{D}(TM)$ of $LE$ inside $P$.
Let $S(TM)$ refer to its boundary, the $\epsilon$-sphere bundle. Let $\hat{P}$ refer to the complement of $P|_{\mathring{D}(TM)}$ in $P$. Our desired collapse map is the composite
\[ P \ra P / \hat{P} \cong P|_{D(TM)}/P|_{S(TM)} \simeq D(e_0^* TM)/S(e_0^* TM) = (LE)^{TM}. \]
The homotopy equivalence in this composite is given by the formula
\vspace{1em}

\centerline{
\begin{tabular}{ccc}
$P|_{D(TM)}$ && $D(e_0^* TM)$ \\[2pt]
\includegraphics[scale=1.7]{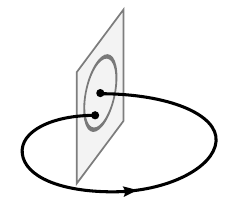}&&
\includegraphics[scale=1.7]{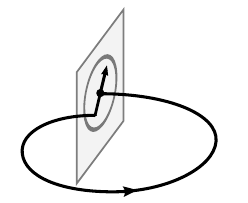}\\[2pt]
$\alpha$ & $\mapsto$ & $(\gamma_{\alpha(1),\alpha(0)} \cdot \alpha, -\exp^{-1}_{\alpha(1)}(\alpha(0)))$ \\[2pt]
\includegraphics[scale=1.7]{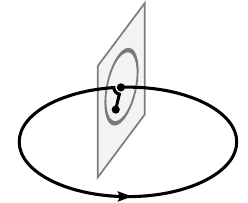}&&
\includegraphics[scale=1.7]{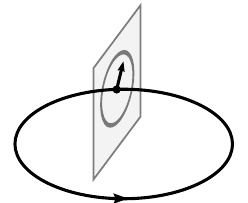}\\[2pt]
$\gamma_{\exp_{\beta(0)}(-x),\beta(0)} \cdot \beta$ & $\mapsfrom$ & $(\beta, x)$ \\
\end{tabular}
\vspace{1em}
}

Here $\cdot$ refers to concatenation of paths and $\gamma_{-,-}$ is our continuous rule for assigning any two points in the same fiber of $E$ that are less than $\epsilon$ apart to the unique short geodesic connecting them. It is elementary to verify that this gives a homotopy equivalence of pairs as stated. It respects the projection to the second copy of $E$, so we may take the fiberwish smash product over $E$ with $E^{-TM}$, giving
\[ P^{-TM} \ra P^{-TM}/\hat P^{-TM} \simar D(e_0^* TM)^{-TM} / S(e_0^* TM)^{-TM} \simar \Sigma^{\infty,\epsilon}_+ LE. \]
We define the last map at spectrum level $n$ by regarding each point in $\R^n$ close to $E$ as a vector $v$ in the normal bundle, of length less than $\epsilon$, using the exponential map. We add the given tangent vector $t$ to get a vector in $\R^n$, pictured in blue below.
\begin{figure}[H]
\def\svgwidth{.5\linewidth}
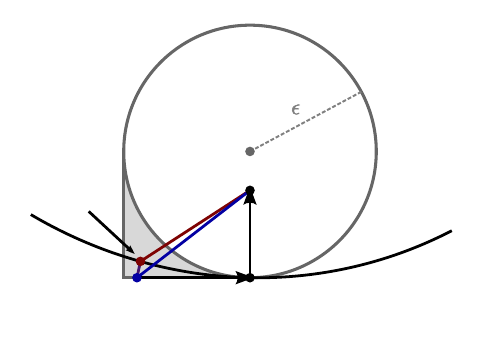
\end{figure}
This may be modified by a homotopy. For each $m$ in the fiber $M$ of $E \ra B$ with a normal vector $v$ and tangent vector $t$, we take the difference $v - \exp_{m}(-t)$, pictured in red. The homotopy between this and $t + v$ is the usual linear interpolation in $\R^n$.

The source of this homotopy consists of points where $\|v\|$ and $\|t\|$ are at most $\epsilon$, and we must check that when $v$ or $t$ has length exactly $\epsilon$, it hits the basepoint. For this purpose, we consider it as a homotopy of maps landing in $\Sigma^{\infty,\epsilon/2}_+ LE$ instead of $\Sigma^{\infty,\epsilon}_+ LE$. Now when $\|v\| = \epsilon$, the homotopy is through vectors that have length at least $\epsilon$, since an open ball of radius $\epsilon$ about $v$ cannot intersect $M$. Therefore these points are sent to the basepoint. On the other hand, since $M$ has a $\epsilon$-tubular neighborhood, each geodesic $\gamma$ has $\|\gamma'\| = 1$ and $\|\gamma''\| \leq \frac1{\epsilon}$ at all times. We focus on the geodesic that moves from $m$ in the direction of $t$, which reaches $t$ at time $\epsilon$. By elementary calculus, the distance between the endpoint of this geodesic and the closed disc representing all possible values of $v$ over $m$ is at least
\[ \left( t - \frac1{2\epsilon}t^2 \right)|_{t = \epsilon} = \frac12 \epsilon \]
Similarly, any smooth path in $\R^n$ with the same initial velocity and the same bound on the second derivative must end at least this far from the disc of all values of $v$. Therefore the straight-line homotopy between our two descriptions of the final map sends the points $(t,v)$ with $\|t\| = \epsilon$ to vectors that are at least $\frac12 \epsilon$ long, and therefore they also land on the basepoint. This concludes the check that our homotopy is well-defined.

Adopting the red-colored description for the map that cancels the $-TM$ and $TM$, it has the formula
\[ (\alpha, v, t) \mapsto ( \exp_{\R^n}(v) - i(\exp_M(-t))) \sma \alpha \]
when $\alpha$ is a free loop in $E$, and $v$ and $t$ are a normal and tangent vector, respectively, at $\alpha(0)$. The map $P^{-TM} \ra \Sigma^{\infty,\epsilon}_+ LE$ then has the formula
\[ (\alpha, v) \mapsto   (v - i(\alpha(0))) \sma \gamma_{\alpha(1),\alpha(0)} \cdot \alpha \]
when $\alpha$ is a path in $E$ with endpoints in the same fiber and $v$ is a point in $\R^n$ with closest point in $E$ given by $\alpha(1)$.
The final map $\mathrm{ \bf eval}$ in our model for the THH transfer is given by exactly the same formula, and the proof is complete.
\end{proof}

Since the Becker-Gottlieb transfer for the smooth bundle $E \ra B$ is well-known to factor into the Pontryagin-Thom umkehr map $B_+ \ra E^{-TM}$ and the inclusion of the zero section $E^{-TM} \ra E$, we can now derive Corollary \ref{cor:intro_becker_gottlieb_smooth_compatibility} from the introduction: for a smooth fiber bundle $f$, the square
\[
\xymatrix{
\Sigma^{\infty}_{+} B \ar[r]^-{\tau(f)} \ar[d]_{c} & \Sigma^{\infty}_{+} E \ar[d]^{c} \\
\Sigma^{\infty}_{+} LB \ar[r]^-{\tau_{\THH}} & \Sigma^{\infty}_{+} LE
}
\]
commutes up to homotopy. 

\begin{remark}
It would be quite striking if this diagram commuted for all perfect fibrations, since the corresponding square for $A$-theory is known not to commute in general.  We also note that the commutativity of this square in general would provide an alternative proof of the recent result of the second author and John Klein that the Becker-Gottlieb transfer is functorial up to homotopy \cite{MaKl}.
\end{remark}

\section{Applications}\label{sec:computations}

We now demonstrate the utility of our geometric model of $\tau_\THH$. Our first application is to covering spaces, where Theorem \ref{thm:geometric_THH_transfer} gives a complete description of the free loop transfer:
\begin{corollary}\label{cor:covering_spaces}
When $f: E \ra B$ is an $n$-sheeted covering space, so that $Lf: LE \ra LB$ is a covering space of at most $n$ sheets on each component, the free loop transfer for $f$ is the ordinary transfer for the covering $Lf$.
\end{corollary}

\begin{proof}
The composite from Theorem \ref{thm:geometric_THH_transfer} becomes
\[ \xymatrix{ \Sigma^\infty_+ LB \ar[r] & \Sigma^\infty_+ E \times_{B \times B} B^I \ar@{<->}[r]^-\cong & \Sigma^\infty_+ E^I \times_{B \times B} B \ar[r] & \Sigma^\infty_+ LE } \]
which is the ordinary transfer for the covering map $E \times_{B \times B} B^I \ra LB$ and the map which deletes the complement of $LE$ in $E^I \times_{B \times B} B$. It is easy to check that the composite of a transfer and such a deletion is a transfer into the smaller subspace $LE$.
\end{proof}

This is a significant extension of the main result of \cite{schlichtkrull}, which dealt with covering maps of the form $BK \arr BG$, when $G$ is a discrete group and $K \leq G$ is a subgroup of finite index. In this special case the covering map $LE \ra LB$ becomes
\[ \xymatrix{ \coprod_{\lambda \in \langle K \rangle} BC_K(\lambda) \ar[r] & \coprod_{\omega \in \langle G \rangle} BC_G(\omega) } \]
where $\lambda$ ranges over conjugacy classes in $K$ and $\omega$ ranges over conjugacy classes in $G$. Schlichtkrull computed the $THH$ transfer in this case as a collection of ordinary transfers, one for each pair $(\lambda,\omega)$ with $\lambda \subseteq \omega$. Our corollary gives this same result by a different method. In the further special case where $G$ is abelian, this covering map simplifies to $K \times BK \ra G \times BG$. The transfer has degree $G/K$ on the components corresponding to $K \subseteq G$ and degree 0 on the other components.

Next we will compute $\tau_\THH$ on cohomology for the bundle $f: BS^1 \arr BS^3$. It turns out that $Lf^*$ is a rational isomorphism in the classes where we need to compute the transfer. So, it suffices to understand the effect of the composite
\[ \xymatrix{ \Sigma^\infty_+ LBS^3 \ar[r]^-{\tau_\THH} & \Sigma^\infty_+ LBS^1 \ar[r]^-{Lf} & \Sigma^\infty_+ LBS^3 } \]
on cohomology. We do this by proving Proposition \ref{prop:intro_euler} from the introduction:

\begin{proposition}\label{prop:euler}
If $f: E \arr B$ is a fibration with finite CW fiber $F$, and $B$ is simply-connected, then the composite map $\tau_\THH^* \circ Lf^*$ on $H^*(LB)$ is multiplication by $\chi(F)$.
\end{proposition}

\begin{proof}
The diagram of cohomology groups
\[ \xymatrix{
H^*(LB) \ar[r]^-{Lf^*} & H^*(LE) \ar[r]^-{\tau_\THH^*} & H^*(LB) \ar[d]^-{c^*} \\
H^*(B) \ar[r]^-{f^*} \ar[u]^-{e_0^*} & H^*(E) \ar[u]^-{e_0^*} \ar[r]^-{\tau^*} & H^*(B) } \]
commutes by Theorem \ref{thm:intro_becker_gottlieb_compatibility}. On $H^0$ this becomes
\[ \xymatrix @R=1.5em{
\Z \ar[r]^-{Lf^*} \ar@{=}[d] & H^0(LE) \ar[r]^-{\tau_\THH^*} & \Z \ar@{=}[d] \\
\Z \ar[r]^-{f^*} & H^0(E) \ar[r]^-{\tau^*} & \Z } \]
since $B$ is simply connected. The bottom row is known to be multiplication by $\chi(F)$, so the top row is, too. The conclusion follows because the desired self-map of $H^*(LB)$ is a $H^*(LB)$-module map (by Proposition \ref{prop:comodule}) which sends the generator to $\chi(F)$ times the generator.
\end{proof}

\begin{remark}
If $F$ is finitely dominated, the composite is multiplication by $\chi(\Sigma^\infty_+ F)$.
\end{remark}

\begin{corollary}\label{prop:calculation_bs1}
Under the ring isomorphisms
\[ H^*(LBS^1) \cong \Lambda[a_1] \otimes \Z[a_2], \qquad H^*(LBS^3) \cong \Lambda[b_3] \otimes \Z[b_4], \qquad |a_i| = |b_i| = i \]
the ordinary map on cohomology $Lf^*$ is the ring map generated by
\begin{align*}
b_3 &\mapsto 2a_1a_2 \\
b_4 &\mapsto a_2^2,
\end{align*}
and the free loop transfer $\tau_\THH^*$ is given on all classes $(n \geq 0)$ by
\begin{align*}
a_2^{2n} &\mapsto 2b_4^n \\
a_1a_2^{2n} &\mapsto 0 \\
a_2^{2n+1} &\mapsto 0 \\
a_1a_2^{2n+1} &\mapsto b_3 b_4^n
\end{align*}
\end{corollary}

\begin{proof}
We recall that the Serre spectral sequence on the path-loop fibration gives
\[ H^*(BS^1) \cong \Z[a_2], \qquad H^*(BS^3) \cong \Z[b_4] \]
and the fibration $S^2 \ra BS^1 \ra BS^3$ tells us that we may choose the generators so that $b_4$ is sent to $a_2^2$. Using the decomposition $LBS^1 \cong S^1 \times BS^1$ and the fibration with a section $S^3 \ra LBS^3 \ra BS^3$, we conclude that in both cases the standard fibration sequence $\Omega X \ra LX \ra X$ gives a Serre spectral sequence with no differentials, giving
\[ H^*(LBS^1) \cong \Lambda[a_1] \otimes \Z[a_2], \qquad H^*(LBS^3) \cong \Lambda[b_3] \otimes \Z[b_4] \]

Next we determine that the map $H^*(LBS^3) \ra H^*(LBS^1)$ sends $b_3$ to $\pm 2a_1a_2$. We recall that the Serre spectral sequence for $\Omega S^2 \ra LS^2 \ra S^2$ has differentials alternating between 0 and 2, starting with 0 at the bottom as the fibration as a section, so that the third cohomology of $LS^2$ is $\Z/2 \oplus \Z$. When this group is considered on the $y$-axis of the cohomology spectral sequence for $LS^2 \ra LBS^1 \ra LBS^3$, the $\Z$ must be killed to accommodate the fact that the limiting cohomology has rank 1, and the $\Z/2$ cannot be killed, so we have an extension of the $\Z$ on the $x$-axis by a $\Z/2$, from which we conclude the map $H^*(LBS^3) \ra H^*(LBS^1)$ has degree 2 in third cohomology. The rest of the calculation follows easily from Proposition \ref{prop:euler}.
\end{proof}

Our third application is a more hands-on computation of $\tau_\THH$ for the Hopf fibration $f: S^3 \arr S^2$. This is a principal $S^1$-bundle, so the composite $\tau_\THH^* \circ Lf^*$ must be zero, and the Becker-Gottlieb transfer $\tau$ vanishes as a map of spectra. Despite these restrictions, we may demonstrate that the free loop transfer $\tau_\THH$ is nonzero.

\begin{proposition}\label{prop:calculation_hopf}
The free loop transfer for the Hopf fibration $S^3 \ra S^2$ is nonzero on integral cohomology.
\end{proposition}

More precisely, $H^q(LS^3) \ra H^q(LS^2)$ is zero unless $q = 3,5,7,\ldots$, where the map is $\Z \ra \Z \oplus \Z/2$ sending the generator of $\Z$ to the generator of $\Z/2$. So the transfer is zero rationally, but nonzero with $\Z/2$ coefficients.

\begin{proof}
The ring $H^*(LS^2)$ has very few nonzero products, so we are forced to take a more direct approach. Recall our convention $P = E \times_B LB$. We use the top line of the Serre spectral sequence for $M \ra P \ra LB$ to compute the umkehr map $LB \ra P^{-TM}$. Then we use a geometric argument to identify $\hat P = P - LE$ and the map $\hat P \ra P$ on homology, which we use to compute the collapse $P \ra P / \hat P \simeq LE^{TM}$.

We recall the cohomology of $LS^3$ as a ring, and $LS^2$ as a group with a few of its multiplications:
\[ H^*(LS^3) \cong \Lambda[v_3] \otimes \Gamma[u_2] \]
\[ H^*(\Omega S^3) \cong \Gamma[u_2] \]
\[ H^*(\Omega S^2) \cong \Lambda[t_1] \otimes \Gamma[u_2'] \]
\[ \xymatrix @R=0.5em @C=0.5em{
q & 0 & 1 & 2 & 3 & 4 & 5 & \ldots \\
H^q(LS^3) & \Z_1 & 0 & \Z_u & \Z_v & \Z_{u^2} & \Z_{uv} & \ldots \\
H^q(LS^2) & \Z_1 & \Z_{a_1} & \Z_{b_2} & \Z_{a_3} \oplus \Z_{a_1b_2}/2\Z & \Z_{b_4} & \Z_{a_5} \oplus \Z_{a_3b_2}/2\Z & \ldots
} \]
Here $\Gamma$ means divided power algebra, so the generators as an abelian group are $1, u_2, u_4, u_6, \ldots$ with the relations $u_{2i}u_{2j} = \frac{(i+j)!}{i!j!}u_{2i+2j}$. In particular, $u_{2n} = \frac1{n!}u_2^n$. The classes $a_{2n+1}$ are the permanent cycles on the $y$-axis in the Serre spectral sequence for $\Omega S^2 \ra LS^2 \ra S^2$, so they map isomorphically to the odd degree classes in $H^*(\Omega S^2)$. The $b_{2n}$ are on the other vertical line and so they map to zero. The classes $2a_{2n+1}b_2$ die on this $E_2$-page but $a_{2n+1}b_2$ survives with $2$-torsion.

Since $S^1 \ra S^3$ is a map of topological groups, we get compatible trivializations
\[ \xymatrix @C=2em @R=2em{
LS^3 \ar[r] \ar[d]^-\cong & P \ar[d]^-\cong & \hat P \ar[l] \ar[d]^-\cong \\
S^3 \times \hofib(* \ra S^3) \ar[d]^-\sim \ar[r] & S^3 \times \hofib(S^1 \ra S^3) \ar[d]^-\sim & S^3 \times \hofib((S^1 \setminus \{*\}) \ra S^3) \ar[l] \ar[d]^-\sim \\
S^3 \times \Omega S^3 \ar[r] & S^3 \times \Omega S^2 & S^3 \times \Omega S^3 \ar[l]
} \]
Since the cofiber of $\hat P \ra P$ is $\Sigma LS^3$, the long exact sequence tells us that $\Omega S^3 \ra \Omega S^2$ is an isomorphism on second cohomology. Therefore we can rewrite our rings as
\[ H^*(\Omega S^3) \cong \Gamma[u_2] \]
\[ H^*(\Omega S^2) \cong \Lambda[t_1] \otimes \Gamma[u_2] \]
and the map sends $t_1$ to 0 and $u_2$ to $u_2$. The map $LS^3 \ra LS^2$ on cohomology then sends all the classes $a_{2n+1}$ to 0, and the class $b_2$ to 0, because it does not hit a multiple of $u_2$. In fact, all the remaining cohomology classes are also sent to 0. One can deduce this from the co-filtration on $\Sigma^\infty_+ LS^n$ coming from the dual of the filtration on the cyclic bar construction on $D_+(S^n)$, and the fact that the map of spaces $S^3 \ra S^2$ induces a map that must respect this filtration.

We can then say
\[ H^*(P) \cong \Lambda[t_1,v_3] \otimes \Gamma[u_2] \]
and the map $LS^3 \ra P$ on cohomology $H^*(P) \ra H^*(LS^3)$ preserves $v_3$ and $u_2$ while killing $t_1$. Returning to the cofiber sequence $\hat P \ra P \ra \Sigma LS^3$, we deduce that $\hat P \ra P$ is surjective on cohomology, with kernel exactly those monomials which contain $t$. It follows that the transfer $H^q(LS^3) \ra H^{q+1}(P)$ is in each degree just multiplication by $\pm t$.

Next we analyze the map $P \ra LS^2$ and its transfer on cohomology, using the Serre spectral sequence for $S^1 \ra P \ra LS^2$. We check that $\pi_1(LS^2)$ acts trivially on $H_*(S^1)$, because the bundle is trivial over $\Omega S^2$. The $E_2$ page is
\[ \xymatrix @R=1em @C=1em{
\Z \ar[rrd]^-\sim & \Z \ar[rrd] & \Z & \Z \oplus \Z/2\Z \ar[rrd] & \Z & \Z \oplus \Z/2\Z & \ldots \\
\Z_1 & \Z_{a_1} & \Z_{b_2} & \Z_{a_3} \oplus \Z_{a_1b_2}/2\Z & \Z_{b_4} & \Z_{a_5} \oplus \Z_{a_3b_2}/2\Z & \ldots
} \]
The unlabeled differentials send the $\Z$ to the $\Z/2$, and the undrawn differentials are zero, so the $E_3 = E_\infty$ page is
\[ \xymatrix @R=1em @C=1em{
0 & 2\Z & \Z & 2\Z \oplus \Z/2\Z & \Z & 2\Z \oplus \Z/2\Z & \ldots \\
\Z_1 & \Z_{a_1} & 0 & \Z_{a_3} & \Z_{b_4} & \Z_{a_5} & \ldots
} \]
Because the cohomology of $P$ is $\Z \oplus \Z$ in degrees 3 and above, each $\Z/2$ on the top line must be a nontrivial extensions of a $\Z_{b_{2n}}$ on the bottom line.
We deduce that in $H^*(LS^2) \ra H^*(P)$, $a_1$ is sent to $t$, and $b_2$ is killed.
Above this range we have
\begin{align*}
a_{2n+1} &\mapsto u_{2n}t_1 + 0u_{2n-2}v_3 \\
b_{2n} &\mapsto 0u_{2n} \pm 2u_{2n-2}t_1v_3 \\
a_{2n+1}b_2 &\mapsto 0
\end{align*}
The terms without $v_3$ are calculated by tracing the classes through the composite $\Omega S^2 \ra P \ra LS^2$, and the terms with $v_3$ are calculated from the composite $LS^3 \ra P \ra LS^2$, together with the above spectral sequence.

Finally we compute the transfer
\[ H^{q+1}(P) \ra H^q(LS^2) \]
This is the map from $H^{q+1}(P)$ into the top line of the above spectral sequence, which kills $1$ and $t$, and sends $u_2$ to $\pm a_1$. Above this range, it is completely determined by the property that it kills exactly the images of the maps into $H^*(P)$ we computed just above. We conclude the map is
\begin{align*}
u_{2n} t_1 &\mapsto 0b_{2n} \\
u_{2n-2} v_3 &\mapsto \pm b_{2n} \\
u_{2n+2} &\mapsto \pm 2a_{2n+1} + ?a_{2n-1}b_2 \\
u_{2n-2} t_1 v_3 &\mapsto 0a_{2n+1} + a_{2n-1}b_2
\end{align*}
Since the transfer $H^q(LS^3) \ra H^{q+1}(P)$ hits just the monomials containing $t$, we conclude that the $THH$ transfer $H^q(LS^3) \ra H^q(LS^2)$ sends each class $u_{2n-2} v_3$ to the nonzero class $a_{2n-1}b_2$ for $n \geq 1$, and all other classes to zero.

\end{proof}

\end{document}